\documentclass[11pt, a4paper, leqno]{amsart}
\usepackage{amsmath,amsthm,amscd,amssymb,amsfonts, amsbsy}
\usepackage{latexsym}
\usepackage{txfonts}
\usepackage{exscale}
\usepackage{mathrsfs}

\usepackage{huncial}
\usepackage[T1]{fontenc}

\usepackage[colorlinks,citecolor=red,pagebackref,hypertexnames=false, breaklinks]{hyperref}
\usepackage[usenames,dvipsnames]{color}

\calclayout
\allowdisplaybreaks

\def\Xint#1{\mathchoice
    {\XXint\displaystyle\textstyle{#1}}%
    {\XXint\textstyle\scriptstyle{#1}}%
    {\XXint\scriptstyle\scriptscriptstyle{#1}}%
    {\XXint\scriptscriptstyle\scriptscriptstyle{#1}}%
    \!\int}
    \def\XXint#1#2#3{{\setbox0=\hbox{$#1{#2#3}{\int}$}
    \vcenter{\hbox{$#2#3$}}\kern-.5\wd0}}
    \def\dashint{\Xint-}

\newcommand{\mm}{\textrm{\hunclfamily m}}

\renewcommand{\chi}{{\bf 1}}

\theoremstyle{plain}
\newtheorem{theorem}[equation]{Theorem}
\newtheorem{lemma}[equation]{Lemma}

\newtheorem{proposition}[equation]{Proposition}

\theoremstyle{definition}
\newtheorem{definition}[equation]{Definition}

\theoremstyle{remark}
\newtheorem{remark}[equation]{Remark}

\numberwithin{equation}{section}

\numberwithin{equation}{section}

\def \R{ \mathbb{R} }
\def \N{ \mathbb{N} }
\def \Z{ \mathbb{Z} }

\def \Scal{ \mathcal{S} }
\def \hh{ \mathrm{H} }
\def \pp{ \mathrm{P} }
\def \Gcal { \mathcal{G} }
\def \Grm{ \mathrm{G} }
\def \Ncal { \mathcal{N} }

\def \mol{ (w,q,p,\varepsilon,M)-\textrm{molecule}}

\def \p{ (w,q,p,\varepsilon,M)-\textrm{representation}}

\def \iint{\int\!\!\!\int}
\def\div{\mathop{\rm div}}
\renewcommand{\Re}{{\rm Re}\,}

\DeclareMathOperator{\supp}{supp}

\begin{document}
\allowdisplaybreaks

\title[Weighted]{Weighted Hardy spaces associated with elliptic operators. 
\\[0.3cm]
{\small Part III: Characterizations of $H_L^{\MakeLowercase{p}}(\MakeLowercase{w})$  and the weighted Hardy space associated with the Riesz transform.}}

\author{Cruz Prisuelos-Arribas}

\address{Cruz Prisuelos-Arribas
Instituto de Ciencias Matem\'aticas CSIC-UAM-UC3M-UCM
\\
Consejo Superior de Investigaciones Cient{\'\i}ficas
\\
C/ Nicol\'as Cabrera, 13-15
\\
E-28049 Madrid, Spain} \email{cruz.prisuelos@icmat.es}

\thanks{The research leading to these results has received funding from the European Research
Council under the European Union's Seventh Framework Programme (FP7/2007-2013)/ ERC
agreement no. 615112 HAPDEGMT.
The author  acknowledges financial support from the Spanish Ministry of Economy and Competitiveness, through the ``Severo Ochoa Programme for Centres of Excellence in R\&D'' (SEV-2015-0554).
}

\date{\today}


\begin{abstract}
We consider  Muckenhoupt weights $w$, and define weighted Hardy spaces $H^p_{\mathcal{T}}(w)$, where $\mathcal{T}$ denotes a conical square function or a non-tangential maximal function defined via the heat or the Poisson semigroup generated by a second order divergence form elliptic operator $L$. In the range $0<p< 1$, we give a molecular characterization of these spaces. Additionally, in the range $p\in \mathcal{W}_w(p_-(L),p_+(L))$ we see that these spaces are isomorphic to the $L^p(w)$ spaces. We also consider the Riesz transform $\nabla L^{-\frac{1}{2}}$, associated with $L$, and show that the Hardy spaces $H^p_{\nabla L^{-1/2},q}(w)$ and $H^p_{\mathcal{S}_{\hh},q}(w)$ are isomorphic, in some range of $p'$s, and $q\in \mathcal{W}_w(q_-(L),q_+(L))$.
\end{abstract}

\maketitle
\vspace*{-1cm}

\tableofcontents

\bigskip
\section{Introduction}

This work ends a series of three papers, started by \cite{MartellPrisuelos} and \cite{MartellPrisuelos:II}, and dedicated to the study of weighted Hardy spaces associated with operators that arise from a second order divergence form elliptic operator  $L$. In particular, we consider 
conical square functions, \eqref{square-H-1}--\eqref{square-P-3}, non-tangential maximal functions, \eqref{non-tangential}, and the Riesz transform, $\nabla L^{-\frac{1}{2}}$.
This generalized Hardy space theory has been started  by P. Auscher, X.T. Duong, and A. McIntosh in an unpublished work, \cite{AuscherDuongMcIntosh}. Besides, P. Auscher and E. Russ in \cite{Auscher-Russ} considered the case on which the heat kernel associated with $L$ is smooth and satisfies pointwise Gaussian bounds, this occurs for instance for real symmetric  operators.  There, among other things, it was shown that the corresponding  Hardy space associated with $L$ agrees with the classical Hardy space.  The question of replacing the Laplacian with another second elliptic operator $L$  was also considered in dimension one by P. Auscher and P. Tchamitchian in \cite{AuscherTchamitchian}.
In the setting of Riemannian manifolds satisfying the doubling volume property,
Hardy spaces associated with the Laplace-Beltrami operator are introduced in
\cite{Auscher-McIntosh-Russ} by P. Auscher, A. McIntosh, and E. Russ and it is
shown that they admit several characterizations.  Simultaneously, in the
Euclidean setting, the study of  Hardy spaces related to the conical square
functions and non-tangential maximal functions associated with the heat and
Poisson  semigroups generated by divergence form elliptic operators, was
taken by S. Hofmann and S. Mayboroda in \cite{HofmannMayboroda}, for $p=1$. The
new point was that only a form of decay weaker than pointwise bounds,
and satisfied in many occurrences, was enough to develop a theory.
Later on S. Hofmann, S. Mayboroda,
and A. McIntosh in \cite{HofmannMayborodaMcIntosh} studied that theory for a general $p$,
and
simultaneously  R. Jiang and D. Yang in \cite{JiangYang} also considered this case.
In the context on weighted Lebesgue measure spaces some progress has been done in \cite{BuiCaoKyYangYang}, \cite{BuiCaoKyYangYang:II}, and \cite{MartellPrisuelos:II}. 
The results obtained in \cite{BuiCaoKyYangYang:II} 
in the particular case $\varphi(x,t):=tw(x)$,  where $w$ is a Muckenhoupt weight, give characterizations of the weighted Hardy spaces which, however, only recover part of the results obtained in the unweighted case by simply taking $w=1$.  
In \cite{MartellPrisuelos:II}, we present a different approach to the theory of weighted Hardy spaces $H^1_L(w)$ associated with a second order divergence form elliptic operator, which naturally generalizes the unweighted setting developed in \cite{HofmannMayboroda}. We define weighted Hardy spaces associated with the conical square functions considered in \eqref{square-H-1}--\eqref{square-P-3} which are written in terms of the heat and Poisson semigroups generated by the elliptic operator. Also, we use  non-tangential maximal functions as defined in \eqref{non-tangential}. We show that the corresponding spaces are all isomorphic and admit  molecular characterizations. This is particularly useful to prove different properties of these spaces as happens in the classical setting and in the context of second order divergence form elliptic operators considered in \cite{HofmannMayboroda}.

Some of the ingredients that were crucial in \cite{MartellPrisuelos:II} and also in the present work are taken from the first part of this series of papers \cite{MartellPrisuelos}, where we already obtained optimal ranges for the weighted norm inequalities satisfied by the heat and Poisson conical square functions associated with the elliptic operator. In \cite{MartellPrisuelos:II} we obtain analogous results for the non-tangential maximal functions  associated with the heat and Poisson semigroups.
All these weighted norm inequalities for the conical square  functions and the non-tangential maximal functions, along with the important fact that our molecules belong naturally to weighted Lebesgue spaces, allow us to impose natural conditions that in particular lead to fully recover the results obtained in \cite{HofmannMayboroda} and \cite{HofmannMayborodaMcIntosh} by simply taking the weight identically one. It is relevant to note that in \cite{BuiCaoKyYangYang, BuiCaoKyYangYang:II} their molecules belong to unweighted Lebesgue spaces and also their ranges of boundedness of the conical square functions are smaller. This makes their hypothesis somehow stronger (although sometimes they cannot be compared with ours) and, despite making a very big effort to present a very general theory, the unweighted case does not follow immediately from their work.

In this  paper we continue with the study of weighted Hardy spaces associated with conical square functions and non-tangential maximal functions $H^p_{\mathcal{T}}(w)$, where $p$ is an exponent different from one. In Section \ref{section:DR}, we give a molecular characterization when $0<p< 1$  (the case $p=1$ was consider in \cite{MartellPrisuelos:II}). The proofs of these results are analogous to those done in \cite{MartellPrisuelos:II}. Therefore, we just sketch them highlighting the main changes.

In Section \ref{section:p>1}, we obtain that the Hardy spaces $H^p_{\mathcal{T}}(w)$ are isomorphic to the $L^p(w)$ spaces for $p\in \mathcal{W}_w(p_-(L),p_+(L))$. The result is the following.
\begin{theorem}\label{theorem:hp=lp}
Given $w\in A_{\infty}$,  if $\mathcal{T}$ is any of the square functions in \eqref{square-H-1}-\eqref{square-P-3} or a non-tangential maximal function in \eqref{non-tangential},
then, for all $p\in \mathcal{W}_w(p_-(L),p_+(L))$, the spaces
$H^p_{\mathcal{T}}(w)$ and $L^p(w)$ are isomorphic with equivalent norms.
\end{theorem}

Finally, in section \ref{sec:Riesz}, we consider another operator: the Riesz transform $\nabla L^{-\frac{1}{2}}$,  and study the Hardy spaces associated with it. In particular, we characterize the Hardy space associated with the Riesz transform through the one associated with the square function $\Scal_{\hh}$ (see below for definitions). The result is the following.
\begin{theorem}\label{thm.Rieszcharacterization}
Given $w\in A_{\infty}$ such that $\mathcal{W}_w(q_-(L),q_+(L))\neq \emptyset$, for all $\max\left\{r_w,\frac{nr_w\widehat{p}_-(L)}{nr_w+\widehat{p}_-(L)}\right\}<p<\frac{q_+(L)}{s_w}$ and $q\in \mathcal{W}_w(q_-(L),q_+(L))$, the spaces
$H^{p}_{\Scal_{\hh},q}(w)$ and $H^p_{\nabla L^{-1/2},q}(w)$ are isomorphic with equivalent norms.
\end{theorem}
We observe that in view of Theorem \ref{theorem:hp=lp}, the dependence on $q$
in the above isomorphism 
can be omitted
when $p\in \mathcal{W}_w(q_-(L),q_+(L))\subset \mathcal{W}_w(p_-(L),p_+(L))$.

In order to prove Theorem \ref{thm.Rieszcharacterization} we need to use interpolation between Hardy spaces. We obtain this from an interpolation result between weighted tent spaces $T^p_q(w)$, and showing that our Hardy spaces are retracts of them (see Section \ref{sec:interpol}). 
\section*{Acknowledgements}

I want to thank my advisor Jos\'e Mar\'ia Martell for their useful comments and corrections, and Li Chen for our conversations and help with references. I also thank Pascal Auscher for some valuable comments .

\section{Preliminaries}
First of all we note that along this work, $C$ or $c$ represent general constant independent of the decisive parameters. 
\subsection{ Weights}
We work with Muckenhoupt weights $w$, which are locally integrable positive functions. We say that a weight $w\in A_1$ if, for every ball $B\subset \mathbb{R}^n$, there holds
\begin{align*}
\dashint_Bw(x) \, dx\leq C w(y), \quad \textrm{for a.e. } y\in B,
\end{align*}
or, equivalently, $\mathcal{M}_u w\le C\,w$  a.e. where $\mathcal{M}_u$ denotes the uncentered Hardy-Littlewood
 maximal operator over balls in $\R^n$.
For each $1<r<\infty$, and $r'$ such that $1/r+1/r'=1$, we say that $w\in A_r$ if 
\begin{align*}
\left(\dashint_Bw(x) \, dx\right)\left(\dashint_Bw(x)^{1-r'} \, dx\right)^{r-1}
\leq C, \quad \forall B\subset \mathbb{R}^n.
\end{align*}
The reverse H\"older classes are defined as follows: for each $1<s<\infty$,
$w\in RH_s$ if
\begin{align*}
\left(\dashint_Bw(x)^s \, dx\right)^{\frac{1}{s}}\leq C\dashint_Bw(x) \, dx \quad \forall B\subset \mathbb{R}^n.
\end{align*}
For $s=\infty$, $w\in RH_{\infty}$ provided that there exists a constant $C$ such that for every ball $B\subset \mathbb{R}^n$
\begin{align*}
w(y)\leq C\dashint_Bw(x) \, dx, \quad \textrm{for a.e. }y\in B.
\end{align*}
Note that we have excluded the case $s = 1$ since the class $RH_1$ consists of all the
weights, and that is the way $RH_1$ is understood in what follows.


\medskip

Besides,
if $w\in A_{r}$, $1\leq r<\infty$,
for every ball $B$ and every measurable set $E\subset B$,
if we denote by 
$
w(E)=\int_Ew(x)\,dx,
$
then
\begin{align}\label{pesosineq:Ap}
\frac{w(E)}{w(B)}\ge [w]_{A_{r}}^{-1}\left(\frac{|E|}{|B|}\right)^{r}.
\end{align}
This implies in particular that $w$ is a doubling measure:
\begin{align}\label{doublingcondition}
w(\lambda B)
\le
[w]_{A_r}\,\lambda^{n\,r}w(B),
\qquad \forall\,B,\ \forall\,\lambda>1.
\end{align}
Moreover, if $w\in RH_{s}$, $1< s\leq\infty$,
\begin{align}\label{pesosineq:RHq}
\frac{w(E)}{w(B)}
\leq
[w]_{RH_{s}}\left(\frac{|E|}{|B|}\right)^{\frac{1}{s'}},
\end{align}

We sum up some of the properties of these classes in the following result, see for instance \cite{GCRF85},
\cite{Duo}, or \cite{Grafakos}. 

\begin{proposition}\label{prop:weights}\
\begin{enumerate}
\renewcommand{\theenumi}{\roman{enumi}}
\renewcommand{\labelenumi}{$(\theenumi)$}
\addtolength{\itemsep}{0.2cm}

\item $A_1\subset A_p\subset A_q$, for $1\le p\le q<\infty$.

\item $RH_{\infty}\subset RH_q\subset RH_p$, for $1<p\le q\le \infty$.

\item If $w\in A_p$, $1<p<\infty$, then there exists $1<q<p$ such
that $w\in A_q$.

\item If $w\in RH_s$, $1<s<\infty$, then there exists $s<r<\infty$ such
that $w\in RH_r$.

\item $\displaystyle A_\infty=\bigcup_{1\le p<\infty} A_p=\bigcup_{1<s\le
\infty} RH_s$.

\item If $1<p<\infty$, $w\in A_p$ if and only if $w^{1-p'}\in
A_{p'}$.



\end{enumerate}
\end{proposition}

\medskip

For a weight $w\in A_{\infty}$, define
\begin{align}\label{rw}
r_w:=\inf\{1\leq r<\infty : w\in A_{r}\},
\qquad
s_w:=\inf\{1\leq s<\infty : w\in RH_{s'}\}.
\end{align}
Note that according to our definition $s_w$ is the conjugated exponent of the one defined in \cite[Lemma 4.1]{AuscherMartell:I}.
Given $0\le p_0<q_0\le \infty$, $w\in A_{\infty}$, and according to \cite[Lemma 4.1]{AuscherMartell:I} we have
\begin{align}\label{intervalrs}
\mathcal{W}_w(p_0,q_0):=\left\{p : p_0<p<q_0, w\in A_{\frac{p}{p_0}}\cap RH_{\left(\frac{q_0}{p}\right)'}\right\}
=
\left(p_0r_w,\frac{q_0}{s_w}\right).
\end{align}
This interval could be empty. Throughout this paper, when we take a point in $\mathcal{W}_w(p_0,q_0)$,  we implicitly understand that we are working with a weight $w\in A_{\infty}$ such that $\mathcal{W}_w(p_0,q_0)\neq \emptyset$, and therefore we can take that point. 
If $p_0=0$ and $q_0<\infty$ it is understood that the only condition that stays is $w\in RH_{\left(\frac{q_0}{p}\right)'}$. Analogously, if $0<p_0$ and $q_0=\infty$ the only assumption is $w\in A_{\frac{p}{p_0}}$. Finally $\mathcal{W}_w(0,\infty)=(0,\infty)$.

Besides, by \cite[Lemma 4.4]{AuscherMartell:I}, we have that 
\begin{align}\label{setweightequivalence}
p\in \mathcal{W}_{w}(p_0,q_0)\Leftrightarrow p'\in\mathcal{W}_{w^{1-p'}}(p_0',q_0').
\end{align}
\subsection{Weighted tent spaces}\label{subsec:wts}
The weighted tent spaces that we consider are defined   
as follows: given $w\in A_{\infty}$, for $0<q,p<\infty$,
\begin{align}\label{deftent}
T^p_q(w):=\{f \textrm{ measurable in }\R^{n+1}_+:\mathcal{A}_q(f)\in L^p(w)\},
\end{align}
endowed with the norm
$
\|f\|_{T^p_q(w)}:=\|\mathcal{A}_qf\|_{L^{p}(w)},
$
when $q=2$ we just write $T^p(w)$, (these spaces were also considered in \cite{Caoetal}).
We denote by $\R^{n+1}_+=\{(y,t): y\in \R^n, 0<t<\infty\}$, the upper half space, and, for all $\alpha>0$ and $0<q<\infty$, the operator $\mathcal{A}^{\alpha}_q$,
$$
\mathcal{A}_q^{\alpha}f(x):=\left(\iint_{\Gamma^{\alpha}(x)}|f(y,t)|^q\frac{dy\,dt}{t^{n+1}}\right)^{\frac{1}{q}},
$$
where
$\Gamma^{\alpha}(x):=\{(y,t)\in \R^{n+1}_+:|x-y|<\alpha t\}$ is the cone of aperture $\alpha>0$ with vertex at $x$. We write $\mathcal{A}f$ and $\Gamma(x)$ when $q=2$ and $\alpha=1$.
By \cite[Proposition 3.39]{MartellPrisuelos} (see also \cite{Latter, Caoetal}) we have that the definition of $T^p_q(w)$ does not depend on the aperture of the cone $\Gamma$ use to define the operator $\mathcal{A}$. This is, for all $0<\alpha,\beta<\infty$,
\begin{align}\label{tentcomparison}
\|\mathcal{A}^{\alpha}_qf\|_{L^p(w)}\approx\|\mathcal{A}^{\beta}_qf\|_{L^p(w)},
\end{align}
with constant depending on the weight $\alpha$ and $\beta$.
Besides, in the same way as in the unweighted case, we can see that these spaces are quasi-Banach spaces for $1\leq q,p<\infty$.

Additionally, note that in our definition of weighted tent spaces the operator $\mathcal{A}$ is the same operator as in the unweighted case, i.e. it does not depend on the weight $w$. Consequently, we can not see $\R^n$ with a doubling measure given by a weight $w$ and apply the interpolation results obtained for tent spaces defined in metric measure spaces or spaces of homogeneous type $(X,\mu)$, since in the definition of those spaces the operator $\mathcal{A}$ is modified to depend on the  measure $\mu$. See for instance \cite{Amenta}, \cite[Lemma 4.6, Proposition 4.9]{HofmannLuMitreaMitreaYan}, or \cite{Russ}. That is why, in Section \ref{sec:interpol}, we also give an interpolation result for these weighted tent spaces.

\subsection{Elliptic operators}\label{subsec:elliptic}
Let $A$ be an $n\times n$ matrix of complex and
$L^\infty$-valued coefficients defined on $\R^n$. We assume that
this matrix satisfies the following ellipticity (or \lq\lq
accretivity\rq\rq) condition: there exist
$0<\lambda\le\Lambda<\infty$ such that
$$
\lambda\,|\xi|^2
\le
\Re A(x)\,\xi\cdot\bar{\xi}
\quad\qquad\mbox{and}\qquad\quad
|A(x)\,\xi\cdot \bar{\zeta}|
\le
\Lambda\,|\xi|\,|\zeta|,
$$
for all $\xi,\zeta\in\mathbb{C}^n$ and almost every $x\in \R^n$. We have used the notation
$\xi\cdot\zeta=\xi_1\,\zeta_1+\cdots+\xi_n\,\zeta_n$ and therefore
$\xi\cdot\bar{\zeta}$ is the usual inner product in $\mathbb{C}^n$. Note
that then
$A(x)\,\xi\cdot\bar{\zeta}=\sum_{j,k}a_{j,k}(x)\,\xi_k\,\bar{\zeta_j}$.
Associated with this matrix we define the second order divergence
form elliptic operator
\begin{align}\label{matrix:A}
L f
=
-\div(A\,\nabla f),
\end{align}
which is understood in the standard weak sense as a maximal-accretive operator on $L^2(\R^n,dx)$ with domain $\mathcal{D}(L)$ by means of a
sesquilinear form. The operator $L$ has a square root $L^{\frac{1}{2}}$, defined as the unique maximal-accretive operator such that
$$
L^{\frac{1}{2}}L^{\frac{1}{2}}=L
$$ 
as unbounded operators (see \cite{Auscher} for a deeper discussion in the operator $L^{\frac{1}{2}}$, and, for a explicit construction, the two references recommended there: \cite[Chapter XIV]{CoifmanMeyer} and \cite[p. 281]{Kato}). We use the following formula to compute $L^{\frac{1}{2}}$:
\begin{align}\label{representationsquarerootofL}
L^{\frac{1}{2}}=\frac{2}{\sqrt{\pi}}\int_{0}^{\infty}tLe^{-t^2L}\frac{dt}{t}.
\end{align}

Moreover, the operator $-L$ generates a $C^0$-semigroup $\{e^{-t L}\}_{t>0}$ of contractions on $L^2(\mathbb{R}^n)$ which is called the heat semigroup.  As in \cite{Auscher} and \cite{AuscherMartell:II}, we denote by $(p_-(L),p_+(L))$ the maximal open interval on which this semigroup is uniformly bounded on $L^p(\mathbb{R}^n)$, and by $(q_-(L),q_+(L))$ the maximal open interval on which the gradient of the heat semigroup, $\{\sqrt{t}\nabla_y e^{-tL}\}_{t>0}$, is uniformly bounded on $L^p(\mathbb{R}^n)$:
\begin{align*}
p_-(L) &:= \inf\left\{p\in(1,\infty): \sup_{t>0} \|e^{-t^2L}\|_{L^p(\mathbb{R}^n)\rightarrow L^p(\mathbb{R}^n)}< \infty\right\},
\\[4pt]
p_+(L)& := \sup\left\{p\in(1,\infty) : \sup_{t>0} \|e^{-t^2L}\|_{L^p(\mathbb{R}^n)\rightarrow L^p(\mathbb{R}^n)}< \infty\right\};
\end{align*}
\begin{align*}
q_-(L) &:= \inf\left\{p\in(1,\infty): \sup_{t>0} \|{t}\nabla_y e^{-t^2L}\|_{L^p(\mathbb{R}^n)\rightarrow L^p(\mathbb{R}^n)}< \infty\right\},
\\[4pt]
q_+(L)& := \sup\left\{p\in(1,\infty) : \sup_{t>0} \|{t}\nabla_y e^{-t^2L}\|_{L^p(\mathbb{R}^n)\rightarrow L^p(\mathbb{R}^n)}< \infty\right\}.
\end{align*}
Note that in place of the semigroup $\{e^{-t L}\}_{t>0}$ we are using its rescaling $\{e^{-t^2 L}\}_{t>0}$. We do so since all the ``heat'' square functions, defined below, are written using the latter, and also because in the context of the off-diagonal estimates it will simplify some computations.
Furthermore, for every $N\in\mathbb{N}$ and $0<q<\infty$, let us set
\begin{align}\label{qN*}
q^{N,*}:=
\left\{
\begin{array}{ll}
\dfrac{q\,n}{n-N\,q}, &\quad\mbox{ if}\quad N\,q<n,
\\[10pt]
\infty, &\quad\mbox{ if}\quad N\,q\ge n.
\end{array}
\right.
\end{align}
Corresponding to the case $N=1$, we write $q^{*}$.

Besides, from \cite{Auscher} (see also \cite{AuscherMartell:II}) we know that $p_-(L)=1$ and  $p_+(L)=\infty$ if $n=1,2$; and if $n\ge 3$ then $p_-(L)<\frac{2\,n}{n+2}$ and $p_+(L)>\frac{2\,n}{n-2}$. Moreover, $q_-(L)=p_-(L)$, $q_+(L)^*\le p_+(L)$, and we always have $q_+(L)>2$, with $q_+(L)=\infty$ if $n=1$.

As in \cite{AuscherMartell:III}, given a weight $w\in A_{\infty}$, we also consider the intervals $\widetilde{\mathcal{J}}_w(L)$ and $\widetilde{\mathcal{K}}_w(L)$ which are, respectively, 
(possibly empty) intervals of $p\in [1,\infty)$ such that $\{e^{-t^2L}\}_{t>0}$ is a bounded set in $\mathcal{L}(L^p(w))$ and 
$\{{t}\nabla e^{-t^2L}\}_{t>0}$ is a bounded set in $\mathcal{L}(L^p(w))$,
(where $\mathcal{L}(X)$ denotes the space of linear continuous maps on a Banach space $X$). 
\subsection{Off-diagonal estimates}\label{sec:od}
We briefly recall the notion of off-diagonal estimates. Let $\{T_t\}_{t>0}$ be a family of linear operators
and let $1\le p\leq q\le \infty$. We say that $\{T_t\}_{t>0}$ satisfies $L^p(\R^n)-L^q(\R^n)$ off-diagonal estimates of exponential type, denoted by $\{T_t\}_{t>0}\in \mathcal{F}_\infty(L^p\rightarrow L^q)$, if for all closed sets $E$, $F$, all $f$, and all $t>0$ we have
$$
\|T_{t}(f\,\chi_E)\,\chi_F\|_{L^q(\mathbb{R}^n)}
\leq
Ct^{-n\left(\frac{1}{p}-\frac{1}{q}\right)}e^{-c\frac{d(E,F)^2}{t^2}}\|f\,\chi_E\|_{L^p(\mathbb{R}^n)}.
$$
Analogously, given $\beta>0$, we say that $\{T_t\}_{t>0}$ satisfies $L^p(\R^n)-L^q(\R^n)$ off-diagonal estimates of polynomial type with order $\beta>0$, denoted by $\{T_t\}_{t>0}\in \mathcal{F}_{\beta}(L^p\rightarrow L^q)$ if for all closed sets $E$, $F$, all $f$, and all $t>0$ we have
$$
\|T_{t}(f\,\chi_E)\,\chi_F\|_{L^q(\mathbb{R}^n)}
\leq
Ct^{-n\left(\frac{1}{p}-\frac{1}{q}\right)}\left(1+\frac{d(E,F)^2}{t^2}
    \right)^{-\left(\beta+\frac{n}{2}\left(\frac{1}{p}-\frac{1}{q}\right)\right)}
\|f\,\chi_E\|_{L^p(\mathbb{R}^n)}.
$$

The heat and Poisson semigroups satisfy respectively off-diagonal estimates of exponential and polynomial type. 
The parameters $p_-(L)$, $p_+(L)$, $q_-(L)$, and $q_+(L)$ besides giving the maximal intervals on which either the heat semigroup or its gradient are uniformly bounded, they characterize the maximal open intervals on which off-diagonal estimates of exponential type hold (see \cite{Auscher} and \cite{AuscherMartell:II}). More precisely, for every $m\in \N_0$, there hold
$$
\{(t^2L)^me^{-t^2L}\}_{t>0}\in \mathcal{F}_\infty(L^p-L^q) \quad \textrm{ for all} \quad p_-(L)<p\leq q<p_+(L)$$
and
$$
\{t\nabla_ye^{-t^2L}\}_{t>0}\in \mathcal{F}_\infty(L^p-L^q) \quad \textrm{ for all} \quad q_-(L)<p\leq q<q_+(L).
$$
From these off-diagonal estimates we have, for every $m\in \N_0:=\N\cup\{0\}$,
\begin{align*}
\{(t\sqrt{L}\,)^{2m}e^{-t\sqrt{L}}\}_{t>0}, \
\in \mathcal{F}_{m+\frac{1}{2}}(L^p\rightarrow L^q),
\end{align*}
for all $p_-(L)<p\leq q< p_+(L)$, and
\begin{align*}
 &\{t\nabla_{y}(t^2L)^me^{-t^2L}\}_{t>0}, \ \{t\nabla_{y,t}(t^2L)^me^{-t^2L}\}_{t>0}\in \mathcal{F}_\infty(L^p\rightarrow L^q),
\\[4pt]
&  \{t\nabla_{y}(t\sqrt{L}\,)^{2m}e^{-t\sqrt{L}}\}_{t>0}\in \mathcal{F}_{m+1}(L^p\rightarrow L^q), \,
\{t\nabla_{y,t}(t\sqrt{L}\,)^{2m}e^{-t\sqrt{L}}\}_{t>0}\in \mathcal{F}_{m+\frac{1}{2}}(L^p\rightarrow L^q),
\end{align*}
for all $q_-(L)<p\leq q< q_+(L)$, (see \cite[Section 2]{MartellPrisuelos}).

Besides, if $\mathcal{W}_w(p_-(L),p_+(L))\neq \emptyset$, there exists a maximal interval of $[1,\infty]$ denoted by ${\mathcal{J}}_w(L)=(\widehat{p}_-(L),\widehat{p}_+(L))$; and if $\mathcal{W}_w(q_-(L),q_+(L))\neq \emptyset$, there exists a maximal interval of $[1,\infty]$ denoted by $\mathcal{K}_w(L)=
(\widehat{q}_-(L),\widehat{q}_+(L))$,
where for all $p,q\in {\mathcal{J}}_w(L)$ or $p,q\in {\mathcal{K}}_w(L)$ and for all $m\in \N_0$, we have that
$\{(t^2L)^me^{t^2L}\}_{t>0}$  or 
$\{{t}\nabla(t^2L)^m e^{t^2L}\}_{t>0}$ satisfy respectively off-diagonal estimates on balls and are bounded sets in $\mathcal{L}(L^p(w))$, (see \cite[Definition 3.2 and Propoition 3.4]{AuscherMartell:III}).

Moreover, in \cite{AuscherMartell:III} the authors proved the following: $\mathcal{W}_w(p_-(L),p_+(L))\subset \mathcal{J}_w(L)\subset \widetilde{\mathcal{J}}_w(L)$, $\mathcal{W}_w(q_-(L),q_+(L))\subset \mathcal{K}_w(L)\subset \widetilde{\mathcal{K}}_w(L)$, $\textrm{Int}\mathcal{J}_w(L)= \textrm{Int}\widetilde{\mathcal{J}}_w(L)$,
$\textrm{Int}\mathcal{K}_w(L)= \textrm{Int}\widetilde{\mathcal{K}}_w(L)$, 
$\textrm{inf}\,\mathcal{J}_w(L)= \textrm{inf}\,{\mathcal{K}}_w(L)$, and
$(\textrm{sup}\mathcal{K}_w(L))^*_w= \textrm{sup}{\mathcal{J}}_w(L)$,
where 
\begin{align}\label{q_wstar}
q_w^*:=\begin{cases}\frac{qnr_w}{nr_w-q}, & nr_w>q,
\\
\infty, & \textrm{otherwise}.
\end{cases}
\end{align}
\subsection{Operators}
Using the heat semigroup and the corresponding Poisson semigroup $\{e^{-t\,\sqrt{L}}\}_{t>0}$, one can define different conical square functions which all have an expression of the form
\begin{align}\label{squarealpha}
Q^{\alpha} f(x)
=\left(\iint_{\Gamma^{\alpha}(x)}|T_t f(y)|^2 \frac{dy \, dt}{t^{n+1}}\right)^{\frac{1}{2}},
\qquad
x\in\R^n,
\end{align}
when $\alpha=1$ we just write $Q f(x)$, ($\Gamma^{\alpha}(x)$ was defined on page \pageref{deftent}).
  More precisely, we introduce the following conical square functions written in terms of the heat semigroup  (hence the subscript $\hh$): for every $m\in \mathbb{N}$,
\begin{align} \label{square-H-1}
\mathcal{S}_{m,\hh}f(x) & = \left(\iint_{\Gamma(x)}|(t^2L)^{m} e^{-t^2L}f(y)|^2 \frac{dy \, dt}{t^{n+1}}\right)^{\frac{1}{2}},
\end{align}
and, for every $m\in \mathbb{N}_0$,
\begin{align}
\mathrm{G}_{m,\hh}f(x)& =\left(\iint_{\Gamma(x)}|t\nabla_y(t^2L)^m e^{-t^2L}f(y)|^2 \frac{dy \, dt}{t^{n+1}}\right)^{\frac{1}{2}},
\label{square-H-2}\\[4pt]
\mathcal{G}_{m,\hh}f(x)&
=
\left(\iint_{\Gamma(x)}|t\nabla_{y,t}(t^2L)^m e^{-t^2L}f(y)|^2 \frac{dy \, dt}{t^{n+1}}\right)^{\frac{1}{2}}.
\label{square-H-3}
\end{align}

In the same way, let us consider conical square functions associated with the Poisson semigroup (hence the subscript $\pp$):  given $K\in \mathbb{N}$,
\begin{align}
\mathcal{S}_{K,\pp}f(x)
&=
\left(\iint_{\Gamma(x)}|(t\sqrt{L}\,)^{2K} e^{-t\sqrt{L}}f(y)|^2 \frac{dy \, dt}{t^{n+1}}\right)^{\frac{1}{2}},
\label{square-P-1}
\end{align}
and for every $K\in \mathbb{N}_0$,
\begin{align}
\mathrm{G}_{K,\pp}f(x)
&=\left(\iint_{\Gamma(x)}|t\nabla_y (t\sqrt{L}\,)^{2K} e^{-t\sqrt{L}}f(y)|^2 \frac{dy \, dt}{t^{n+1}}\right)^{\frac{1}{2}},
\label{square-P-2}
\\[4pt]
\mathcal{G}_{K,\pp}f(x)
&=
\left(\iint_{\Gamma(x)}|t\nabla_{y,t}(t\sqrt{L}\,)^{2K} e^{-t\sqrt{L}}f(y)|^2 \frac{dy \, dt}{t^{n+1}}\right)^{\frac{1}{2}}.
\label{square-P-3}
\end{align}
Corresponding to the cases $m=0$ or $K=0$ we simply write $\mathrm{G}_{\hh}f:=\mathrm{G}_{0,\hh}f$,
$\mathcal{G}_{\hh}f:=\mathcal{G}_{0,\hh}f$,  $\mathrm{G}_{\pp}f:=\mathrm{G}_{0,\pp}f$, and
$\mathcal{G}_{\pp}f:=\mathcal{G}_{0,\pp}f$. Besides, we set $\Scal_{\hh}f:=\Scal_{1,\hh}f$ and $\Scal_{\pp}f:=\Scal_{1,\pp}f$.

We also consider the following non-tangential maximal functions.

\begin{align}\label{non-tangential}
\mathcal{N}_{\hh}f(x)=\sup_{(y,t)\in \Gamma(x)}\left(\int_{B(y,t)}
|e^{-t^2L}f(z)|^2 \frac{dz}{t^n}\right)^{\frac{1}{2}} \,\textrm{ and }\,
\mathcal{N}_{\pp}f(x)=\sup_{(y,t)\in \Gamma(x)}\left(\int_{B(y,t)}
|e^{-t\sqrt{L}}f(z)|^2 \frac{dz}{t^n}\right)^{\frac{1}{2}};
\end{align}
and  denote the Riesz transform associated with the operator $L$ by
$
\nabla L^{-\frac{1}{2}}.
$
We have the following representation for it
\begin{align}\label{Rieszrepresentation}
\nabla L^{-\frac{1}{2}}f=\frac{2}{\sqrt{\pi}}\int_0^{\infty}t\nabla_y e^{-t^2L}f\frac{dt}{t}.
\end{align}
From \cite{AuscherMartell:III} we know the following boundedness result for $\nabla L^{-\frac{1}{2}}$.
\begin{theorem}\label{RieszboundednessAM}\cite[Theorem 5.2]{AuscherMartell:III}
Let $w\in A_{\infty}$ be such that $\mathcal{W}_w(q_-(L),q_+(L))\neq \emptyset$. For all $p\in \rm{Int}$\,$\mathcal{K}_w(L)$ and $f\in L^{\infty}_c(\R^n),$
$$
\|\nabla L^{-\frac{1}{2}}f\|_{L^p(w)}\lesssim \|f\|_{L^p(w)}.
$$
Hence $\nabla L^{-\frac{1}{2}}$ has a bounded extension to $L^p(w)$.
\end{theorem}
\subsection{Complex interpolation}\label{subsec:interpol}
Let us defined the interpolation space described in \cite{Calderon} by A.P. Calder\'on.

Let $A, B$ be Banach spaces embedded in a complex topological vector space $V$, and such that $\|\cdot\|_{A}$ and $\|\cdot\|_B$ denote the norm in $A$ and $B$, respectively.
Now, consider the space $A+B=\{x=y+z: y\in A,\,z\in B\}$ endowed with the norm 
$$
\|x\|_{A+B}:=\inf\{\|y\|_A+\|z\|_B: x=y+z,\, y\in A,\,z\in B\}.
$$
Then, the space $A+B$ becomes a Banach space.

Now, consider the linear space of functions $\mathcal{F}:=\mathcal{F}(A,B)$ as the space of all functions $f(\xi)$, $\xi=\theta+i t$, defined in the strip $0\leq \theta\leq 1$ of the $\xi-plane$, with values in $A+B$ continues and bounded with respect to the norm of $A+B$ in 
$0\leq \theta\leq 1$ and analytic in $0< \theta< 1$, and such that 
$f(it)\in A$ is $A$-continuous and tends to zero as $|t|$ tends to infinity and $f(1+it)\in B$ is $B$-continuous and tends to zero as $|t|$ tends to infinity. The norm that we consider in this space is the following
$$
\|f\|_{\mathcal{F}}:=\max\left\{\sup_{t}\|f(it)\|_A,\sup_t\|f(1+it)\|_B\right\},
$$
under this norm $\mathcal{F}$ becomes a Banach space.

Finally, for a given $\theta$, $0\leq \theta\leq 1$, we define the space $[A,B]_{\theta}:=\{x\in A+B:\,x=f(\theta), f\in \mathcal{F}\}$ endowed with the norm
$$
\|x\|_{[A,B]_{\theta}}:=\inf\{\|f\|_{\mathcal{F}}:x=f(\theta)\}.
$$
Then $[A,B]_{\theta}$ is a Banach space continuously embedded in $A+B$.

\subsection{Extrapolaion}

In some proofs we shall use the following extrapolation result that appear in \cite{CruzMartellPerez}.
\begin{theorem}\label{thm:extrapol}
Let $\mathcal{F}$ be a given family of pairs $(f,g)$ of non-negative and not identically zero measurable functions.
\begin{list}{$(\theenumi)$}{\usecounter{enumi}\leftmargin=1cm
\labelwidth=1cm\itemsep=0.2cm\topsep=.2cm
\renewcommand{\theenumi}{\alph{enumi}}}

\item Suppose that for some fixed exponent $q_0$, $1\le q_0<\infty$, and every weight $w\in RH_{q_0'}$,
\begin{equation*}
\int_{\mathbb{R}^{n}}f(x)^{\frac1{q_0}}\,w(x)\,dx
\leq
C_{w}
\int_{\mathbb{R}^{n}}g(x)^{\frac1{q_0}}\,w(x)\,dx,
\qquad\forall\,(f,g)\in\mathcal{F}.
\end{equation*}
Then, for all $1<q<\infty$ and for all $w\in RH_{q'}$,
\begin{equation*}
\int_{\mathbb{R}^{n}}f(x)^{\frac1q}\,w(x)\,dx
\leq
C_{w,q}
\int_{\mathbb{R}^{n}}g(x)^{\frac1q}\,w(x)\,dx,
\qquad\forall\,(f,g)\in\mathcal{F}.
\end{equation*}

\item Suppose that for some fixed exponent $r_0$, $0< r_0<\infty$, and  every weight $w\in A_{\infty}$
\begin{equation*}
\int_{\mathbb{R}^{n}}f(x)^{r_0}\,w(x)\,dx\leq 
C_{w}
\int_{\mathbb{R}^{n}}g(x)^{r_0}\,w(x)\,dx,
\qquad\forall\,(f,g)\in\mathcal{F}.
\end{equation*}
Then, for all  $0<r<\infty$, and for all $w\in A_{\infty}$,
\begin{equation*}
\int_{\mathbb{R}^{n}}f(x)^{r}\,w(x)\,dx
\leq
C_{w,r}
\int_{\mathbb{R}^{n}}g(x)^{r}\,w(x)\,dx,
\qquad\forall\,(f,g)\in\mathcal{F}.
\end{equation*}

\end{list}
\end{theorem}

Part $(a)$ is not written explicity in \cite{CruzMartellPerez} but can be easily obtained using \cite[Theorem 4.9]{AuscherMartell:I} and \cite[Theorem 3.31]{CruzMartellPerez} (see also \cite[Lemma 3.3, part ($b$)]{MartellPrisuelos}); part ($b$) appears in \cite[Corollary 3.15]{CruzMartellPerez} assuming that the left-hand sides in the inequalities are finite. Here we do not take such assumptions, and in particular, we have that the infiniteness of the left-hand side will imply that of the ride-hand one. To prove ($b$) in the present form just proceed as in the proof of   \cite[Lemma 3.3, part ($a$)]{MartellPrisuelos} but using \cite[Corollary 3.15]{CruzMartellPerez} instead of \cite[Theorem 3.9]{CruzMartellPerez}.
\section{Weighted Hardy spaces}
For $w\in A_{\infty}$, $q\in\mathcal{W}_w(p_-(L),p_+(L))$, and $0<p<\infty$.  We define the weighted Hardy spaces associated with operators and the molecular weighted Hardy space.
\subsection{Weighted Hardy spaces associated with operators}
\begin{definition}
Given a sublinear operator $\mathcal{T}$  acting on functions of $L^q(w)$ 
we define the 
weighted Hardy space $H^{p}_{\mathcal{T},q}(w)$  as the completion of the set
\begin{align*}
\mathbb{H}^{p}_{\mathcal{T},q}(w):=\left\{f\in L^{q}(w):\mathcal{T}f\in L^p(w)\right\},
\end{align*}
with respect to the semi-norm 
 $
\|f\|_{\mathbb{H}^{p}_{\mathcal{T},q}(w)}:=\|\mathcal{T}f\|_{L^p(w)}.
$
\end{definition} 

In our results $\mathcal{T}$ will be any of the square functions
 in \eqref{square-H-1}--\eqref{square-P-3}, or a non-tangential maximal function in \eqref{non-tangential}.
\begin{remark}
In \cite{HofmannMayborodaMcIntosh}, where the unweighted case was considered, the Hardy spaces are defined taking the completion of a set of functions in $L^2(\R^n)$. Here we take functions in $L^q(w)$, where $q\in \mathcal{W}_w(p_-(L),p_+(L))$ because we don't know whether $2$ is in $\mathcal{W}_w(p_-(L),p_+(L))$ or not. In any case, we shall show that for $0<p\leq 1$ or $p\in \mathcal{W}_w(p_-(L),p_+(L))$ this choice of $q$ is irrelevant since all the spaces $H^p_{\mathcal{T},q}(w)$ are isomorphic for all $q\in \mathcal{W}_w(p_-(L),p_+(L))$.
\end{remark}
\subsection{Molecular weighted Hardy spaces}
In order to define the molecules and the molecular decompositions we introduce the following notation: given a cube $Q\subset \R^n$ we set
\begin{align}\label{100}
Q_i:=2^{i+1}Q,\quad \textrm{for all}\quad i\geq 1,\quad
C_1(Q):=4Q \quad \textrm{and,}\quad \textrm{for}\quad i\geq 2,\quad C_i(Q):=2^{i+1}Q\backslash 2^{i}Q.
\end{align}
Besides, $\ell(Q)$ denotes the side length of $Q$.

We next define the notion of molecule and molecular representation. These objects are a weighted version of those defined in \cite{HofmannMayborodaMcIntosh} in the unweighted case.

\begin{definition}\label{moleculas}
Apart from the conditions stated at the beginning of the section for $w$, $q$, and $p$, let us take
$\varepsilon >0$, and $ M\in \N$ such that $M>\frac{n}{2}\left(\frac{r_w}{p}-\frac{1}{p_-(L)}\right)$.
\begin{list}{$(\theenumi)$}{\usecounter{enumi}\leftmargin=1cm \labelwidth=1cm\itemsep=0.2cm\topsep=.0cm \renewcommand{\theenumi}{\alph{enumi}}}

\item \textbf{Molecules:} We say that a function $\mm\in L^q(w)$ (belonging to the range of $L^M$ in $L^q(w)$) is a $\mol$ if there exists a cube $Q\subset \mathbb{R}^n$ such that $\mm$ satisfies
\begin{align*}
\|\mm\|_{mol,w}:=\sum_{i\geq 1}2^{i\varepsilon}w(2^{i+1}Q)^{\frac{1}{p}-\frac{1}{q}}\sum_{k=0}^{M}\left\|\left((\ell(Q)^2L)^{-k}\mm\right)\chi_{C_i(Q)}\right\|_{L^q(w)}<1.
\end{align*}
Henceforth, we refer to the previous expression as the molecular $w$-norm of $\mm$. Besides, any cube $Q$ satisfying that expression, is called a cube associated with $\mm$.

Note that if $\mm$ is a $\mol$, for all associated cubes $Q$:
\begin{align}\label{propertiesmol}
\left\|\left((\ell(Q)^2L)^{-k}\mm\right)\chi_{C_i(Q)}\right\|_{L^q(w)}
\leq 2^{-i\varepsilon}w(2^{i+1}Q)^{\frac{1}{q}-\frac{1}{p}}\quad i=1,2,\ldots;\, k=0,1,\ldots,M. 
\end{align}

\item  \textbf{Molecular representations:} For any function $f$ we say that $\sum_{i\in \N}\lambda_i\mm_i $ is a $\p$ of $f$, if the following conditions are satisfied:
\begin{list}{$(\theenumi)$}{\usecounter{enumi}\leftmargin=1cm \labelwidth=1cm\itemsep=0.2cm\topsep=.2cm \renewcommand{\theenumi}{\roman{enumi}}}

\item $\displaystyle\{\lambda_i\}_{i\in \N}\in \ell^p$.
\item For every $i\in\N$, $\mm_i$ is a $\mol$.
\item  $f=\sum_{i\in \N}\lambda_i\mm_i$  in $L^q(w)$.
\end{list}

\end{list}

\end{definition}

We finally define the molecular weighted Hardy spaces.

\begin{definition}
Let $w$, $q$, $p$, $\varepsilon$, and $M$ be as in the previous definition, we
define the molecular weighted Hardy space $H^p_{L,q,\varepsilon,M}(w)$  as the completion of the set
$$
\mathbb{H}_{L,q,\varepsilon,M}^p(w):=\left\{f=\sum_{i=1}^{\infty}\lambda_i\mm_i
: \sum_{i=1}^{\infty}\lambda_i\mm_i \ \textrm{is a} \ \p\,\textrm{of } f\right\},
$$
with respect to the quasi-norm,
$$
\|f\|_{\mathbb{H}_{L,q,\varepsilon,M}^p(w)}:=\inf\left\{\left(\sum_{i=1}^{\infty}|\lambda_i|^p\right)^{\frac{1}{p}}:
\sum_{i=1}^{\infty}\lambda_i\mm_i \ \textrm{is a} \ \p\, \textrm{of }f\right\}.
$$
\end{definition}
\begin{remark}Although we shall just show molecular characterization for weighted Hardy spaces in the range $0<p\leq 1$,
we have given the definition of the molecular weighted Hardy spaces for all $0<p<\infty$. This is because we can always obtain a molecular decomposition of functions $f\in \mathbb{H}^p_{\mathcal{T},q}$. This is easily seen by following the proof and noticing that there is not restriction over $p$. In particular, we have that $\mathbb{H}^p_{\mathcal{T},q}(w)\subset\mathbb{H}^p_{L,q,\varepsilon,M}(w)$, with $\|f\|_{\mathbb{H}^p_{L,q,\varepsilon,M}(w)}\lesssim \|f\|_{\mathbb{H}^p_{\mathcal{T},q}(w)}$ for all $0<p<\infty$.
\end{remark}

\begin{remark}\label{notation:H1w}
We shall show below that, for $0<p\leq 1$, the Hardy space $H^p_{L,q,\varepsilon,M}(w)$  does not depend on the choice of the allowable parameters $q$, $\varepsilon$, and $M$. Hence, at this point, it is convenient for us to make a choice of these parameters and define the weighted Hardy space as the one associated with this choice:
from now on for every $w\in A_{\infty}$, we fix $q_0\in \mathcal{W}_w(p_-(L),p_+(L))$, $\varepsilon_0>0$, and $M_0\in \N$ such that $M_0>\frac{n}{2}\left(\frac{r_w}{p}-\frac{1}{p_-(L)}\right)$, and set $H_L^p(w):=H_{L,q_0,\varepsilon_0,M_0}^p(w)$, for $0<p\leq 1$.
\end{remark}
\subsection{Weighted Hardy space associated with the Riesz transform}
 \begin{definition}\label{def:riesz}
Given $w\in A_{\infty}$, $q\in \mathcal{W}_w(q_-(L),q_+(L))$, and $0<p<\infty$, we define the weighted Hardy space associated with the Riesz transform $H^p_{\nabla L^{-1/2},q}(w)$, as the completion of the set 
$$
\mathbb{H}^p_{\nabla L^{-1/2},q}(w):=\{f\in L^q(w):\|\nabla L^{-\frac{1}{2}}f\|_{L^p(w)}<\infty\},
$$
with respect to the norm 
$\|f\|_{\mathbb{H}^p_{\nabla L^{-1/2}}}:=\|\nabla L^{-\frac{1}{2}}f\|_{L^p(w)}.$
 \end{definition}

\begin{remark}\label{remark:extenssion}
By \cite[Remark 3.5]{AuscherMartell:III}, we know that, for $q\in \mathcal{W}_w(p_-(L),p_+(L))$, $e^{-tL}$ has an infinitesimal generator $L_{q,w}$, on $L^q(w)$. In particular $e^{-tL_{q,w}}$ and $e^{-tL}$ agree in $L^q(w)\cap L^2(\R^n)$. In our definitions of weighted Hardy spaces, we consider functions  $f\in L^q(w)$, $q\in \mathcal{W}_w(p_-(L),p_+(L))$ ($q\in \mathcal{W}_w(q_-(L),q_+(L))\subset \mathcal{W}_w(p_-(L),p_+(L))$ in the case of the Riesz transform), such that some operator that depends on $L$, let us write $\mathfrak{O}_L$, satisfies that $\mathfrak{O}_Lf\in L^p(w)$, $0<p<\infty$.  Abusing notation we  write $\mathfrak{O}_L$, but it should be understood that, when $f\in L^q(w)\setminus L^2(\R^n)$, $\mathfrak{O}_Lf$ is in fact $\mathfrak{O}_{L_{q,w}}f$. On the other hand, in the case that $f\in L^q(w)\cap L^2(\R^n)$ note that $\mathfrak{O}_{L_{q,w}}f=\mathfrak{O}_Lf$.
\end{remark}


%
%
\section{Auxiliary results}
The proofs of our main results are long and with a lot of technical details. Therefore, in order to facilitate a smooth reading of those proofs, we include some of those technicalities in this section.  

We also introduce the following notation. Let $E\subset \R^n$ be a cube or a ball, for any function $f$ and weight $w\in A_{\infty}$, recalling the notation in \eqref{100}, we set
$$
\dashint_{E}f(x)dw:=\frac{1}{w(E)}\int_Ef(x)w(x)dx\quad\textrm{and}\quad
\dashint_{C_j(E)}f(x)dw:=\frac{1}{w(2^{j+1}E)}\int_{C_j(E)}f(x)w(x)dx,\, \forall\,j\geq 2.
$$

\begin{proposition}\label{prop:offdiagonalweightedballs}
Given $w\in A_{\infty}$ such that $\mathcal{W}_w(p_-(L),p_+(L))\neq \emptyset$, for $0<t,s<\infty$, $p,q\in \mathcal{J}_w(L)$, $p\leq q$, $M\in \N$, and a ball $B\subset \R^n$ with radius $r_B$,  we have that $\{\mathcal{T}_{t,s}^M:=(e^{-t^2L}-e^{-(t^2+s^2)L})^M\}_{t>0}$ satisfies the following $L^p(w)-L^q(w)$ off-diagonal estimates on balls: there exist $\theta_1,\theta_2>0$ such that, for all $j\geq 2$, 
\begin{multline}\label{offdiagonalballs}
\left(\dashint_{C_j(B)}|\mathcal{T}_{t,s}^M(f\chi_{B})(y)|^qdw\right)^{\frac{1}{q}}
\\
\lesssim 2^{j\theta_1}
\max\left\{\frac{2^jr_B}{t},\frac{\sqrt{s^2+t^2}}{2^jr_B}\right\}^{\theta_2}
\left(\frac{s^2}{t^2}\right)^M
e^{-c\frac{4^jr_B^2}{t^2+s^2}}
\left(\dashint_{B}|f(y)|^pdw\right)^{\frac{1}{p}};
\end{multline}
and, 
\begin{align}\label{offdiagonalballs-2}
\left(\dashint_{B}|\mathcal{T}_{t,s}^M(f\chi_{B})(y)|^qdw\right)^{\frac{1}{q}}\lesssim 
\max\left\{\frac{r_B}{t},\frac{\sqrt{s^2+t^2}}{r_B}\right\}^{\theta_2}
\left(\frac{s^2}{t^2}\right)^M
\left(\dashint_{B}|f(y)|^pdw\right)^{\frac{1}{p}}.
\end{align}
\end{proposition}
\begin{proof}
We have for $j\geq 2$,
\begin{align*}
&\left(\dashint_{C_j(B)}|\mathcal{T}_{t,s}^M(f\chi_{B})(y)|^qdw\right)^{\frac{1}{q}}
=
w(2^{j+1}B)^{-\frac{1}{q}}
\left\|\left(\left(\int_{0}^{s^2}\partial_r e^{-(r+t^2)L}\,dr\right)^M f\chi_B\,  \right)\chi_{C_j(B)}\right\|_{L^{q}(w)}
\\ \nonumber
&\leq
w(2^{j+1}B)^{-\frac{1}{q}}
\\ \nonumber
&\quad
\int_{0}^{s^2}\!\!\!\cdots \int_{0}^{s^2}\left\|\left(
\left(\left(\sum_{i=1}^Mr_i+M t^2\right) L \right)^M e^{-(\sum_{i=1}^Mr_i+M t^2) L}(f\chi_B)\right)\chi_{C_j(B)}\right\|_{L^{q}(w)} \ \frac{dr_1 \cdots dr_M}{(\sum_{i=1}^Mr_i+M t^2)^M}
\\ \nonumber
&\lesssim 2^{j\theta_1}
\int_{0}^{s^2}\!\!\!\cdots \int_{0}^{s^2}\Upsilon\left(\frac{2^jr_B}{\sqrt{\sum_{i=1}^Mr_i+M t^2}}\right)^{\theta_2}
e^{-c\frac{4^jr_B^2}{\sum_{i=1}^Mr_i+M t^2}}\frac{dr_1 \cdots dr_M}{(\sum_{i=1}^Mr_i+M t^2)^M}w(B)^{-\frac{1}{p}}\|f\chi_B\|_{L^{p}(w)}
\\ \nonumber
&
\lesssim
2^{j\theta_1}
\max\left\{\frac{2^jr_B}{t},\frac{\sqrt{s^2+t^2}}{2^jr_B}\right\}^{\theta_2}
\left(\frac{s^2}{t^2}\right)^M
e^{-c\frac{4^jr_B^2}{t^2+s^2}}\left(\dashint_{B}|f(y)|^pdw\right)^{\frac{1}{p}},
\end{align*}
where $\theta_1, \theta_2>0$, $\Upsilon(u):=\max\{u,u^{-1}\}$, and we have used the fact that  $(tL)^Me^{-tL}$ satisfies $L^p(w)-L^q(w)$  off-diagonal estimates on balls (see \cite{AuscherMartell:II}).
The proof of \eqref{offdiagonalballs-2} follows similarly.
\end{proof}
In the unweighted case we have a similar result to the previous one, see \cite[(5.12)]{MartellPrisuelos:II} (see also \cite[Proof of Lemma 2.2]{HofmannMartell}).
\begin{proposition}\label{prop:lebesgueoff-dBQ}
For $0<t,s<\infty$, $p\in (p_-(L),p_+(L))$, $M\in \N$, and  for  ${E}_1, {E}_2$ closed subsets in $\mathbb{R}^n$, and $f\in L^{p}(\mathbb{R}^n)$ such that $\textrm{supp}(f)\subset {E}_1$,  we have that $\{\mathcal{T}_{t,s}^M\}_{t>0}$ satisfies the following $L^p(\R^n)-L^p(\R^n)$ off-diagonal estimates: 
\begin{align*}
\left\|\mathcal{T}_{t,s}^Mf\right\|_{L^{p}({E}_2)}
\lesssim
\left(\frac{s^2}{t^2}\right)^M
e^{-c\frac{d({E}_1,{E}_2)^2}{t^2+s^2}}\|f\|_{L^{p}({E}_1)}.
\end{align*}
\end{proposition}
We next give a change of angle result similar to  \cite[Proposition 3.30]{MartellPrisuelos}, but with the difference that in the one below we keep some control on the support by imposing another condition over the radius of the ball.
\begin{proposition}\label{prop:q-extra}
Let  $1\le q \leq s<\infty$, $w\in RH_{s'}$, and $0\leq \alpha\leq 1$. Then, for every ball $B$ with radius $r_B$, and $0<t\leq r_B$, there hold, for $j\geq 2$, 
\begin{align}\label{G-alpha-III}
\int_{C_j(B)}\left(\int_{B(x,\alpha t)}|h(y,t)| \, dy \right)^{\frac{1}{q}}w(x)dx\lesssim \alpha^{\frac{n}{s}}
\int_{2^{j+3}B\setminus 2^{j-2} B}\left(\int_{B(x, t)}|h(y,t)| \, dy \right)^{\frac{1}{q}}w(x)dx;
\end{align}
and, 
\begin{align}\label{G-alpha-III-2}
\int_{4B}\left(\int_{B(x,\alpha t)}|h(y,t)| \, dy \right)^{\frac{1}{q}}w(x)dx\lesssim \alpha^{\frac{n}{s}}
\int_{6B}\left(\int_{B(x, t)}|h(y,t)| \, dy \right)^{\frac{1}{q}}w(x)dx.
\end{align}
\end{proposition}

\begin{proof}
Note that for every $0<t\leq r_B$,
$x\in C_j(B)$, and $0<\alpha\leq 1$, we have that $B(x,\alpha t)\subset 2^{j+2}B\setminus 2^{j-1}B$, for all $j\geq 2$, and $B(x,\alpha t)\subset 5B$, for $j=1$. Besides, if $y\in 2^{j+2}B\setminus 2^{j-1}B$ and $0<t\leq r_B$, then $B(y,t)\subset 2^{j+3}B\setminus 2^{j-2}B$, for all $j\geq 2$; on the other hand if $y\in 5B$ and $0<t\leq r_B$, then $B(y,t)\subset 6B$, for $j=1$. Therefore, 
following  the proof of \cite[Proposition 3.30]{MartellPrisuelos} 
but keeping the above conditions on the support of the integral in $x$ we conclude the proof.
\end{proof}
We next establish some results for a general $0<p<\infty$ that were proved in \cite{MartellPrisuelos:II} for $p=1$.
The following lemma related to $\mol$s is analogous to that of \cite[Lemma 4.6]{MartellPrisuelos:II}.
\begin{lemma}\label{lemma:m-h}
Given $p_0<q$, $0<p<	\infty$,  $w\in A_{\frac{q}{p_0}}$, $\varepsilon>0$, and $M\in \N$. Let $\mm$ be a $\mol$ and let $Q$ be a cube associated with $\mm$.  For every $i\geq 1$ and every $0\le k\le M$, $k\in \N_0$, there holds
\begin{list}{$(\theenumi)$}{\usecounter{enumi}\leftmargin=1cm \labelwidth=1cm\itemsep=0.2cm\topsep=.2cm \renewcommand{\theenumi}{\alph{enumi}}}

\item[]  $\left\|\left((\ell(Q)^2L)^{-k}\mm\right)\chi_{C_i}(Q)\right\|_{L^{p_0}(\R^n)}\lesssim 2^{-i\varepsilon}w(2^{i+1}Q)^{-\frac{1}{p}}|2^{i+1}Q|^{\frac{1}{p_0}}$.
\end{list}
\end{lemma}
\begin{proof}
First of all, recall that
 if $\mm$ is a $\mol$, in particular we have \eqref{propertiesmol}.
That, H\"older's inequality, and the fact that $w\in A_{\frac{q}{p_0}}$ imply
\begin{multline*}
\left\|\left((\ell(Q)^2L)^{-k}\mm\right)\chi_{C_i(Q)}\right\|_{L^{p_0}(\mathbb{R}^n)}
\\
\leq \left(\int_{C_i(Q)}|(\ell(Q)^2L)^{-k}\mm(y)|^{q}w(y) \, dy\right)^{\frac{1}{q}}
\left(\dashint_{2^{i+1}Q}w(y)^{1-\left(\frac{q}{p_0}\right)'} \, dy\right)^{\frac{1}{q}\left(\frac{q}{p_0}-1\right)}|2^{i+1}Q|^{
\frac{1}{p_0}-\frac{1}{q}}
\\
 \lesssim 2^{-i\varepsilon}
w(2^{i+1}Q)^{-\frac{1}{p}}|2^{i+1}Q|^{\frac{1}{p_0}}.
\end{multline*}
\end{proof}
The following propositions are generalizations of
 \cite[Proposition 5.2, Proposition 5.3]{MartellPrisuelos:II}.  They give us uniform  boundedness of the norm on $L^p(w)$, for $0<p\leq 1$, of the square functions applied to $\mol$s. 
\begin{proposition}\label{prop:acotacion-T}
Let $w\in A_{\infty}$ and let $\{\mathcal{T}_t\}_{t>0}$ be a family of sublinear operators satisfying the following conditions:

\begin{list}{$(\theenumi)$}{\usecounter{enumi}\leftmargin=1cm \labelwidth=1cm\itemsep=0.2cm\topsep=.2cm \renewcommand{\theenumi}{\alph{enumi}}}

\item $\{\mathcal{T}_t\}_{t>0}\in \mathcal{F}_{\infty}(L^{p_0}\rightarrow L^2)$ for all $p_-(L)<p_0\leq 2$.

\item $\widehat{S}f(x):=\left(\iint_{\Gamma(x)}|\mathcal{T}_tf(y)|^2\frac{dy \ dt}{t^{n+1}}\right)^{\frac{1}{2}}$ is bounded on $L^q(w)$ for every $q\in \mathcal{W}_w(p_-(L),p_+(L))$.

\item There exists $C>0$ so that for every $t>0$ there holds $\mathcal{T}_{t}=C\,\mathcal{T}_{\frac{t}{\sqrt{2}}}\circ e^{-\frac{t^2}{2}L}$.

\item For every $\lambda>0$, there exists $C_\lambda>0$ such that for every $t>0$ it follows that
$$
\mathcal{T}_{\sqrt{1+\lambda}\,t}=C_\lambda\,\mathcal{T}_{t}\circ e^{-\lambda t^2L}.
$$
\end{list}
Then, for every $\mm$ a $\mol$ with $q\in \mathcal{W}_w(p_-(L),p_+(L))$, $0<p\leq 1$, $\varepsilon>0$, and $M>\frac{n}{2}\left(\frac{r_w}{p}-\frac{1}{p_-(L)}\right)$, it follows that $
\|\widehat{S}\mm\|_{L^p(w)}\lesssim 1,
$
with constants independent of $\mm$.
\end{proposition}
The proof of this proposition follows line by line that of \cite[Proposition 5.2]{MartellPrisuelos:II} (i.e for the case $p=1$).
It suffices to replace $p$ with $q$, the $L^1(w)$ norm with the $L^p(w)$ norm, \cite[Lemma 4.6]{MartellPrisuelos} with Lemma \ref{lemma:m-h}, and \cite[(3.3)]{MartellPrisuelos:II} with \eqref{propertiesmol}; apart from other minor and easy changes as consequences of the above replacements.
From this result we have:
\begin{proposition}\label{prop:contro-mol-SF}
Let $S$ be any of the square functions considered in \eqref{square-H-1}--\eqref{square-P-3}. For every $w\in A_{\infty}$ and $\mm$ a $\mol$ with $q\in \mathcal{W}_w(p_-(L),p_+(L))$, $0<p\leq 1$, $\varepsilon>0$, and  $M>\frac{n}{2}\left(\frac{r_w}{p}-\frac{1}{p_-(L)}\right)$, there hold
\begin{list}{$(\theenumi)$}{\usecounter{enumi}\leftmargin=1cm \labelwidth=1cm\itemsep=0.2cm\topsep=.2cm \renewcommand{\theenumi}{\alph{enumi}}}
\item$
\|S \mm\|_{L^p(w)}\leq C.
$

\item For all $f\in \mathbb{H}_{L,q,\varepsilon,M}^1(w)$, 
$
\|S f\|_{L^p(w)}\lesssim \|f\|_{\mathbb{H}_{L,q,\varepsilon,M}^1(w)}.
$
\end{list}
\end{proposition}
\begin{proof}
Note that, in view of \cite[Theorems 1.14 and 1.15, and Remark 4.22]{MartellPrisuelos} and the fact that $\Scal_{\hh}f\leq \frac{1}{2}\Gcal_{\hh}f$, to prove part $(a)$ it suffices to show the desired estimate for $\Gcal_{\hh}$. To this end, we observe that
$
|t\nabla_{y,t}e^{-t^2L}f|^2= |t\nabla_{y}e^{-t^2L}f|^2+4|t^2Le^{-t^2L}f|^2.
$
Besides, both
$
\mathcal{T}_t:=t\nabla_{y}e^{-t^2L}$
and 
$
\mathcal{T}_t:=t^2Le^{-t^2L}
$
satisfy the hypotheses of Proposition \ref{prop:acotacion-T}:
 $(a)$ follows from the off-diagonal estimates that the families $
\mathcal{T}_t:=t\nabla_{y}e^{-t^2L}$
and 
$
\mathcal{T}_t:=t^2Le^{-t^2L}
$ satisfy (see Section \ref{sec:od});  $(b)$ is contained in \cite[ Theorem 1.12, part $(a)$]{MartellPrisuelos}; and finally $(c)$ and $(d)$ follow from easy calculations. Thus, we can apply  Proposition \ref{prop:acotacion-T} and obtain the desired estimate.

As for part $(b)$, fix $w\in A_{\infty}$ and take $q\in \mathcal{W}_w(p_-(L),p_+(L))$, $0<p\leq 1$, $\varepsilon>0$, and $M\in \N$ such that $M>\frac{n}{2}\left(\frac{r_w}{p}-\frac{1}{p_-(L)}\right)$. Then,
for $f\in \mathbb{H}_{L,q,\varepsilon,M}^p(w)$,
 there exists a $\p$ of $f$, $f=\sum_{i=1}^{\infty}{\lambda}_i {\mm}_i$, such that
 $$
 \left(\sum_{i=1}^{\infty}|{\lambda}_i|^p\right)^{\frac{1}{p}}
 \leq 2\|f\|_{\mathbb{H}_{L,q,\varepsilon,M}^p(w)}.
 $$
On the other hand, since $\sum_{i=1}^{\infty}{\lambda}_i {\mm}_i$ converges in $L^q(w)$ and since for any choice of $S$, we have that $S$ is a sublinear operator bounded on $L^q(w)$ (see \cite[Theorems 1.12 and 1.13]{MartellPrisuelos}). This, part ($a$), and the fact that $0<p\leq 1$ imply
\begin{align*}
\|Sf\|_{L^p(w)}
=\left\|S\left(\sum_{i=1}^{\infty}{\lambda}_i {\mm}_i\right)\right\|_{L^p(w)}
\leq 
\left(\sum_{i=1}^{\infty}|{\lambda}_i|^p\, \|S {\mm}_i\|_{L^p(w)}^p\right)^{\frac{1}{p}}\leq C \left(\sum_{i=1}^{\infty}|{\lambda}_i|^p\right)^{\frac{1}{p}}
\lesssim  \|f\|_{\mathbb{H}_{L,q,\varepsilon,M}^p(w)},
\end{align*}
as desired.
\end{proof}
As for the non-tangential maximal functions considered in \eqref{non-tangential}, we generalize \cite[Proposition 7.22]{MartellPrisuelos:II}.
\begin{proposition}\label{prop:control-mol-nontangential}
Let $w \in A_{\infty}$, $0<p\leq 1$, $q \in  \mathcal{W}_w(p_-(L), p_+(L))$, $\varepsilon > 0$, and $M \in \N$ such that $M >\frac{n}{2}\left(\frac{r_w}{p}-\frac{1}{p_-(L)}\right)
$, and let $\mm$
be a $(w,q, p, \varepsilon, M)$-molecule. Then,
\begin{list}{$(\theenumi)$}{\usecounter{enumi}\leftmargin=1cm \labelwidth=1cm\itemsep=0.2cm\topsep=.2cm \renewcommand{\theenumi}{\alph{enumi}}}

\item[(a)]$\|\mathcal{N}_{\hh}\mm\|_{L^p(w)}+\|\mathcal{N}_{\pp}\mm\|_{L^p(w)}\leq C$

\item[(b)] For all $f\in \mathbb{H}^{p}_{L,q,\varepsilon,M}(w)$,
$$
\|\mathcal{N}_{\hh}f\|_{L^p(w)}+\|\mathcal{N}_{\pp}f\|_{L^p(w)}\leq \|f\|_{\mathbb{H}^{p}_{L,q,\varepsilon,M}(w)}.
$$

\end{list}
\end{proposition}
The proof of part $(a)$ follows as \cite[Proposition 7.22, part (a)]{MartellPrisuelos:II}, computing the norms $\|\mathcal{N}_{\hh}\mm\|_{L^p(w)}^p$ and $\|\mathcal{N}_{\pp}\mm\|_{L^p(w)}^p$ instead of the $L^1(w)$ norms, replacing $p$ with $q$, and using
Lemma \ref{lemma:m-h} instead of \cite[Lemma 4.6]{MartellPrisuelos:II}. The only mayor change comes when, while computing $\|\mathcal{N}_{\pp}\mm\|_{L^p(w)}^p$, this integral appears
$$
\int_{\R^n}\left(\int_{\frac{1}{4}}^{\infty}e^{-u}\,\Scal_{\hh}^{4\sqrt{u}}\mm(x)du\right)^pw(x)dx,
$$
(see \eqref{squarealpha} and \eqref{square-H-1} for the deinition of $\Scal_{\hh}^{4\sqrt{u}}$).
In order to estimate it, note that for all  $w_0\in A_{\infty}$,  $r_0>r_w$, applying Minkowski's integral inequality and \cite[Proposition 3.2]{MartellPrisuelos}, we have
\begin{align*}
\int_{\R^n}\left(\int_{\frac{1}{4}}^{\infty}e^{-u}\,\Scal_{\hh}^{4\sqrt{u}}\mm(x)du\right)^2w_0(x)dx&
\leq
\left(\int_{\frac{1}{4}}^{\infty}e^{-u}\left(\int_{\R^n}\,(\Scal_{\hh}^{4\sqrt{u}}\mm(x))^2w_0(x)dx\right)^{\frac{1}{2}}du\right)^2
\\&
\leq
\left(\int_{\frac{1}{4}}^{\infty}e^{-u}(4\sqrt{u})^{\frac{nr_0}{2}}\left(\int_{\R^n}\,(\Scal_{\hh}\mm(x))^2w_0(x)dx\right)^{\frac{1}{2}}du\right)^2
\\&
\lesssim
\int_{\R^n}\,(\Scal_{\hh}\mm(x))^2w_0(x)dx.
\end{align*}
Hence, applying  Theorem \ref{thm:extrapol}, part ($b$), we obtain for all $0<\widetilde{r}<\infty$ and $\widetilde{w}\in A_{\infty}$
$$
\int_{\R^n}\left(\int_{\frac{1}{4}}^{\infty}e^{-u}\,\Scal_{\hh}^{4\sqrt{u}}\mm(x)du\right)^{\widetilde{r}}\widetilde{w}(x)dx
\lesssim 
\int_{\R^n}\,(\Scal_{\hh}\mm(x))^{\widetilde{r}}\widetilde{w}(x)dx.
$$
In particular, for every $0<p\leq 1$ and $w\in A_{\frac{q}{p_-(L)}}\cap RH_{\left(\frac{p_+(L)}{q}\right)'}$, and by Proposition \ref{prop:contro-mol-SF}, part $(a)$, we have that
$$
\int_{\R^n}\left(\int_{\frac{1}{4}}^{\infty}e^{-u}\,\Scal_{\hh}^{4\sqrt{u}}\mm(x)du\right)^{p}{w}(x)dx
\lesssim 
\int_{\R^n}\,(\Scal_{\hh}\mm(x))^{p}{w}(x)dx\leq C.
$$

The proof of part $(b)$ follows as the proof of Proposition \ref{prop:contro-mol-SF}, part $(b)$, but using \cite[Proposition 7.1]{MartellPrisuelos:II} instead of \cite[Theorems 1.12 and 1.13]{MartellPrisuelos}.
Finally, let us see that the Riesz transform applied to $\mol$s is also uniformly bounded.
\begin{proposition}\label{prop:contro-mol-Riesz}
For every $w\in A_{\infty}$, $q\in \mathcal{W}_w(q_-(L),q_+(L))$, $0<p\leq 1$, $\varepsilon>0$,  $M\in \N$ such that $M>\frac{n}{2}\left(\frac{r_w}{p}-\frac{1}{p_-(L)}\right)$,  and $\mm$ a $\mol$,  there hold
\begin{list}{$(\theenumi)$}{\usecounter{enumi}\leftmargin=1cm \labelwidth=1cm\itemsep=0.2cm\topsep=.2cm \renewcommand{\theenumi}{\alph{enumi}}}
\item$
\|\nabla L^{-\frac{1}{2}} \mm\|_{L^p(w)}\leq C.
$

\item For all $f\in \mathbb{H}_{L,q,\varepsilon,M}^p(w)$, 
$
\|\nabla L^{-\frac{1}{2}} f\|_{L^p(w)}\lesssim \|f\|_{\mathbb{H}_{L,q,\varepsilon,M}^p(w)}.
$
\end{list}
\end{proposition}
\begin{proof}
Assuming part ($a$) the proof of part ($b$) follows as the proof of Proposition \ref{prop:contro-mol-SF}, part $(b)$, but using Theorem \ref{RieszboundednessAM}  instead of \cite[Theorems 1.12 and 1.13]{MartellPrisuelos}.

We next prove part ($a$).
Fix $w$, $p$, $q$, $\varepsilon$, and $M$ as in the statement of the proposition. Note that since $w\in A_{\frac{q}{q_-(L)}}\cap RH_{\left(\frac{q_+(L)}{q}\right)'}$ and $M>\frac{n}{2}\left(\frac{r_w}{p}-\frac{1}{q_-(L)}\right)$ (recall that $p_-(L)=q_-(L)$), we can take $r_0>r_w$, $p_0$, and $q_0$, $q_-(L)<p_0<q<q_0<q_+(L)$, close enough to $r_w$, $q_-(L)$, and $q_+(L)$, respectively, so that $w\in A_{\frac{q}{p_0}}\cap RH_{\left(\frac{q_0}{q}\right)'}$ and 
\begin{align}\label{M-choice-riesz}
M>\frac{n}{2}\left(\frac{r_0}{p}-\frac{1}{p_0}\right).
\end{align}
Besides, take $\mm$  a $(w,q,p,\varepsilon,M)-$molecule and $Q\subset \R^n$ one of its associated cubes, with sidelength $\ell(Q)$, and consider
$
B_Q:=\left(I-e^{-\ell(Q)^2L}\right)^M$ and $
A_Q:=I-B_Q.
$
Additionally, recalling the notation given in \eqref{100} and considering $\widetilde{A}_{Q}^k:=\left(k\ell(Q)^2L\right)^Me^{-k\ell(Q)^2L}$, for every $k\in \{1,\ldots,M\}$, we can write
\begin{multline}\label{splitingmolecule}
\nabla L^{-\frac{1}{2}}\mm=\nabla L^{-\frac{1}{2}}B_Q\mm+\nabla L^{-\frac{1}{2}}A_Q\mm=
\nabla L^{-\frac{1}{2}}B_Q\mm+\sum_{k=1}^MC_{k,M}\nabla L^{-\frac{1}{2}}\widetilde{A}_Q^k\widetilde{\mm}
\\
=\sum_{i\geq 1}\left(\nabla L^{-\frac{1}{2}}B_Q\mm_i+\sum_{k=1}^MC_{k,M}\nabla L^{-\frac{1}{2}}\widetilde{A}_Q^k\widetilde{\mm}_i\right),
\end{multline}
where $\widetilde{\mm}:=(\ell(Q)^2L)^{-M}\mm$ and for any function $f$, we denote $f_i:=f\chi_{C_i(Q)}$, for all $i\geq 1$.
Thus, we have
\begin{align}\label{splittingnorm}
\|\nabla L^{-\frac{1}{2}}B_Q\mm_i\|_{L^p(w)}^p\lesssim \sum_{j\geq 1}\|\chi_{C_j(Q_i)}(\nabla L^{-\frac{1}{2}}B_Q\mm_i)\|_{L^p(w)}^p=:\sum_{j\geq 1}I_{ij}.
\end{align}
Then, for $j=1$, applying H\"older's inequality,
the boundedness of $\nabla L^{-\frac{1}{2}}$ and $B_Q$  on $L^q(w)$  (see Theorem \ref{RieszboundednessAM} and \cite{AuscherMartell:II}), and by \eqref{propertiesmol} and \eqref{doublingcondition}, we obtain
\begin{multline}\label{termI-i1-riesz}
I_{i1}
\lesssim
w(2^{i+1}Q)^{1-\frac{p}{q}}
\left(\int_{4Q_i}\left|\nabla L^{-\frac{1}{2}}B_Q \mm_i(x)\right|^qw(x)dx\right)^{\frac{p}{q}}
\lesssim 
\|\mm_i\|_{L^q(w)}^pw(2^{i+1}Q)^{1-\frac{p}{q}}\leq 2^{-ip\varepsilon}.
\end{multline}
As for $j\geq 2$, denoting $\mathcal{T}_t:=t\nabla_ye^{-t^2L}$, using \eqref{Rieszrepresentation} 
and splitting the integral in $t$, we obtain
\begin{align}\label{estimateI_ij}
 I_{ij}
\leq 
\int_{C_j(Q_i)}\left|\int_{0}^{\ell(Q)}\mathcal{T}_tB_Q \mm_i(x)\frac{dt}{t}\right|^pw(x)dx
+ 
\int_{C_j(Q_i)}\left|\int_{\ell(Q)}^{\infty}\mathcal{T}_tB_Q \mm_i(x)\frac{dt}{t}\right|^pw(x)dx=:I_{ij}^1+I_{ij}^2.
\end{align}
The estimate of $I_{ij}^1$, follows applying
twice H\"older's inequality, the fact that $w\in RH_{\left(\frac{q_0}{q}\right)'}$, and Mikowski's integral inequality. Besides, we expand the binomial, apply respectively the $L^{p_0}(\R^n)-L^{q_0}(\R^n)$ and the  $L^{p_0}(\R^n)-L^{p_0}(\R^n)$
off-diagonal estimates that the families $\{t\nabla_y  e^{-t^2L}\}_{t>0}$ and $\{e^{-t^2L}\}_{t>0}$ satisfied, (see Section \ref{sec:od}), and also \cite[Lemma 2.1]{MartellPrisuelos} (see also \cite[Lemma 2.3]{HofmannMartell}), Lemma \ref{lemma:m-h}, and \eqref{doublingcondition}. Then,
\begin{align}\label{term_Iij-controlmol}
&I_{ij}^1
\lesssim
w(2^{j+1}Q_i)^{1-\frac{p}{q}}
\left(\int_{C_j(Q_i)}\Bigg|\int_{0}^{\ell(Q)}\mathcal{T}_tB_Q \mm_i(x)\frac{dt}{t}\Bigg|^qw(x)dx\right)^{\frac{p}{q}}
\\\nonumber
&
\lesssim
 w(2^{j+1}Q_i)\,
|2^{j+1}Q_i|^{-\frac{p}{q_0}}
\left(\int_0^{\ell(Q)}\left(\int_{C_j(Q_i)}\left|\mathcal{T}_tB_Q \mm_i(x)\right|^{q_0}dx\right)^{\frac{1}{q_0}}\frac{dt}{t}\right)^{p}.
\\\nonumber
&
\lesssim
w(2^{j+1}Q_i)\,
|2^{j+1}Q_i|^{-\frac{p}{q_0}}
\left(\int_0^{\ell(Q)}\left(\int_{C_j(Q_i)}\left|t\nabla_y 
e^{-t^2L} \mm_i(x)\right|^{q_0}dx\right)^{\frac{1}{q_0}}\frac{dt}{t}\right.
\\\nonumber
&\qquad
+\sum_{k=1}^{M}C_{k,M}
\left.\int_0^{\ell(Q)}\ell(Q)^{-1}\left(\int_{C_j(Q_i)}\left|\sqrt{k}\ell(Q)\nabla_y 
e^{-k\ell(Q)^2L}e^{-t^2L} \mm_i(x)\right|^{q_0}dx\right)^{\frac{1}{q_0}}\,dt\right)^p
\\\nonumber
&
\lesssim 
\|\mm_i\|_{L^{p_0}(\R^n)}^pw(2^{j+1}Q_i) 
|2^{j+1}Q_i|^{-\frac{p}{q_0}}
\left(\int_0^{\ell(Q)}\left(e^{-c\frac{4^{i+j}\ell(Q)^2}{t^2}}
t^{-\left(\frac{n}{p_0}-\frac{n}{q_0}\right)}
+
\frac{te^{-c4^{i+j}}}{\ell(Q)}\ell(Q)^{-\left(\frac{n}{p_0}-\frac{n}{q_0}\right)}
\right)\,\frac{dt}{t}\right)^p 
\\\nonumber
&\,\lesssim 
e^{-c4^{i+j}}.
\end{align}
As for $I_{ij}^2$, we proceed as before  but also changing the variable $t$ into $\sqrt{M+1}t=:C_Mt$ and considering $B_{Q,t}:=\left(e^{-t^2L}-e^{-(t^2+\ell(Q)^2)L}\right)^M$. Next, we apply \cite[Lemma 2.1]{MartellPrisuelos} using that the families $\{\mathcal{T}_t\}_{t>0}$  and $\{B_{Q,t}\}_{t>0}$ satisfy respectively
$L^{p_0}(\R^n)-L^{q_0}(\R^n)$ and  $L^{p_0}(\R^n)-L^{p_0}(\R^n)$ off-diagonal estimates (see Section \ref{sec:od} and Proposition \ref{prop:lebesgueoff-dBQ}). Besides, note that $\ell(Q)\leq tC_M$. Then, by Lemma \ref{lemma:m-h}  and  changing the  variable $t$ into $\frac{2^{j+i}\ell(Q)}{t}$, we get
\begin{align*}
I_{ij}^2
&
\lesssim 
w(2^{j+1}Q_i)
|2^{j+1}Q_i|^{-\frac{p}{q_0}}
\left(\int_{\frac{\ell(Q)}{C_M}}^{\infty}\left(\int_{C_j(Q_i)}\left|
\mathcal{T}_tB_{Q,t} \mm_i(x)\right|^{q_0}dx\right)^{\frac{1}{q_0}}\frac{dt}{t}\right)^p
\\
&
\lesssim
w(2^{j+1}Q_i)
|2^{j+1}Q_i|^{-\frac{p}{q_0}}\|\mm_i\|_{L^{p_0}(\R^n)}^p
\left(\int_{\frac{\ell(Q)}{C_M}}^{\infty}\left(\frac{\ell(Q)^2}{t^2}\right)^Mt^{-n\left(\frac{1}{p_0}-\frac{1}{q_0}\right)}e^{-c\frac{4^{i+j}\ell(Q)^2}{t^2}}\frac{dt}{t}\right)^p
\\
&
\lesssim
2^{-ip\left(2M+\varepsilon\right)}2^{-jp\left(2M+\frac{n}{p_0}-\frac{r_0n}{p}\right)}
.
\end{align*}
Hence, by this, \eqref{splittingnorm}, \eqref{termI-i1-riesz}, \eqref{estimateI_ij}, \eqref{term_Iij-controlmol}, and \eqref{M-choice-riesz}, we have
\begin{align}\label{termBQ}
\sum_{i\geq 1}\|\nabla L^{-\frac{1}{2}}B_Q\mm_i\|_{L^p(w)}^p\le C.
\end{align}
Now, proceeding as in the estimate of $I_{i1}$,
since the Riesz transform and $\widetilde{A}_Q^k$ are bounded on $L^q(w)$ (see Theorem \ref{RieszboundednessAM} and \cite[Proposition 5.9]{AuscherMartell:II}), and by \eqref{propertiesmol}, we get
\begin{align}\label{estimateII_{i1}}
\|\chi_{4Q_i}\nabla L^{-\frac{1}{2}}\widetilde{A}_Q^k\widetilde{\mm}_i\|_{L^p(w)}\lesssim w(4Q_i)^{\frac{1}{p}-\frac{1}{q}}\|\widetilde{\mm}_i\|_{L^q(w)}
\lesssim 2^{-i\varepsilon}.
\end{align}
Next, for $j\geq 2$, we use \eqref{Rieszrepresentation}, and proceed as in the estimate of $I_{ij}$,
\begin{align*}
\|\chi_{C_j(Q_i)}\nabla L^{-\frac{1}{2}}\widetilde{A}_Q^k\widetilde{\mm}_i\|_{L^p(w)}^p
&
\lesssim w(2^{j+1}Q_i)|2^{j+1}Q_i|^{-\frac{p}{q_0}}
\left(\int_0^{\ell(Q)}\left(\int_{C_j(Q_i)}\left|
\mathcal{T}_t\widetilde{A}_Q^k \widetilde{\mm}_i(x)\right|^{q_0}dx\right)^{\frac{1}{q_0}}\frac{dt}{t}\right.
\\&\hspace*{4cm}
+
\left.\int_{\ell(Q)}^{\infty}\left(\int_{C_j(Q_i)}\left|
\mathcal{T}_t\widetilde{A}_Q^k \widetilde{\mm}_i(x)\right|^{q_0}dx\right)^{\frac{1}{q_0}}\frac{dt}{t}\right)^p
\\
&
=:
w(2^{j+1}Q_i)|2^{j+1}Q_i|^{-\frac{p}{q_0}}\left(II^1_{ij}+II^2_{ij}\right)^p.
\end{align*}
We first estimate $II^1_{ij}$, proceeding as in the estimate of $I_{ij}^1$ in \eqref{term_Iij-controlmol}. We apply \cite[Lemma 2.1]{MartellPrisuelos} 
with the families $\{t\nabla_y (t^2L)^M e^{-t^2L}\}_{t>0}$ and $\{e^{-t^2L}\}_{t>0}$ that satisfy respectively $L^{p_0}(\R^n)-L^{q_0}(\R^n)$ and $L^{p_0}(\R^n)-L^{p_0}(\R^n)$  off-diagonal estimates (see Section \ref{sec:od}). Then, by Lemma \ref{lemma:m-h}, we obtain
\begin{multline}\label{local-AQ}
II^1_{ij}
=\frac{c_{k,M}}{\ell(Q)}
\int_0^{\ell(Q)}\left(\int_{C_j(Q_i)}\left|k^{\frac{1}{2}}\ell(Q)
\nabla_y  \widetilde{A}_{Q}^ke^{-t^2L}\widetilde{\mm}_i(x)\right|^{q_0}dx\right)^{\frac{1}{q_0}}dt
\\
\lesssim 
e^{-c4^{j+i}}\ell(Q)^{-n\left(\frac{1}{p_0}-\frac{1}{q_0}\right)}\|\widetilde{\mm}_i\|_{L^{p_0}(\R^n)}
\lesssim w(Q_i)^{-\frac{1}{p}}
\ell(Q)^{\frac{n}{q_0}}e^{-c4^{i+j}}.
\end{multline}
Now consider $s_{Q,t}^k:=k\ell(Q)^2+t^2$. Then 
changing the variable $t$ into $\sqrt{2}t$, and proceeding as in the estimate of $I^2_{ij}$ but
applying this time \cite[Lemma 2.1]{MartellPrisuelos}  with the families  $\{\mathcal{T}_t\}_{t>0}$ and $\{(t^2L)^M e^{-t^2L}\}_{t>0}$ that satisfy respectively $L^{p_0}(\R^n)-L^{q_0}(\R^n)$ and $L^{p_0}(\R^n)-L^{p_0}(\R^n)$ off-diagonal estimates, we have that
\begin{align*}
II^2_{ij}&
\lesssim
\int_{\frac{\ell(Q)}{\sqrt{2}}}^{\infty}\left(\frac{\ell(Q)^2}{t^2}\right)^M
\left(\int_{C_j(Q_i)}\left|
\mathcal{T}_t\left(s_{Q,t}^kL\right)^Me^{-s_{Q,t}^kL} \widetilde{\mm}_i(x)\right|^{q_0}dx\right)^{\frac{1}{q_0}}\frac{dt}{t}
\\
& \lesssim\|\widetilde{\mm}_i\|_{L^{p_0}(\R^n)}
\int_{\frac{\ell(Q)}{\sqrt{2}}}^{\infty}\left(\frac{\ell(Q)^2}{t^2}\right)^Mt^{-n\left(\frac{1}{p_0}-\frac{1}{q_0}\right)}e^{-c\frac{4^{j+i}\ell(Q)^2}{t^2}}\frac{dt}{t}
\\
&
\lesssim 2^{-j\left(2M+n\left(\frac{1}{p_0}-\frac{1}{q_0}\right)\right)}\ell(Q)^{\frac{n}{q_0}}
\,w(Q_i)^{-\frac{1}{p}}2^{-i\left(2M-\frac{n}{q_0}+\varepsilon\right)}.
\end{align*}
Therefore, from this inequality and \eqref{local-AQ}, we have that 
$$
\|\chi_{C_j(Q_i)}\nabla L^{-\frac{1}{2}}\widetilde{A}_Q^k\widetilde{\mm}_i\|_{L^p(w)}^p\lesssim e^{-c4^{i+j}}+2^{-jp\left(2M+\frac{n}{p_0}-\frac{r_0n}{p}\right)}
2^{-ip\left(2M+\varepsilon\right)}.
$$
This, \eqref{estimateII_{i1}}, and \eqref{M-choice-riesz} give us 
$
\sum_{i\geq 1}\|\nabla L^{-\frac{1}{2}}\widetilde{A}_Q^k\widetilde{\mm}_i\|_{L^p(w)}^p\leq C$,
which together with \eqref{termBQ} and in view of \eqref{splitingmolecule}, allows us to conclude the proof.
\end{proof}
\section{Interpolation results}\label{sec:interpol}
The aim of this section is to  prove the following interpolation result for Hardy spaces.
\begin{theorem}\label{thm:interpolationhardy}
Given $w\in A_{\infty}$ and $q\in \mathcal{W}_w(p_-(L),p_+(L))$, suppose $1\leq p_0<p_1<\frac{p_+(L)^{2,*}}{s_w}$ and $\frac{1}{p}=\frac{1-\theta}{p_0}+\frac{\theta}{p_1}$, $0<\theta<1$. Then
$$
[H^{p_0}_{\Scal_{\hh},q}(w),H^{p_1}_{\Scal_{\hh},q}(w)]_{\theta}=H^{p}_{\Scal_{\hh},q}(w).
$$
\end{theorem}
We denote by $[\,\,, \,]_{\theta}$ the complex interpolation method
described in \cite{Calderon} (see Section \ref{subsec:interpol}).
As we explained in the introduction we obtain this result from the corresponding one for the weighted tent spaces defined in \eqref{deftent} , (a real interpolation result involving weighted tent spaces was proved in \cite{Caoetal}):
\begin{theorem}\label{thm:interpolationtent}
Suppose $1\leq p_0<p_1<\infty$ and $\frac{1}{p}=\frac{1-\theta}{p_0}+\frac{\theta}{p_1}$, $0<\theta<1$. Then
$$
[T^{p_0}(w),T^{p_1}(w)]_{\theta}=T^p(w).
$$
\end{theorem}
\begin{remark}\label{remark:tent}
In the proof of the inclusion $
T^p\subset[T^{p_0},T^{p_1}]_{\theta}
$, in \cite[Lemma 5, Section 7]{CoifmanMeyerStein}
(following the notation there) the authors
 claim that the support of the function $\mathcal{A}(f_k)$ is contained in $O^*_{k}\setminus O_{k+1}$. It is easy to see that $\supp \mathcal{A}(f_k)\subset O^*_{k}$, but it is not clear that the support of $\mathcal{A}(f_k)$ is away from $O_{k+1}$. In fact, we can construct 1-dimensional
examples which show that this is false in general. This was noticed by A. Amenta in \cite[Remark 3.20]{Amenta}.

We refer to \cite{Bernal} and \cite{HarboureTorreaViviani} for a different proof of that inclusion and hence, of
 \cite[Theorem 4, Section 7]{CoifmanMeyerStein}.
\end{remark}
Here, we give a prove of Theorem \ref{thm:interpolationtent} based on  the proof of
\cite[Lemma 4 , Section 7]{CoifmanMeyerStein} to show that
$
[T^{p_0}(w),T^{p_1}(w)]_{\theta}\subset T^p(w)
$,
and on the proof of \cite[Lemma 5 , Section 7]{CoifmanMeyerStein} to show
that
$
 T^p(w)\subset[T^{p_0}(w),T^{p_1}(w)]_{\theta}.
$
However, in view of Remark \ref{remark:tent}, in order to show this last inclusion, we need to complete and slightly modify the proof given in \cite[Lemma 5, Section 7]{CoifmanMeyerStein}.
\subsection{Tent spaces interpolation: proof of Theorem \ref{thm:interpolationtent}}
As we said a few lines above, following the proof of \cite[Lemma 4 , Section 7]{CoifmanMeyerStein}, but using interpolation between $L^p(w)$ spaces (note that the proof in \cite[Theorem 5.1.1]{BerghLofstrom} also works in the weighted case), instead of the usual interpolation in $L^p(\R^n)$ spaces,  we get 
$$
[T^{p_0}(w),T^{p_1}(w)]_{\theta}\subset T^p(w),\quad 1\leq p_0<p<p_1<\infty,\quad \frac{1}{p}=\frac{1-\theta}{p_0}+\frac{\theta}{p_1}, \quad 0<\theta<1.
$$ 
As for the converse inclusion, 
fix 
$1\leq p_0<p<p_1<\infty$ and $0<\theta<1$ such that $\frac{1}{p}=\frac{1-\theta}{p_0}+\frac{\theta}{p_1}$, and
take a function $f\in T^p(w)$ such that $\|f\|_{T^p(w)}\leq 1$. Then, we need to find a function $F$, $z\rightarrow F(z)$, from the closed strip $0\leq \textrm{Re}(z)\leq 1$ to the Banach space $T^{p_0}(w)+T^{p_1}(w)$ (see Section \ref{subsec:interpol} for definitions).
 The function $F$ must be continuous and bounded on the full strip, with respect to the norm of $T^{p_0}(w)+T^{p_1}(w)$ , and analytic on the open strip, and such that $F(iy)\in T^{p_0}(w)$ is continuous in $T^{p_0}(w)$ and tends to zero as $|y|\rightarrow \infty$, and $F(1+iy)\in T^{p_1}(w)$ is continuous in $T^{p_1}(w)$ and tends to zero as $|y|\rightarrow \infty$. 
Besides, $F$ must satisfy that $F(\theta)=f$ in $T^p(w)$, and 
$$
\|F(iy)\|_{L^{p_0}(w)}+\|F(1+iy)\|_{L^{p_1}(w)}\leq C,
$$
uniformly on $f$.
To this end fix $\alpha>1$, to be determined during the proof,
and, for each $k\in \Z$, consider the sets
$
O_k:=\{x\in \R^n: \mathcal{A}^{\alpha}(f)(x)>2^k\},$ $ E_k:=\R^n\setminus O_k,
$
and, for some fixed $\gamma$, $0<\gamma<1$, the set 
$
E_k^*:=\left\{x\in \R^n:\forall\, r>0,\, \frac{|E_k\cap B(x,r)|}{|B(x,r)|}\geq \gamma\right\}
$
and its complement
$
O_k^*:=\R^n\setminus E_k^*=\{x\in \R^n: \mathcal{M}(\chi_{O_k})(x)>1-\gamma\},
$
where $\mathcal{M}$ is the centred Hardy-Littlewood maximal operator over balls. 
Besides, for every $k\in \Z$, we can consider the ``tent'' over  $O_k^*$, define by $\widehat{O^*_k}:=\{(y,t)\in \R^{n+1}_+:d(y,\R^n\setminus O^*_k)\geq t\}$. Note that  $\mathcal{R}(E^*_k):=\bigcup_{x\in E^*_k}\Gamma(x)=\R^{n+1}_+\setminus\widehat{O^*_k}$.
Additionally,
note that $\supp f\subset \left(\bigcup_{k\in \Z}\widehat{O^*_{k}}\setminus\widehat{O^*_{k+1}}\right)\bigcup \mathbb{F}$, where $\mathbb{F}\subset \R^{n+1}_+$ and $\int_{0}^{\infty}\int_{\R^n}\chi_{\mathbb{F}}(y,t)\frac{dy\,dt}{t^{n+1}}=0$.  This follows proceeding as in the proof of \cite[Proposition 5.1, part (a)]{MartellPrisuelos:II}.
Then, we can write $f=\sum_{k\in \Z}f\chi_{\widehat{O^*_k}\setminus \widehat{O^*_{k+1}}}=:\sum_{k\in \Z}f_k$  in $T^p(w)$.

Now,  consider the function
$
F(z):=e^{z^2-\theta^2}\sum_{k\in \Z}2^{k\left(\alpha(z)p-1\right)}f_k,
$
where $\alpha(z):=\frac{1-z}{p_0}+\frac{z}{p_1}$. 
We shall see that $F$ satisfies all the conditions that we mentioned above. First note that $F(\theta)=f$ in $T^p(w)$. Moreover,
for all $z\in \mathbb{C}$ such that $0<\textrm{Re}z<1$, applying Young's inequality, we have that
\begin{align}\label{openstrip}
|F(z)|&=
\left|e^{z^2-\theta^2}\sum_{k\in \Z}2^{k\left(\alpha(z)p-1\right)}f_k\right|
\leq e \sum_{k\in \Z}
2^{k\left(p\frac{1-\textrm{Re}(z)}{p_0}+p\frac{\textrm{Re}(z)}{p_1}-1\right)}|f_k|
\\ \nonumber
&
=
e\sum_{k\in \Z}\left(2^{k\left(\frac{p}{p_0}-1\right)}|f_k|\right)^{1-\textrm{Re}(z)}
\left(2^{k\left(\frac{p}{p_1}-1\right)}|f_k|\right)^{\textrm{Re}(z)}
\\ \nonumber
&
\leq
e(1-\textrm{Re}(z))\sum_{k\in \Z}2^{k\left(\frac{p}{p_0}-1\right)}|f_k|+
e\textrm{Re}(z)\sum_{k\in \Z}2^{k\left(\frac{p}{p_1}-1\right)}|f_k|
\\ \nonumber
&
\leq
e\sum_{k\in \Z}2^{k\left(\frac{p}{p_0}-1\right)}|f_k|+
e\sum_{k\in \Z}2^{k\left(\frac{p}{p_1}-1\right)}|f_k|.
\end{align}
Besides, for all $-\infty<y<\infty$, 
\begin{align}\label{tendtozero}
|F(iy)|\leq e^{-y^2}\sum_{k\in \Z}2^{k\left(\frac{p}{p_0}-1\right)}|f_k|\quad\textrm{and}\quad
|F(1+iy)|\leq e^{1-y^2}\sum_{k\in \Z}2^{k\left(\frac{p}{p_1}-1\right)}|f_k|.
\end{align}
Then, in order to see that $F$ satisfies the desired conditions, it suffices to show that 
\begin{align}\label{desiredconditions}
\left\|\sum_{k\in \Z}2^{k\left(\frac{p}{p_0}-1\right)}|f_k|\right\|_{T^{p_0}(w)}+\left\|\sum_{k\in \Z}2^{k\left(\frac{p}{p_1}-1\right)}|f_k|\right\|_{T^{p_1}(w)}\leq C.
\end{align}
Indeed, combining this with \eqref{tendtozero}, we obtain that  
$F(iy)\in T^{p_0}(w)$ is continuous in $T^{p_0}(w)$ and tends to zero as $|y|\rightarrow \infty$, and $F(1+iy)\in T^{p_1}(w)$ is continuous in $T^{p_1}(w)$ and tends to zero as $|y|\rightarrow \infty$.
On the other hand, by \eqref{openstrip}, \eqref{tendtozero}, and \eqref{desiredconditions}, we easily obtain that  $F$ is a continuous and bounded function with respect to the norm of $T^{p_0}(w)+T^{p_1}(w)$ on the full strip. Finally, to see that $F$ is analytic on the open strip we apply Morera's theorem for Banach-space valued functions. We have that $F(z)$ is continuous, so it just remains to show that for all triangle $T$ in the open set  $\widehat{C}:=\{z\in\mathbb{C}: 0<\textrm{Re}z<1\}$, we have that
$
\int_TF(z)=0.
$
To see this, consider for each $k\in \Z$ $g_k(z):=e^{z^2-\theta^2}2^{k(p\alpha(z)-1)}f_k$, we have that 
these functions are analytic on $\widehat{C}$. Then, for all $T$ triangle in $\widehat{C}$ and each $k\in \Z$, by Cauchy's theorem,
$
\int_Tg_k(z)=0.
$
Hence, it suffices to justify that we can take the sum in $k\in \Z$ out of the integral.  This follows by the dominated convergence theorem for Bochner integrals. Note that
\begin{align*}
\int_TF(z)=\sum_{j=1}^3\int_{0}^{t_j}F(\gamma_j(t))\gamma_j'(t)dt,
\end{align*}
where $\gamma_j$ is a parametrization of each side of the triangle $T$, and that, by \eqref{openstrip} and \eqref{desiredconditions},
for $j=1,2,3$,
\begin{align*}
\int_{0}^{t_j}\|F(\gamma_j(t))\gamma_j'(t)\|_{T^{p_0}(w)+T^{p_1}(w)}dt\lesssim \int_0^{t_j}\|F(\gamma_j(t))\|_{T^{p_0}(w)+T^{p_1}(w)}dt<\infty.
\end{align*}
Consequently the function $t\rightarrow F(\gamma_j(t))\gamma_j'(t)$, for all $t\in [0,t_j]$ is Bochner integrable.
Moreover, for all $M>0$ and $j=1,2,3$, again by \eqref{openstrip} and \eqref{desiredconditions},
\begin{align*}
\left\|\sum_{|k|\leq M}e^{\gamma_j(t)^2-\theta^2}2^{k(\alpha(\gamma_j(t))p-1)}
f_k\gamma_j'(t)\right\|_{T^{p_0}(w)+T^{p_1}(w)}\lesssim C,
\end{align*}
which implies that we can apply the dominated convergence theorem 
for Bochner integrals and then conclude that $F$ is analytic.
Besides \eqref{tendtozero}, \eqref{desiredconditions}, and the fact that $F(\theta)=f$ in $T^p(w)$ with $\|f\|_{T^p(w)}\leq 1$, imply that 
$f\in [T^{p_0}(w),T^{p_1}(w)]_{\theta}$ and that 
$
\|f\|_{[T^{p_0}(w),T^{p_1}(w)]_{\theta}}\lesssim\|f\|_{T^p(w)}.
$
 Thus, let us prove \eqref{desiredconditions}. Let $q$ be $p_0$ or $p_1$, then, since $\supp(\mathcal{A}f_k)\subset O_k^*$ and $O_{k+1}\subset O_{k+1}^*\subset O_{k}^*$, we have 
\begin{align*}
&\left\|\sum_{k\in \Z}2^{k\left(\frac{p}{q}-1\right)}|f_k|\right\|_{T^{q}(w)}
\leq
\sup_{\|\psi\|_{L^{q'}(w)}\leq 1}\sum_{k\in \Z}2^{k\left(\frac{p}{q}-1\right)}\int_{{O^*_k}}|\mathcal{A}f_{k}(x)||\psi(x)|w(x)dx
\\&\qquad
\leq
\sup_{\|\psi\|_{L^{q'}(w)}\leq 1}\sum_{k\in \Z}2^{k\left(\frac{p}{q}-1\right)}\int_{{O^*_k}\setminus O_{k+1}}|\mathcal{A}f_{k}(x)||\psi(x)|w(x)dx
\\&\qquad\qquad
+
\sup_{\|\psi\|_{L^{q'}(w)}\leq 1}\sum_{k\in \Z}2^{k\left(\frac{p}{q}-1\right)}\int_{O_{k+1}}|\mathcal{A}f_{k}(x)||\psi(x)|w(x)dx
=:\sup_{\|\psi\|_{L^{q'}(w)}\leq 1}I+\sup_{\|\psi\|_{L^{q'}(w)}\leq 1}II.
\end{align*}
Using that for $\alpha>1$, $\mathcal{A}f_k\leq \mathcal{A}^{\alpha}f_k$, and that for $x\in {O^*_k}\setminus O_{k+1}$, $\mathcal{A}^{\alpha}f_k(x)\leq 2^{k+1}$, we obtain
$
I\lesssim 2^{k\frac{p}{q}}\int_{O_k^{*}}|\psi(x)|w(x)dx.
$
Next we see that $II$ is controlled by the same expression in the right-hand side of the previous inequality.
For every $k\in \Z$, since $\mathcal{M}:L^{r}(w)\rightarrow L^{r,\infty}(w)$, for all $r>r_w$, we have that $w(O_k^*)\lesssim w(O_k)<\infty$. Then, we can take a Whitney decomposition of $O^*_{k+1}$: $O^*_{k+1}=\bigcup_{j\in \N}Q_{k+1}^j$, which satisfies
$$
\sqrt{n}\ell(Q_{k+1}^j)\leq d(Q_{k+1}^j,\R^n\setminus O_{k+1}^*)\leq 4\sqrt{n}\ell(Q_{k+1}^j),
$$
and the $Q_{k+1}^j$ have disjoint interiors. Now, for every $x\in Q_{k+1}^j$, we split $\mathcal{A}f_k(x)$ as follows
$$
\mathcal{A}f_k(x)\!\leq\!\left(\int_{0}^{\ell(Q_{k+1}^j)/2}\int_{B(x,t)}|f_k(y,t)|^2\frac{dy\,dt}{t^{n+1}}\right)^{\frac{1}{2}}+
\left(\int_{\ell(Q_{k+1}^j)/2}^{\infty}\int_{B(x,t)}|f_k(y,t)|^2\frac{dy\,dt}{t^{n+1}}\right)^{\frac{1}{2}}\!=:\!G_1(x)+G_2(x).
$$
On one hand, note that 
$$
\mathcal{E}\bigcap\left( \widehat{O^*_k}\setminus\widehat{O^*_{k+1}}\right):=\left\{(y,t)\in \Gamma(x): x\in Q_{k+1}^j, 0<t<\ell(Q_{k+1}^j)/2\right\}\bigcap\left( \widehat{O^*_k}\setminus\widehat{O^*_{k+1}}\right)=\emptyset.
$$
Indeed for $(y,t)\in\mathcal{E}$, since $d(y,Q_{k+1}^j)\leq t$, we have
$$
d(y,\R^n\setminus O^*_{k+1})\geq d(Q_{k+1}^j,\R^n\setminus O^*_{k+1})-d(y,Q_{k+1}^j)\geq 
\sqrt{n}\ell(Q_{k+1}^j)-t> (2\sqrt{n}-1)t\geq t,
$$
which implies that $(y,t)\in \widehat{O^*_{k+1}}$. 
Hence $G_1(x)=0$, for all $x\in Q_{k+1}^j$.

On the other hand, take $x_j\in \R^{n}\setminus O^*_{k+1}$ such that
$d(x_j,Q_{k+1}^j)\leq 4\sqrt{n}\ell(Q_{k+1}^j)$, and note that if $\ell(Q_{k+1}^j)/2\leq t<\infty$ and $x\in Q_{k+1}^j$, then $B(x,t)\subset B(x_j,\alpha t)$, for $\alpha\geq 11\sqrt{n}$. Indeed, for $x_0\in B(x,t)$, we have
$$
|x_0-x_j|\leq|x_0-x|+|x-x_j|< t+\sqrt{n}\ell(Q_{k+1}^j)+4\sqrt{n}\ell(Q_{k+1}^j)\leq t(1+2\sqrt{n}+8\sqrt{n})\leq 11\sqrt{n}t.
$$
Hence,
$
G_2(x)\leq \mathcal{A}^{\alpha}f_k(x_j)\leq  2^{k+1}, \,\forall\,x\in Q_{k+1}^j.
$
 Therefore, since $O_{k+1}\subset O_{k+1}^*\subset O_k^*$
 \begin{align*}
{II}
& \leq
2^{k\left(\frac{p}{q}-1\right)}\sum_{j\in \N}\int_{Q_{k+1}^j}|\mathcal{A}f_k(x)||\psi(x)|w(x)dx
 \\
 &
 \leq 
2^{k\left(\frac{p}{q}-1\right)}\sum_{j\in \N}\left(\int_{Q_{k+1}^j}|G_1(x)||\psi(x)|w(x)dx+
 \int_{Q_{k+1}^j}|G_2(x)||\psi(x)|w(x)dx\right)
  \\
 &
 \lesssim
2^{k\frac{p}{q}}\sum_{j\in \N}
 \int_{Q_{k+1}^j}|\psi(x)|w(x)dx
 = 2^{k\frac{p}{q}}
 \int_{O_{k+1}^*}|\psi(x)|w(x)dx
  \leq 2^{k\frac{p}{q}}
 \int_{O_{k}^*}|\psi(x)|w(x)dx.
 \end{align*}
Now consider respectively $\mathcal{M}_d$ and $\mathcal{M}_c$ the dyadic maximal function  and   the centred maximal function over cubes.
For some dimensional constant $c_n$, we have that  $\mathcal{M}(\chi_{O_k})(x)\leq c_n \mathcal{M}_c(\chi_{O_k})(x)$. Next, for each $k\in \Z$, we define the set 
$\widetilde{O}_{1-\gamma,k}:=\left\{x\in \R^n:\mathcal{M}_d(\chi_{O_k})(x)>\frac{1-\gamma}{4^nc_n}\right\}$; and we take a Calder\'on-Zygmud decomposition of this set at height $\frac{1-\gamma}{4^nc_n}$:
$\widetilde{O}_{1-\gamma,k}=\bigcup_{l\in \N}\widetilde{Q}_k^l$, where $\{\widetilde{Q}_k^l\}_{l\in \N}$ is a collection of disjoints dyadic cubes such that
 $$
\dashint_{\widetilde{Q}_k^l}\chi_{O_k}(x)dx\approx 1-\gamma.
$$
Then, since
$$
O_{k}^*\subset \left\{x\in \R^n:\mathcal{M}_c(\chi_{O_k})(x)>\frac{1-\gamma}{c_n}\right\}\subset \bigcup_{l\in \N}2\widetilde{Q}_k^l,
$$
(see \cite[proof of Lemma 2.12]{Duo}), for $r>r_w$, we have that
\begin{align}\label{OK-II-estimate}
\int_{O_k^{*}}|\psi(x)|w(x)dx&\leq\sum_{l\in \N}\frac{1}{(1-\gamma)^{r}}\int_{2\widetilde{Q}_k^l}(1-\gamma)^{r}|\psi(x)|w(x)dx
\\\nonumber&
\approx
\frac{1}{(1-\gamma)^{r}}\sum_{l\in \N}\int_{2\widetilde{Q}_k^l}\left(\dashint_{\widetilde{Q}_k^l}\chi_{O_k}(y)dy\right)^{r}|\psi(x)|w(x)dx
\\\nonumber&
\leq
\frac{1}{(1-\gamma)^{r}}\sum_{l\in \N}\int_{2\widetilde{Q}_k^l}\left(\dashint_{\widetilde{Q}_k^l}\chi_{O_k}(y)w(y)dy\right)\left(\dashint_{\widetilde{Q}_k^l}w^{1-r'}(y)dy\right)^{r-1}|\psi(x)|w(x)dx
\\\nonumber&
\lesssim
\frac{1}{(1-\gamma)^{r}}\sum_{l\in \N}\int_{2\widetilde{Q}_k^l}\left(\int_{\widetilde{Q}_k^l}\chi_{O_k}(y)w(y)dy\right)\,w(\widetilde{Q}_k^l)^{-1}|\psi(x)|w(x)dx
\\\nonumber &
\lesssim
\frac{1}{(1-\gamma)^{r}}\sum_{l\in \N}\int_{\widetilde{Q}_k^l}\chi_{O_k}(y)\mathcal{M}^w(\psi)(y)w(y)dy
\\\nonumber&
\leq\frac{1}{(1-\gamma)^{r}}\int_{O_k}\mathcal{M}^w(\psi)(y)w(y)dy,
\end{align}
where 
\begin{align}\label{weightedmaximal}
\mathcal{M}^wf(x):=\sup_{Q\ni x}\frac{1}{w(Q)}\int_{Q}|f(y)|w(y)dy.
\end{align}
Moreover, by \eqref{doublingcondition} and since $1<q'\leq \infty$, $\mathcal{M}^w:L^{q'}(w)\rightarrow L^{q'}(w)$. Therefore, by the estimates obtained for $I$ and $II$, by \eqref{OK-II-estimate}, and by \eqref{tentcomparison}, we conclude   
\begin{align*}
\left\|\sum_{k\in \Z}2^{k\left(\frac{p}{q}-1\right)}|f_k|\right\|_{T^{q}(w)}
&\lesssim
\sup_{\|\psi\|_{L^{q'}(w)}\leq 1}
\sum_{k\in \Z}2^{k\frac{p}{q}}\int_{O_k}\mathcal{M}^w(\psi)(x)w(x)dx
\\
&\lesssim
\sup_{\|\psi\|_{L^{q'}(w)}\leq 1}
\sum_{k\in \Z}\int_{2^{k-1}}^{2^{k}}\lambda^{\frac{p}{q}}\int_{\{x\in \R^n:\mathcal{A}^{\alpha}f(x)>\lambda\}}\mathcal{M}^w(\psi)(x)w(x)dx\frac{d\lambda}{\lambda}
\\
&=
\sup_{\|\psi\|_{L^{q'}(w)}\leq 1}
\int_{0}^{\infty}\lambda^{\frac{p}{q}}\int_{\{x\in \R^n:\mathcal{A}^{\alpha}f(x)>\lambda\}}\mathcal{M}^w(\psi)(x)w(x)dx\frac{d\lambda}{\lambda}
\\
&\approx
\sup_{\|\psi\|_{L^{q'}(w)}\leq 1}
\int_{\R^n}|\mathcal{A}^{\alpha}f(x)|^{\frac{p}{q}}\mathcal{M}^w(\psi)(x)w(x)dx
\\
&\lesssim
\sup_{\|\psi\|_{L^{q'}(w)}\leq 1}
\|\mathcal{A}^{\alpha}f\|_{L^{p}(w)}^{\frac{p}{q}}
\|\psi\|_{L^{q'}(w)}
\lesssim \|\mathcal{A}f\|_{L^{p}(w)}^{\frac{p}{q}}=\|f\|_{T^{p}(w)}^{\frac{p}{q}}\leq 1. 
\end{align*}
\qed
\subsection{Hardy spaces interpolation: proof of Theorem \ref{thm:interpolationhardy}}
 
As we explained above, in view of Theorem \ref{thm:interpolationtent}, it is enough to show that
 the Hardy spaces are retracts of the weighted tent spaces, i.e. that there exists an operator from any tent space to the corresponding Hardy space having a right inverse.

Fix $q\in \mathcal{W}_w(p_-(L),p_+(L))$. Note that for a function $f\in L^q(w)$ and $m\in \N$, we have the following Calder\'on reproducing formula of $f$, 
\begin{align}\label{calderonIII}
f=C_m\int_{0}^{\infty}\left((t^2L)^{m}e^{-t^2L}\right)^2f\frac{dt}{t} \quad \textrm{in}\quad L^q(w),
\end{align}
where $C_m$ is a positive constant and the equality is in $L^q(w)$.
\begin{remark}\label{remark:calderonreproducing}
A priori, by $L^2(\R^n)$ functional calculus we have the above equality for functions in $L^2(\R^n)$. But, as in Remark \ref{remark:extenssion}, for $q\in \mathcal{W}_w(p_-(L),p_+(L))$, we consider the infinitesimal generator $L_{q,w}$  of the $e^{-tL}$ on $L^q(w)$ (see \cite[Remark 3.5]{AuscherMartell:III}). Hence, by abuse of notation, we have the above Calder\'on reproducing formula for functions in $L^q(w)$, understanding that in $L^q(w)\setminus L^2(\R^n)$, $L$ denotes $L_{q,w}$.
\end{remark}
Besides, if we define for each $(y,t)\in \R^{n+1}_+$,  $\mathcal{F}(y,t):=(t^2L)^me^{-t^2L}f(y)$ and the operator  $\mathcal{Q}_{L,m}f:=\mathcal{F}$ acting over functions in $L^q(w)$, 
by \cite[Theorem 1.12, part (b)]{MartellPrisuelos},
$$
\|\mathcal{Q}_{L,m}f\|_{T^p(w)}=\|\Scal_{m,\hh}f\|_{L^p(w)}\leq \|\Scal_{\hh}f\|_{L^p(w)},
$$
 then $\mathcal{Q}_{L,m}$ is bounded from $\mathbb{H}^p_{\Scal_{\hh},q}(w)$ to $T^p(w)$, for all $q\in \mathcal{W}_w(p_-(L),p_+(L))$ and $1\leq p<\infty$. Thus,
by the definition of $H_{\Scal_{\hh},q}^p(w)$, it can be extended to a bounded operator, denoted by $\mathfrak{Q}_{L,m}$ from $H_{\Scal_{\hh},q}^p(w)$ to $T^p(w)$.
Similarly if we consider $\mathcal{Q}_{L^*,m}$, defined for all functions $f\in L^{q'}(w^{1-q'})$ by $\mathcal{Q}_{L^*,m}f(y,t):=(t^2L^*)^me^{-t^2L^*}f(y)$, for all $(y,t)\in  \R^{n+1}_+$. Again by \cite[Theorem 1.12, part (b)]{MartellPrisuelos}, by \cite[Lemma 4.4]{AuscherMartell:I}, and since $p_{\pm}(L^*)=p_{\mp}(L)'$, see \cite{Auscher}, we have that $\mathcal{Q}_{L^*,m}:L^{q'}(w^{1-q'})\rightarrow T^{q'}(w^{1-q'})$ for all $q'\in \mathcal{W}_{w^{1-q'}}(p_-(L^*),p_+(L^*))$. Moreover, for all $F\in T^2(\R^n)$,  its adjoint operator $(\mathcal{Q}_{L^*,m})^*$, has the following representation 
\begin{align}\label{extension}
(\mathcal{Q}_{L^*,m})^*F(y)=\int_0^{\infty}(t^2L)^me^{-t^2L}F(y,t)\frac{dt}{t}.
\end{align}
Then since for all $F\in T^2(\R^n)$ and $g\in L^2(\R^n)$,
\begin{multline*}
\left|\int_0^{\infty}(t^2L)^me^{-t^2L}F(y,t)\frac{dt}{t}\right|=\left|\int_{\R^n}\int_0^{\infty}
F(y,t)\overline{(t^2L^*)^me^{-t^2L^*}g(y)}\int_{B(y,t)}dx\frac{dt\,dy}{t^{n+1}}\right|
\\
\leq \int_{\R^n}|||F|||_{\Gamma(x)}|||(t^2L^*)^me^{-t^2L^*}g|||_{\Gamma(x)}dx,
\end{multline*}
where $|||F|||_{\Gamma(x)}=\left(\iint_{\Gamma(x)}|F(y,t)|^2\frac{dy\,dt}{t^{n+1}}\right)^{\frac{1}{2}}$. By a density argument, we conclude that $(\mathcal{Q}_{L^*,m})^*$
 has a bounded extension, denoted by $\widetilde{\mathcal{Q}}_{L,m}$, from $T^q(w)$ to $L^q(w)$, for all $q\in \mathcal{W}_w(p_-(L),p_+(L))$. If we replace $L$ by $L_{q,w}$, this extension satisfies the equality \eqref{extension} for functions on $T^q(w)$. But, abusing notation (see Remarks \ref{remark:extenssion} and \ref{remark:calderonreproducing}), we just write $L$. 

Besides we shall show that for all functions $F\in T^p(w)\cap T^q(w)$ and $m\in \N$ big enough, 
\begin{align}\label{S_HF(tent)}
\|\Scal_{\hh}\widetilde{\mathcal{Q}}_{L,m}F\|_{L^p(w)}\lesssim \|F\|_{T^p(w)},\quad \textrm{for all} \quad 1\leq p<\frac{p_+(L)^{2,*}}{s_w},
\end{align}
where $p_+(L)^{2,*}$ is defined in \eqref{qN*}.
Assuming this, since $T^p(w)\cap T^q(w)$ is dense in $T^p(w)$, $\widetilde{\mathcal{Q}}_{L,m{|_{T^p(w)}}}$ can be extended to a bounded operator, denoted by $\widetilde{\mathfrak{Q}}_{L,m}$, from $T^p(w)$ to $H^p_{\Scal_{\hh},q}(w)$. Then,  by \eqref{calderonIII}, we have that
$
C_m\widetilde{\mathfrak{Q}}_{L,m}\circ\mathfrak{Q}_{L,m}=I
$
in $\mathbb{H}_{\Scal_{\hh},q}^p(w)$, and by density in ${H}_{\Scal_{\hh},q}^p(w)$. Hence, for $w\in A_{\infty}$, $1\leq p<\frac{p_+(L)^{2,*}}{s_w}$, and  $q\in \mathcal{W}_w(p_-(L),p_+(L))$ the Hardy spaces $H_{\Scal_{\hh},q}^p(w)$ are retracts of the tent spaces $T^p(w)$.

Therefore, to finish the proof it just remains to show \eqref{S_HF(tent)}. 
Applying Minkowski's integral inequality we obtain
\begin{multline*}
\|\Scal_{\hh}\widetilde{\mathcal{Q}}_{L,m}F\|_{L^p(w)}
\\
\leq
\left(\int_{\R^n}\left(\int_{0}^{\infty}\left(\int_{0}^{t}\left(\int_{B(x,t)}|t^2Le^{-t^2L}(s^2L)^me^{-s^2L}F(y,s)|^2dy\right)^{\frac{1}{2}}\frac{ds}{s}\right)^2\frac{dt}{t^{n+1}}\right)^{\frac{p}{2}}w(x)dx\right)^{\frac{1}{p}}
\\\quad
+
\left(\int_{\R^n}\left(\int_{0}^{\infty}\left(\int_{t}^{\infty}\left(\int_{B(x,t)}|t^2Le^{-t^2L}(s^2L)^me^{-s^2L}F(y,s)|^2dy\right)^{\frac{1}{2}}\frac{ds}{s}\right)^2\frac{dt}{t^{n+1}}\right)^{\frac{p}{2}}w(x)dx\right)^{\frac{1}{p}}
\\
=:I+II.
\end{multline*}
We first show by extrapolation that $II\lesssim \|F\|_{T^p(w)}$, for every $p$ as in \eqref{S_HF(tent)} and every $m\in \N$. To this end, in view of Theorem \ref{thm:extrapol}, part ($a$), (or  part ($b$) if $p_+(L)^{2,*}=\infty$) it is enough to consider the case $p=2$ and $w\in RH_{\left(\frac{p_+(L)^{2,*}}{2}\right)'}$. That is, to prove that, for every $w\in RH_{\left(\frac{p_+(L)^{2,*}}{2}\right)'}$ and $m\in \N$,
$$
\mathcal{II}\!:=\!\left(\int_{\R^n}\int_{0}^{\infty}\left(\int_{t}^{\infty}\left(\int_{B(x,t)}|t^2Le^{-t^2L}(s^2L)^me^{-s^2L}F(y,s)|^2dy\right)^{\frac{1}{2}}\frac{ds}{s}\right)^2\frac{dt}{t^{n+1}}w(x)dx\right)^{\frac{1}{2}}\!\lesssim \|F\|_{T^2(w)}.
$$
Under this assumption, note that we can find $q_0$ and $r$ so that $2 < q_0 < p_+(L),$
$\frac{q_0}{2}\leq r<\infty$, $w\in RH_{r'}$, and
\begin{align}\label{positiveinterpolhardy}
2+\frac{n}{2r}-\frac{n}{q_0}>0.
\end{align}
Indeed, if
$n>2p_+(L)$, since $w\in RH_{\left(\frac{p_+(L)^{2,*}}{2}\right)'}$, we have that 
$s_w<\frac{np_+(L)}{2(n-2p_+(L))}$. Therefore, there exist $\varepsilon_1>0$ small enough and $2<q_0<p_+(L)$ close enough to $p_+(L)$ so that 
$$
s_w<\frac{nq_0}{2(1+\varepsilon_1)(n-2q_0)}.
$$
Besides, there exists $\varepsilon_2>0$ so that
$$
q_0<\frac{nq_0}{(1+\varepsilon_2)(n-2q_0)}.
$$
Hence, taking $\varepsilon_0:=\min\{\varepsilon_1,\varepsilon_2\}$ and $r:=\frac{nq_0}{2(1+\varepsilon_0)(n-2q_0)}$, we have that $2<q_0<p_+(L)$, $q_0/2\leq r<\infty$, and $w\in RH_{r'}$. Moreover
\begin{align*}
2+\frac{n}{2r}-\frac{n}{q_0}=\varepsilon_0\left(\frac{n}{q_0}-2\right)
>
\varepsilon_0\left(\frac{n}{p_+(L)}-2\right)>0.
\end{align*}
If now we consider $n\leq 2p_+(L)$, we have that $p_+(L)^{2,*}=\infty$. Then, note that the assumption $w\in RH_{\left(\frac{p_+(L)^{2,*}}{2}\right)'}$ becomes 
$w\in A_{\infty}$. Hence, we fix $r>s_w$, and $q_0$ satisfying
$\max\left\{2,\frac{2rp_+(L)}{p_+(L)+2r}\right\}<q_0<\min\left\{p_+(L),2r\right\}$ if $p_+(L)<\infty$, and $q_0=2r$ if $p_+(L)=\infty$. Therefore, we have that $2<q_0<p_+(L)$, $q_0/2\leq r<\infty$, and $w\in RH_{r'}$. Besides, 
\begin{align*}
2+\frac{n}{2r}-\frac{n}{q_0}
>
2-\frac{n}{p_+(L)}\geq 0.
\end{align*}
Keeping these choices in mind, we apply the $L^2(\R^n)-L^2(\R^n)$ off-diagonal estimates satisfied by  $\{e^{-t^2L}\}_{t>0}$, change the variable $s$ into $st$, and apply Jensen's inequality and Minkowski's integral inequality. Then, we have
\begin{align*}
\mathcal{II}
&\lesssim
\sum_{j\geq 1}e^{-c4^j}\left(\int_{\R^n}
\int_0^{\infty}\left(\int_{t}^{\infty}\frac{t^2}{s^2}\left(\int_{B(x,2^{j+1}t)}|(s^2L)^{m+1}e^{-s^2L}F(y,s)|^2dy\right)^{\frac{1}{2}}\frac{ds}{s}\right)^{{2}}\frac{dt}{t^{n+1}}w(x)dx\right)^{\frac{1}{2}}
\\\nonumber
&\lesssim
\sum_{j\geq 1}e^{-c4^j}
\int_{1}^{\infty}s^{-2}\left(\int_{\R^n}\int_0^{\infty}\left(\int_{B(x,2^{j+1}t)}|((st)^2L)^{m+1}e^{-(st)^2L}F(y,st)|^{q_0}\frac{dy}{t^n}\right)^{\frac{2}{q_0}}\frac{dt}{t}w(x)dx\right)^{\frac{1}{2}}\frac{ds}{s}
\\\nonumber&
=:
\sum_{j\geq 1}e^{-c4^j}
\int_{1}^{\infty}s^{-2}\left(\int_{\R^n}\mathcal{J}(x,s)^2w(x)dx\right)^{\frac{1}{2}}\frac{ds}{s}.
\end{align*}
Note now that, applying Fubini's theorem, \cite[Proposition 3.30]{MartellPrisuelos}, and changing the variable $t$ into $t/s$; and next, applying the $L^2(\R^n)-L^{q_0}(\R^n)$ off-diagonal estimates satisfied by the family $\{(t^2L)^{m+1}e^{-t^2L}\}_{t>0}$ and \cite[Proposition 3.2]{MartellPrisuelos}, taking $r_0>r_w$, and also recalling our choices of $q_0$ and $r$, we obtain that, for every $s>1$,
\begin{align*}
\int_{\R^n}\mathcal{J}(x,s)^2w(x)dx
&\lesssim s^{-\frac{n}{r}}
\int_0^{\infty}\int_{\R^n}
\left(\int_{B(x,2^{j+1}st)}|((st)^2L)^{m+1}e^{-(st)^2L}F(y,st)|^{q_0}\frac{dy}{t^n}\right)^{\frac{2}{q_0}}w(x)dx\frac{dt}{t}
\\
&\approx
s^{-\frac{n}{r}+\frac{2n}{q_0}}
\int_{\R^n}\int_0^{\infty}
\left(\int_{B(x,2^{j+1}t)}|(t^2L)^{m+1}e^{-t^2L}F(y,t)|^{q_0}\frac{dy}{t^n}\right)^{\frac{2}{q_0}}\frac{dt}{t}w(x)dx
\\
&\lesssim s^{-\frac{n}{r}+\frac{2n}{q_0}}\sum_{l\geq 1}e^{-c4^l}
\int_{\R^n}\int_0^{\infty}
\int_{B(x,2^{j+l+2}t)}|F(y,t)|^{2}\frac{dy\,dt}{t^{n+1}}w(x)dx
\\
&\lesssim 2^{jnr_0}\sum_{l\geq 1}e^{-c4^l}
s^{-\frac{n}{r}+\frac{2n}{q_0}}\|F\|_{T^2(w)}^2.
\end{align*}
Hence, by \eqref{positiveinterpolhardy}, we have
\begin{align*}
\mathcal{II}
\lesssim \sum_{j\geq 1}e^{-c4^j}
\int_{1}^{\infty}s^{-2-\frac{n}{2r}+\frac{n}{q_0}}\frac{ds}{s}
\|F\|_{T^2(w)}
\lesssim 
\|F\|_{T^2(w)}
\end{align*}
which, as we observed above, implies that
$
II\lesssim \|F\|_{T^p(w)}, 
$
for all $1\leq p <\frac{p_+(L)^{2,*}}{s_w}$ and all $m\in \N$.

Next, in order to estimate $I$ we apply the $L^2(\R^n)-L^2(\R^n)$ off-diagonal estimates satisfied by  $\{(t^2L)^{m+1}e^{-t^2L}\}_{t>0}$ and $\{e^{-s^2L}\}_{s>0}$, and \cite[Lemma 2.3]{HofmannMartell} recalling that, in this case, $s<t$. Then,
\begin{multline*}
\Bigg(\int_{B(x,t)}|t^2Le^{-t^2L}(s^2L)^me^{-s^2L}F(y,s)|^2dy\Bigg)^{\frac{1}{2}}
=
\left(\frac{s^2}{t^2}\right)^m\left(\int_{B(x,t)}|(t^2L)^{m+1}e^{-t^2L}e^{-s^2L}F(y,s)|^2dy\right)^{\frac{1}{2}}
\\
\lesssim
\left(\frac{s^2}{t^2}\right)^m\sum_{j\geq 1}e^{-c4^j}\left(\int_{B(x,2^{j+1}t)}|F(y,s)|^2dy\right)^{\frac{1}{2}}
.
\end{multline*}
Using this, changing the variable $s$ into $st$, applying Minkowski's integral inequality twice, changing the variable $t$ into $t/s$,  applying \cite[Proposition 3.2]{MartellPrisuelos}, and taking $2m>\frac{nr_0}{p}-\frac{n}{2}$, where $r_0>\max\{p/2,r_w\}$, we obtain that
\begin{align*}
I&
\lesssim 
\sum_{j\geq 1}e^{-c4^j}\left(\int_{\R^n}\left(\int_{0}^{\infty}
\left(\int_{0}^{1}s^{2m}\left(\int_{B(x,2^{j+1}t)}|F(y,st)|^2dy\right)^{\frac{1}{2}}\frac{ds}{s}\right)^2\frac{dt}{t^{n+1}}\right)^{\frac{p}{2}}w(x)dx\right)^{\frac{1}{p}}
\\&
\lesssim
\sum_{j\geq 1}e^{-c4^j}\int_{0}^{1}s^{2m}\left(\int_{\R^n}\left(\int_{0}^{\infty}
\int_{B(x,2^{j+1}t)}|F(y,st)|^2\frac{dy\,dt}{t^{n+1}}\right)^{\frac{p}{2}}w(x)dx\right)^{\frac{1}{p}}\frac{ds}{s}
\\&
=
\sum_{j\geq 1}e^{-c4^j}\int_{0}^{1}s^{2m+\frac{n}{2}}\left(\int_{\R^n}\left(\int_{0}^{\infty}\int_{B(x,2^{j+1}t/s)}|F(y,t)|^2\frac{dy\,dt}{t^{n+1}}\right)^{\frac{p}{2}}w(x)dx\right)^{\frac{1}{p}}\frac{ds}{s}
\\&
\lesssim
\sum_{j\geq 1}2^{j\frac{nr_0}{p}}e^{-c4^{j}}\int_{0}^{1}s^{2m+\frac{n}{2}-\frac{nr_0}{p}}\frac{ds}{s}
\|F\|_{T^p(w)}
\lesssim
\|F\|_{T^p(w)},
\end{align*}
which finishes the proof.
\qed
%
\section{Boundedness improvement for conical square functions}
In \cite{AuscherMartell:III} the authors observed that if $L$ is a real operator and $w\in A_{\infty}$ such that $\mathcal{W}_w(p_-(L),p_+(L))\neq \emptyset$, then
$\mathcal{W}_w(p_-(L),p_+(L))=\textrm{Int}\mathcal{J}_w(L)=:
(\widehat{p}_-(L),\widehat{p}_+(L))$ (see page \pageref{qN*} for definitions). However, in the case of complex operators we do not know whether $\mathcal{J}_w(L)$ and $\mathcal{W}_w(p_-(L),p_+(L))$ have different end-points.

 Therefore, motivated by getting the expected lower exponent in Proposition \ref{prop:1-Riesz} in the case of complex operators (see for instance \cite[Theorem 6.2]{AuscherMartell:III}), we improve the range of boundedness, obtained in \cite[Theorem 1.12]{MartellPrisuelos}, of $\Scal_{\hh}$ in the case that $w\in A_{\infty}$ with $\mathcal{W}_w(p_-(L),p_+(L))\neq \emptyset$.
\begin{theorem}\label{thm:SF-heat-III}
Given $w\in A_{\infty}$ such that $\mathcal{W}_w(p_-(L),p_+(L))\neq \emptyset$, for all $p\in (\widehat{p}_{-}(L),\infty)$, there hold:
\begin{list}{$(\theenumi)$}{\usecounter{enumi}\leftmargin=1cm \labelwidth=1cm\itemsep=0.2cm\topsep=.2cm \renewcommand{\theenumi}{\alph{enumi}}}

\item $\Scal_{\hh}$ is bounded on  $L^p(w)$.

\item Given $m\in\mathbb{N}$, $\Scal_{m,\hh}$, $\Grm_{m, \hh}$, and $\Gcal_{m, \hh}$ are bounded on $L^p(w)$.
\end{list}
If we further assume that $\mathcal{W}_w(q_-(L),q_+(L))\neq \emptyset$
we have that, for  $p\in (\widehat{p}_{-}(L),\infty)$, $\Grm_{\hh}$ and  $\Gcal_{\hh}$ are bounded on  $L^p(w)$.
\end{theorem}
\begin{proof}
Note that $\Gcal_{\hh}\leq 2\Scal_{\hh}+\Grm_{\hh}$, then
by  \cite[Theorem 1.14]{MartellPrisuelos}, we just need to prove the theorem for $\Scal_{\hh}$ and $\Grm_{\hh}$. Let $\mathcal{Q}$ be $\Scal_{\hh}$ or $\Grm_{\hh}$

By \cite[Theorem 2.4]{AuscherMartell:III}, to conclude our result, it is enough to prove that for every ball $B=B(x_B,r_B)\subset \R^n$
\begin{align}\label{1-teorema}
\left(\dashint_{C_j(B)}|\mathcal{Q}B_{r_B}f(x)|^{p}dw\right)^{\frac{1}{p}}\lesssim g(j)\left(\dashint_{B}|f(x)|^{p}dw\right)^{\frac{1}{p}},\quad  \textrm{for all}\quad j\geq 2;
\end{align}
and, 
\begin{align}\label{2-teorema}
\left(\dashint_{C_j(B)}|A_{r_B}f(x)|^{q}dw\right)^{\frac{1}{q}}\lesssim g(j)\left(\dashint_{B}|f(x)|^{p}dw\right)^{\frac{1}{p}},\quad\textrm{for all}\quad j\geq 1,
\end{align}
where $A_{r_B}:=I-(I-e^{-r_B^2L})^M$ and $B_{r_B}:=I-A_{r_B}$, for some $M\in \N$ arbitrarily large, $q$ is such that $\mathcal{Q}$ is bounded on $L^q(w)$, $f\in L^{\infty}_c(\R^n)$ such that $\supp f\subset B$, and $g(j)$ is such that $\sum_{j\geq 1}g(j)2^{nr}$, for some $r>r_w$.

We start by taking $q\in \mathcal{W}_{w}(p_-(L),p_+(L))$ when $\mathcal{Q}=\Scal_{\hh}$ and $q\in \mathcal{W}_{w}(q_-(L),q_+(L))$ when $\mathcal{Q}=\Grm_{\hh}$, by \cite[Theorem 1.12]{MartellPrisuelos}, we know that, in any case, $\mathcal{Q}$ is bounded on $L^q(w)$. Besides, also by that result, we only need to consider the case $\widehat{p}_-(L)<p\leq r_wp_-(L)$ (we recall that $p_-(L)=q_-(L)$). Next, we fix $p_0$ so that $p_-(L)<p_0<\min\{2,q\}$ and $w\in A_{q/p_0}$. 

The proof of \eqref{2-teorema} follows by expanding the binomial and using that, for $1\leq k\leq M$, $e^{-\sqrt{k}r_B L}$ satisfies $L^p(w)-L^q(w)$ off-diagonal estimates on balls, (see \cite{AuscherMartell:II, AuscherMartell:III}).

As for  \eqref{1-teorema}, 
first note that it is enough to prove
$$
I:=\left(\dashint_{C_j(B)}\left(\int_{0}^{\infty}\int_{B(x,t)}|T_tB_{r_B}f(y)|^{2}\frac{dy\,dt}{t^{n+1}}\right)^{\frac{p}{2}}dw\right)^{\frac{1}{p}}
\lesssim g(j)\left(\dashint_B|f(x)|^pdw\right)^{\frac{1}{p}},
$$
for $T_t$ being $t^2Le^{-t^2L}$ or $t\nabla_ye^{-t^2L}$.
Splitting the integral in $t$ we have that
\begin{multline}\label{originalterm}
I
\leq
\left(\dashint_{C_j(B)}\left(\int_{0}^{r_B}\int_{B(x,t)}|T_tB_{r_B}f(y)|^{2}\frac{dy\,dt}{t^{n+1}}\right)^{\frac{p}{2}}dw\right)^{\frac{1}{p}}
\\
+
\left(\dashint_{C_j(B)}\left(\int_{r_B}^{\infty}\int_{B(x,t)}|T_tB_{r_B}f(y)|^{2}\frac{dy\,dt}{t^{n+1}}\right)^{\frac{p}{2}}dw\right)^{\frac{1}{p}}
=:I_1+I_2.
\end{multline}
In order to estimate  $I_2$, consider $B_{r_B,t}:=(e^{-t^2L}-e^{-(t^2+r_B^2)L})^M$. Then, changing the variable $t$ into $t\sqrt{M+1}=:tC_M$ and 
applying that $\left\{T_{t}\right\}_{t>0}$ satisfies $L^{p_0}(\R^n)-L^2(\R^n)$ off-diagonal estimates  (see Section \ref{sec:od}), we have
\begin{multline*}
I_2 \lesssim
\left(\dashint_{C_j(B)}\left(\int_{\frac{r_B}{C_M}}^{\infty}
\int_{B(x,tC_M)}|T_{t}B_{r_B,t}f(y)|^{2}\frac{dy\,dt}{t^{n+1}}\right)^{\frac{p}{2}}dw\right)^{\frac{1}{p}}
\\
\lesssim\sum_{i\geq 1}e^{-c4^i}
\left(\dashint_{C_j(B)}\left(\int_{\frac{r_B}{C_M}}^{\infty}
\left(\int_{B(x,2^{i+1}tC_M)}|B_{r_B,t}f(y)|^{p_0}\frac{dy}{t^{n}}\right)^{\frac{2}{p_0}}\frac{dt}{t}\right)^{\frac{p}{2}}dw\right)^{\frac{1}{p}}.
\end{multline*}
Besides, since $w\in A_{\frac{q}{p_0}}$, note that we have the following estimate for the integral in $y$: 
\begin{align}\label{toweight}
&\left(\int_{B(x,2^{i+1}C_Mt)}|B_{r_B,t}f(y)|^{p_0}\frac{dy}{t^{n}}\right)^{\frac{2}{p_0}}
\\
&\quad\lesssim 2^{\frac{i2n}{p_0}}\left(\int_{B(x,2^{i+1}C_Mt)}|B_{r_B,t}f(y)|^{q}w(y)dy\right)^{\frac{2}{q}}\left(\dashint_{B(x,2^{i+1}tC_M)}w(y)^{1-\left(\frac{q}{p_0}\right)'}dy\right)^{\frac{2}{q}\left(\frac{q}{p_0}-1\right)}(2^{i}t)^{-\frac{2n}{q}}
\nonumber
\\\nonumber
&\quad\lesssim 2^{\frac{i2n}{p_0}}\left(\dashint_{B(x,2^{i+1}C_Mt)}|B_{r_B,t}f(y)|^{q}{dw}\right)^{\frac{2}{q}}.
\end{align}
 By \eqref{toweight}, we can split $I_2$ as follows:
\begin{multline*}
I_2\lesssim \sum_{i=1}^{j-2}e^{-c4^i}2^{\frac{in}{p_0}}
\left(\dashint_{C_j(B)}\left(\int_{\frac{r_B}{C_M}}^{\infty}\left(\dashint_{B(x,2^{i+1}tC_M)}|B_{r_B,t}f(y)|^{q}dw\right)^{\frac{2}{q}}\frac{dt}{t}\right)^{\frac{p}{2}}dw\right)^{\frac{1}{p}}
\\
+e^{-c4^j}\sum_{i\geq j-1}e^{-c4^i}2^{\frac{in}{p_0}}
\left(\dashint_{C_j(B)}\left(\int_{\frac{r_B}{C_M}}^{\infty}\left(\dashint_{B(x,2^{i+1}tC_M)}|B_{r_B,t}f(y)|^{q}dw\right)^{\frac{2}{q}}\frac{dt}{t}\right)^{\frac{p}{2}}dw\right)^{\frac{1}{p}}
\\
=:\sum_{i=1}^{j-2}e^{-c4^i}I_{2i}^1+e^{-c4^j}\sum_{i\geq j-1}e^{-c4^i}I_{2i}^2.
\end{multline*}
The sum $\sum_{i=1}^{j-2}e^{-c4^i}I_{2i}^1$ only appears when $j\geq 3$. In this case,
we split the integral in $t$ and observe that for  $x\in C_j(B)$, and $r_BC_M^{-1}<t<r_BC_M^{-1}2^{j-i-2}$, we have that $B(x,2^{i+1}C_Mt)\subset 2^{j+2}B\setminus 2^{j-1}B$; besides, for $1\leq i\leq j-2$ and $t\geq r_BC_M^{-1}2^{j-i-2}$, $B,B(x,2^{i+1}C_Mt)\subset B(x_B,2^{j+2}C_Mt)$. Then, applying \eqref{doublingcondition}, Proposition \ref{prop:offdiagonalweightedballs}, and the fact that $tC_M\geq r_B$, we obtain 
\begin{align*}
I_{2i}^1&\lesssim 2^{\frac{jn}{p_0}}
\left(\int_{\frac{r_B}{C_M}}^{\frac{2^{j-i-2}r_B}{C_M}}\left(\frac{r_B}{t}\right)^{\frac{2n}{p_0}}\left(\sum_{l=j-1}^{j+1}\left(\dashint_{C_{l}(B)}|B_{r_B,t}f(y)|^{q}{dw}\right)^{\frac{1}{q}}\right)^{2}\frac{dt}{t}\right)^{\frac{1}{2}}
\\&\qquad\qquad\qquad
+ 2^{\frac{jn}{p_0}}
\left(\int_{\frac{2^{j-i-2}r_B}{C_M}}^{\infty}\left(\dashint_{B(x_B,2^{j+2}C_M t)}|B_{r_B,t}(f\chi_{B(x_B,2^{j+2}C_M t)})(y)|^{q}{dw}\right)^{\frac{2}{q}}\frac{dt}{t}\right)^{\frac{1}{2}}
\\
&\lesssim \frac{2^{j\left(\frac{n}{p_0}+\theta_2\right)}\|f\|_{L^p(w)}}{w(B)^{\frac{1}{p}}}
\left(2^{j\theta_1}
\left(\int_{\frac{r_B}{C_M}}^{\frac{2^{j-i-2}r_B}{C_M}}\left(\frac{r_B}{t}\right)^{\frac{2n}{p_0}+4M+2\theta_2}e^{-c\frac{4^jr_B^2}{t^2}}\frac{dt}{t}\right)^{\frac{1}{2}}
\!\!+\!
\left(\int_{\frac{2^{j-i-2}r_B}{C_M}}^{\infty}\left(\frac{r_B}{t}\right)^{4M}\frac{dt}{t}\right)^{\frac{1}{2}}\right)
\\
&\lesssim 2^{i2M}
2^{-j\left(2M-\frac{n}{p_0}-\theta_1-\theta_2\right)}
\left(\dashint_{B}|f(y)|^{p}{dw}\right)^{\frac{1}{p}}.
\end{align*}
The estimate of $I_{2i}^2$ follows applying Proposition \ref{prop:offdiagonalweightedballs} and the fact that for $x\in C_j(B)$, $j\geq 2$, $i\geq j-1$, and $tC_M\geq r_B$, we have that $B,B(x,2^{i+1}C_Mt)\subset B(x_B,2^{i+3}C_M t)$
\begin{multline*}
I_{2i}^2
\lesssim
2^{\frac{in}{p_0}}
\left(\int_{\frac{r_B}{C_M}}^{\infty}
\left(\dashint_{B(x_B,2^{i+3}C_M t)}|B_{r_B,t}(f\chi_{B(x_B,2^{i+3}C_M t)})(y)|^{q}{dw}\right)^{\frac{2}{q}}\frac{dt}{t}\right)^{\frac{1}{2}}
\\
\lesssim
2^{i\left(\frac{n}{p_0}+\theta_2\right)}
\left(\int_{\frac{r_B}{C_M}}^{\infty}\left(\frac{r_B}{t}\right)^{4M}
\frac{dt}{t}\right)^{\frac{1}{2}}\left(\dashint_{B}|f(y)|^{p}dw\right)^{\frac{1}{p}}
\lesssim
2^{i\left(\frac{n}{p_0}+\theta_2\right)}
\left(\dashint_{B}|f(y)|^{p}dw\right)^{\frac{1}{p}}.
\end{multline*}
Therefore, for all $j\geq 2$, we have
\begin{align}\label{term2}
I_2\lesssim \left(2^{-j\left(2M-\frac{n}{p_0}-\theta_1-\theta_2\right)}+e^{-c4^j}\right)
\left(\dashint_{B}|f(y)|^{p}{dw}\right)^{\frac{1}{p}}.
\end{align}
In order to estimate $I_1$,
 we expand the binomial. Then, 
\begin{multline}\label{term1}
I_1\leq\left(\dashint_{C_j(B)}\left(\int_{0}^{r_B}
\int_{B(x,t)}|T_tf(y)|^{2}\frac{dy\,dt}{t^{n+1}}\right)^{\frac{p}{2}}dw\right)^{\frac{1}{p}}
\\
+
\sum_{k=1}^{M}C_{k,M}\left(\dashint_{C_j(B)}
\left(\int_{0}^{r_B}
\int_{B(x,t)}|T_te^{-kr_B^2L}f(y)|^{2}\frac{dy\,dt}{t^{n+1}}\right)^{\frac{p}{2}}dw\right)^{\frac{1}{p}}=:I_1^1+\sum_{k=1}^{M}C_{k,M}\mathcal{I}_{k}.
\end{multline}
We first estimate $I_1^1$, noticing that $T_t=cT_{t/\sqrt{2}}e^{-\frac{t^2}{2}L}$, and applying the $L^{p_0}(\R^n)- L^2(\R^n)$ off-diagonal estimates satisfied by $T_{t/\sqrt{2}}$,  we have
 \begin{align*}
 I_1^1&\lesssim \sum_{i\geq 1}e^{-c4^{i}}
 \left(\dashint_{C_j(B)}\left(\int_0^{r_B}
\left(\int_{B(x,2^{i+1}t)}|e^{-\frac{t^2}{2}L}f(y)|^{p_0}\frac{dy}{t^{n}}\right)^{\frac{2}{p_0}}\frac{dt}{t}\right)^{\frac{p}{2}}dw\right)^{\frac{1}{p}}
\\
&\lesssim \sum_{i=1}^{j-2}e^{-c4^{i}}
 \left(\dashint_{C_j(B)}\left(\int_0^{r_B}
\left(\int_{B(x,2^{i+1}t)}|e^{-\frac{t^2}{2}L}f(y)|^{p_0}\frac{dy}{t^{n}}\right)^{\frac{2}{p_0}}\frac{dt}{t}\right)^{\frac{p}{2}}dw\right)^{\frac{1}{p}}
\\
& \quad+e^{-c4^j}\sum_{i\geq j-1}e^{-c4^{i}}
 \left(\dashint_{C_j(B)}\left(\int_0^{r_B}
\left(\int_{B(x,2^{i+1}t)}|e^{-\frac{t^2}{2}L}f(y)|^{p_0}\frac{dy}{t^{n}}\right)^{\frac{2}{p_0}}\frac{dt}{t}\right)^{\frac{p}{2}}dw\right)^{\frac{1}{p}}
\\&
=:\sum_{i=1}^{j-2}e^{-c4^{i}}II_{i}+e^{-c4^j}\sum_{i\geq j-1}e^{-c4^{i}}III_{i},
 \end{align*}
where the sum $\sum_{i=1}^{j-2}e^{-c4^{i}}II_{i}$ only appears if $j\geq 3$. Then,  proceeding  as in \eqref{toweight}, and noticing that for $x\in C_j(B)$, $j\geq 3$, $1\leq i\leq j-2$, and $0<t<r_B$, we have that
$B(x,2^{i+1}t)\subset 2^{j+2}B\setminus 2^{j-1}B$, applying the $L^{p}(w)-L^{q}(w)$ off-diagonal estimates on balls satisfied by $e^{-\frac{t^2}{2}L}$ (see \cite{AuscherMartell:II, AuscherMartell:III}), we obtain that, for some constants $\theta_1,\theta_2>0$, 
\begin{multline*}
II_i\lesssim
\left(\int_{0}^{r_B}\left(\frac{2^jr_B}{t}\right)^{\frac{2n}{p_0}}
\left(\sum_{l=j-1}^{j+1}\left(\dashint_{C_l(B)}|e^{-\frac{t^2}{2}L}f(y)|^{q}dw\right)^{\frac{1}{q}}\right)^2\frac{dt}{t}\right)^{\frac{1}{2}}
\\
\lesssim 2^{j\theta_1}
e^{-c4^j}
\left(\dashint_{B}|f(y)|^{p}dw\right)^{\frac{1}{p}}
\left(\int_{0}^{r_B}
\left(\frac{2^{j}r_B}{t}\right)^{\frac{2n}{p_0}+2\theta_2}
e^{-c\frac{4^jr_B^2}{t^2}}\frac{dt}{t}\right)^{\frac{1}{2}} 
\lesssim 
e^{-c4^j}
\left(\dashint_{B}|f(y)|^{p}dw\right)^{\frac{1}{p}}.
\end{multline*}
Now we split $III_i$ as follows
\begin{multline*}
III_i\lesssim
\left(\dashint_{C_j(B)}\left(\int_0^{\frac{r_B}{2^{i+1}}}
\left(\int_{B(x,2^{i+1}t)}|e^{-\frac{t^2}{2}L}f(y)|^{p_0}\frac{dy}{t^{n}}\right)^{\frac{2}{p_0}}\frac{dt}{t}\right)^{\frac{p}{2}}dw\right)^{\frac{1}{p}}
\\
+
\left(\dashint_{C_j(B)}\left(\int_{\frac{r_B}{2^{i+1}}}^{r_B}
\left(\int_{B(x,2^{i+1}t)}|e^{-\frac{t^2}{2}L}f(y)|^{p_0}\frac{dy}{t^{n}}\right)^{\frac{2}{p_0}}\frac{dt}{t}\right)^{\frac{p}{2}}dw\right)^{\frac{1}{p}}=:III_i^1+III_i^2.
\end{multline*}
Note that for $x\in C_j(B)$ and $0<t<r_{B}/2^{i+1}$ we have that $B(x,2^{i+1}t)\subset 2^{j+2}B\setminus 2^{j-1}B$. Then, if $j\geq 3$, the estimate of $III_i^1$ follows as the estimate of $II_i$.
If $j=2$, we write  $B(x,2^{i+1}t)\subset \sum_{l=2}^{3}C_l(B)\cup (4B\setminus 2B)$ and proceed as in the estimate of $II_i$, applying \cite[Lemma 6.5]{AuscherMartell:II}. Hence, we obtain
$
III_i^1\lesssim 
e^{-c4^j}
\left(\dashint_{B}|f(y)|^{p}dw\right)^{\frac{1}{p}}.
$

In order to estimate $III_i^2$, we observe that for $x\in C_j(B)$, $j\geq 2$, $i\geq j-1$, and $r_B/2^{i+1}\leq t<r_B$, we have that $B,B(x,2^{i+1}t)\subset 2^{i+3}B$. Thus, proceeding as before,
 \begin{multline*}
 III_i^2\lesssim 2^{\frac{in}{p_0}}
\left(\int_{\frac{r_B}{2^{i+1}}}^{r_B}\left(\frac{r_B}{t}\right)^{\frac{2n}{p_0}}
\left(\dashint_{2^{i+3}B}|e^{-\frac{t^2}{2}L}(f\chi_{2^{i+3}B})(y)|^{q}dw\right)^{\frac{2}{q}}\frac{dt}{t}\right)^{\frac{1}{2}}
\\
\quad\lesssim 
\left(\dashint_{B}|f(y)|^{p}dw\right)^{\frac{1}{p}}
\left(\int_{\frac{r_B}{2^{i+1}}}^{r_B}
\left(\frac{2^{i}r_B}{t}\right)^{2\theta_2+\frac{2n}{p_0}}
\frac{dt}{t}\right)^{\frac{1}{2}}
\quad\lesssim 2^{ic}
\left(\dashint_{B}|f(y)|^{p}dw\right)^{\frac{1}{p}}.
 \end{multline*}
Consequently, we conclude that
\begin{align}\label{term1-1}
I_1^1\lesssim e^{-c4^j}\left(\dashint_{B}|f(y)|^{p}dw\right)^{\frac{1}{p}}.
\end{align}  
Let us now estimate $\mathcal{I}_{k}$.
 We shall use extrapolation
to show that  $\mathcal{I}_{k}\lesssim e^{-c4^j}\left(\dashint_{B}|f(y)|^pdw\right)^{\frac{1}{p}}$ for all $k\in \N$. To this end,
 we first show 
that for every $w_0\in RH_{\left(\frac{p_+(L)}{2}\right)'}$ if $T_t=t^2Le^{-t^2L}$ (or $w_0\in RH_{\left(\frac{q_+(L)}{2}\right)'}$ if $T_t=t\nabla_ye^{-t^2L}$), and $k\in \N$,
\begin{multline}\label{extraimprov}
\int_{C_j(B)}\int^{r_B}_0\int_{B(x,t)}|T_te^{-kr_B^2L}f(y)|^2
\frac{dy\,dt}{t^{n+1}}w_0(x)dx
\\
\lesssim
\int_{2^{j+3}B\setminus 2^{j-2}B}\left(\sum_{i\geq 1}e^{-c4^i}\left(\int_{B(x,2^{i+1}r_B)}|e^{-\frac{kr_B^2}{2}L}f(y)|^{p_0}\frac{dy}{r_B^n}\right)^{\frac{1}{p_0}}\right)^2
w_0(x)dx.
\end{multline}
 Then, note that since $p\leq r_wp_-(L)<q<\frac{p_+(L)}{s_w}$ (or $p\leq r_wp_-(L)<q<\frac{q_+(L)}{s_w}$), we have that $w\in RH_{\left(\frac{p_+(L)}{p}\right)'}$ (or $w\in RH_{\left(\frac{q_+(L)}{p}\right)'}$). Hence, \eqref{extraimprov} and
Theorem \ref{thm:extrapol}, part ($a$), (or  part ($b$) if $q_+(L),p_+(L)=\infty$),  imply that for all $k\in \N$,
$$
w(2^{j+1}B)\mathcal{I}_k^p
\lesssim
\int_{2^{j+3}B\setminus 2^{j-2}B}\left(\sum_{i\geq 1}e^{-c4^i}\left(\int_{B(x,2^{i+1}r_B)}|e^{-\frac{kr_B^2}{2}L}f(y)|^{p_0}\frac{dy}{r_B^n}\right)^{\frac{1}{p_0}}
\right)^{p}w(x)dx=:II.
$$
Thus, once proved \eqref{extraimprov}, to estimate $\mathcal{I}_k$ we just need to consider $II$. Let us postpone the proof of \eqref{extraimprov} until later and continue with the estimate of $\mathcal{I}_k$. Since  $w\in A_{\frac{q}{p_0}}$, proceeding as in \eqref{toweight}, we have 
$$
II\lesssim \left(\sum_{i\geq 1}e^{-c4^i}2^{\frac{in}{p_0}}\left(
 \int_{2^{j+3}B\setminus 2^{j-2}B}
\left(\dashint_{B(x,2^{i+1}r_B)}|e^{-\frac{kr_B^2}{2}L}f(y)|^{q}{dw}\right)^{\frac{p}{q}}dw\right)^{\frac{1}{p}}\right)^p.
$$
For $2\leq j\leq 4$, note that if $x\in 2^{j+3}B\setminus 2^{j-2}B$ then $B,B(x,2^{i+1}r_B)\subset 2^{i+7}B$. Hence, using \eqref{doublingcondition} and the $L^{p}(w)-L^{q}(w)$ off-diagonal estimates on balls satisfied by $e^{-\frac{kr_B^2}{2}L}$, we get 
\begin{multline*}
II^{\frac{1}{p}}
\lesssim
\sum_{i\geq 1}e^{-c4^i}w(B)^{\frac{1}{p}}
\left(\dashint_{2^{i+7}B}|e^{-\frac{kr_B^2}{2}L}(f\chi_{2^{i+7}B})(y)|^{q}{dw}\right)^{\frac{1}{q}}
\\
\lesssim
\sum_{i\geq 1}e^{-c4^i}
\left(\int_{B}|f(y)|^{p}{dw}\right)^{\frac{1}{p}}
\lesssim
\left(\int_{B}|f(y)|^{p}{dw}\right)^{\frac{1}{p}},
\end{multline*}
And, for $j\geq 5$, we proceed as before, but noticing this time that for $x\in 2^{j+3}B\setminus 2^{j-2}B$, if $1\leq i\leq j-4$  then $B(x, 2^{i+1}r_B)\subset 2^{j+4}B\setminus 2^{j-3}B$; and if $i\geq j-3$ $B,B(x, 2^{i+1}r_B)\subset 2^{i+7}B$. Hence,
\begin{multline*}
II^{\frac{1}{p}}
\lesssim  2^{\frac{jn}{p_0}}w(2^{j+1}B)^{\frac{1}{p}}
\sum_{i= 1}^{j-4}e^{-c4^i}\sum_{l=j-3}^{j+3}
\left(\dashint_{C_l(B)}|e^{-\frac{kr_B^2}{2}L}f(y)|^{q}{dw}\right)^{\frac{1}{q}}
\\
+e^{-c4^j}w(2^{j+1}B)^{\frac{1}{p}}
\sum_{i\geq j-3}e^{-c4^i}
\left(\dashint_{2^{i+7}B}|e^{-\frac{kr_B^2}{2}L}(f\chi_{2^{i+7}B})(y)|^{q}{dw}\right)^{\frac{1}{q}}
\\
\lesssim e^{-c4^j}w(2^{j+1}B)^{\frac{1}{p}}
\left(\dashint_{B}|f(y)|^{p}{dw}\right)^{\frac{1}{p}}.
\end{multline*}
Let us next prove \eqref{extraimprov}. When $T_t=t^2Le^{-t^2L}$, for $w_0\in RH_{\left(\frac{p_+(L)}{2}\right)'}$, if $p_+(L)<\infty$, we chose $2<\widetilde{q}<p_+(L)$ so that $w_0\in RH_{\left(\frac{\widetilde{q}}{2}\right)'}$; if $p_+(L)=\infty$, the condition $w_0\in RH_{\left(\frac{p_+(L)}{2}\right)'}$ becomes $w_0\in A_{\infty}$. In this case, we take $\widetilde{q}/2>s_w$, consequently $w\in RH_{\left(\frac{\widetilde{q}}{2}\right)'}$ and $\widetilde{q}/2>1$. When $T_t=t\nabla_ye^{-t^2L}$, we do the same but replacing $p_+(L)$ with $q_+(L)$. Hence, by Proposition \ref{prop:q-extra} and applying the $L^{p_0}(\R^n)-L^{\widetilde{q}}(\R^n)$ off-diagonal estimates satisfied by
 $T_{\sqrt{t^2+kr_B^2/2}}$, we have that, for $\eta=2$ if $T_t=t\nabla_y e^{-t^2L}$ or $\eta=4$ if $T_t=t^2Le^{-t^2L}$,
\begin{align*}
&\int_{C_j(B)}\int^{r_B}_0\int_{B(x,t)}|T_te^{-kr_B^2L}f(y)|^2
\frac{dy\,dt}{t^{n+1}}w_0(x)dx
\\&\quad
\lesssim
\int^{r_B}_0\left(\frac{t}{r_B}\right)^{\eta}\int_{C_j(B)}\left(\int_{B(x,r_Bt/r_B)}|T_{\sqrt{t^2+kr_B^2/2}}e^{-\frac{kr_B^2}{2}L}f(y)|^{\widetilde{q}}\frac{dy}{t^n}\right)^{\frac{2}{\widetilde{q}}}w_0(x)dx
\frac{dt}{t}
\\&\quad
\lesssim
\int^{r_B}_0\left(\frac{t}{r_B}\right)^{\eta}\int_{2^{j+3}B\setminus 2^{j-2}B}\left(\int_{B(x,r_B)}|T_{\sqrt{t^2+kr_B^2/2}}e^{-\frac{kr_B^2}{2}L}f(y)|^{\widetilde{q}}\frac{dy}{r_B^n}\right)^{\frac{2}{\widetilde{q}}}w_0(x)dx
\frac{dt}{t}
\\&\quad
\lesssim
\int^{r_B}_0\left(\frac{t}{r_B}\right)^{\eta}\frac{dt}{t}\int_{2^{j+3}B\setminus 2^{j-2}B}\left(\sum_{i\geq 1}e^{-c4^i}\left(\int_{B(x,2^{i+1}r_B)}|e^{-\frac{kr_B^2}{2}L}f(y)|^{p_0}\frac{dy}{r_B^n}\right)^{\frac{1}{p_0}}\right)^2w_0(x)dx
\\&\quad
\lesssim
\int_{2^{j+3}B\setminus 2^{j-2}B}\left(\sum_{i\geq 1}e^{-c4^i}\left(\int_{B(x,2^{i+1}r_B)}|e^{-\frac{kr_B^2}{2}L}f(y)|^{p_0}\frac{dy}{r_B^n}\right)^{\frac{1}{p_0}}\right)^2w_0(x)dx.
\end{align*}
Therefore, we conclude that $\mathcal{I}_k\lesssim e^{-c4^j}
\left(\dashint_{B}|f(y)|^{p}{dw}\right)^{\frac{1}{p}}$, for all $k\in \N$. This, \eqref{originalterm}, \eqref{term2} with $2M> n/p_0+\frac{nq}{p_0}+\theta_1+\theta_2$, \eqref{term1}, and \eqref{term1-1}, allow us to conclude the proof. 
\end{proof}
\begin{remark}
Note that, from the previous result and \cite[Theorem 1.15]{MartellPrisuelos}, in the case that $w\in A_{\infty}$ satisfying $\mathcal{W}_w(p_-(L),p_+(L))\neq \emptyset$, we also improve the lower exponent of the range of $p'$s where the conical square function associated with the Poisson semigroup \eqref{square-P-1}-\eqref{square-P-3} are bounded on $L^p(w)$. With the exception that in the case of $\Grm_{\pp}$ and $\Gcal_{\pp}$, we need to assume further  that $\mathcal{W}_w(q_-(L),q_+(L))\neq \emptyset$.
\end{remark}


\section{Characterization of $H^p_{L}(w)$, $0<p< 1$ }\label{section:DR}

In this section we include the case $p=1$ in the statement of our results, because they are also true on that case, but it should be noticed that, as we said above, the case $p=1$ was already obtained in \cite{MartellPrisuelos:II}.

\begin{theorem}\label{thm:hardychartz}
Given $w\in A_{\infty}$ and $0<p\leq 1$, let $H^p_L(w)$ be the fixed molecular Hardy space as in Remark \ref{notation:H1w}. For every $q\in\mathcal{W}_w(p_-(L),p_+(L))$, $\varepsilon>0$, and $ M\in \N$ such that $M>\frac{n}{2}\left(\frac{r_w}{p}-\frac{1}{p_-(L)}\right)$, the following spaces are isomorphic to $H^p_L(w)$ (and therefore one another) with equivalent norms
$$
H^p_{L,q,\varepsilon,M}(w);
\qquad
H^p_{\Scal_{m,\hh},q}(w),\ m\in \N;
\qquad
H^p_{\Grm_{m,\hh},q}(w),\ m\in \N_0;
\quad \textrm{and}
\qquad
H^p_{\Gcal_{m,\hh},q}(w),\ m\in \N_0.
$$
In particular, each of the previous spaces does not depend (modulo isomorphisms) on the choice of the allowable parameters $q$, $\varepsilon$, $M$, and $m$. 
\end{theorem}

\begin{theorem}\label{thm:hardychartzPoisson}
Given $w\in A_{\infty}$ and $0<p\leq 1$, let $H^p_L(w)$ be the fixed molecular Hardy space as in Remark \ref{notation:H1w}. For every $q\in\mathcal{W}_w(p_-(L),p_+(L))$, the following spaces are isomorphic to $H^p_L(w)$ (and therefore  one another) with equivalent norms
$$
H^p_{\Scal_{K,\pp},q}(w),\ K\in \N;
\qquad
H^p_{\Grm_{K,\pp},q}(w),\ K\in \N_0;
\qquad\textrm{and}\qquad
H^p_{\Gcal_{K,\pp},q}(w),\ K\in \N_0.
$$
In particular, each of the previous spaces does not depend (modulo isomorphisms) on the choice of $q$, and $K$.
\end{theorem}

\begin{theorem}\label{thm:hardychartzNontangential}
Given $w\in A_{\infty}$ and $0<p\leq 1$, let $H^p_L(w)$ be the fixed molecular Hardy space as in Remark \ref{notation:H1w}. For every $q\in\mathcal{W}_w(p_-(L),p_+(L))$, the following spaces are isomorphic to $H^p_L(w)$ (and therefore one another) with equivalent norms
$$
H^p_{\Ncal_{\hh},q}(w)
\qquad\textrm{and}\qquad
H^p_{\Ncal_{\pp},q}(w).
$$
In particular, each of the previous spaces does not depend (modulo isomorphisms) on the choice of $q$.
\end{theorem}

The proofs of these theorems are analogous to those of the case $p=1$ (obtained in \cite{MartellPrisuelos:II}). Hence, we need similar results to \cite[Propositions 5.1, 6.1, and 7.31]{MartellPrisuelos:II}. These are the following.
\begin{proposition}\label{lema:SH-1}
Let $w\in A_{\infty}$, $q_1,q_2\in \mathcal{W}_w(p_-(L),p_+(L))$,  $0<p\leq 1$, $\varepsilon>0$,  and $M\in \N$ be
 such that $M>\frac{n}{2}\left(\frac{r_w}{p}-\frac{1}{p_-(L)}\right)$, there hold
\begin{list}{$(\theenumi)$}{\usecounter{enumi}\leftmargin=1cm \labelwidth=1cm\itemsep=0.2cm\topsep=.2cm \renewcommand{\theenumi}{\alph{enumi}}}

\item 
$
\mathbb{H}^p_{L,q_1,\varepsilon,M}(w)=\mathbb{H}^p_{\Scal_{m,\hh},q_1}(w)
$
with equivalent norms, for all $m\in \N$,;

\item the spaces
$
H^p_{\Scal_{m,\hh},q_1}(w)$ and $H_{\Scal_{m,\hh},q_2}^p(w)
$ are isomorphic, for all $m\in \N$,;

\item  
 $
\mathbb{H}_{L,q_1,\varepsilon,M}^p(w)= \mathbb{H}_{\Grm_{m,\hh},q_1}^p(w)= \mathbb{H}_{\Gcal_{m,\hh},q_1}^p(w),
$  with equivalent norms, for all $m\in \N_{0}$,.
\end{list}

\end{proposition}

\begin{proposition}\label{lemma:SKP}
Given $w\in A_{\infty}$, $q_1,q_2\in \mathcal{W}_w(p_-(L),p_+(L))$, $0<p\leq 1$ $K,\,M \in \N$
such that $M>\frac{n}{2}\left(\frac{r_w}{p}-\frac{1}{2}\right)$,  and
$\varepsilon_0=2M+2K+\frac{n}{2}-\frac{nr_w}{p}$, there hold

\begin{list}{$(\theenumi)$}{\usecounter{enumi}\leftmargin=1cm \labelwidth=1cm\itemsep=0.2cm\topsep=.2cm \renewcommand{\theenumi}{\alph{enumi}}}
\item $\mathbb{H}^p_{L,q_1,\varepsilon_0,M}(w)= \mathbb{H}^p_{\Scal_{K,\pp},q_1}(w),$
with equivalent norms;

\item the spaces
$
H^p_{\Scal_{K,\pp},q_1}(w)$ and $H_{\Scal_{K,\pp},q_2}^p(w)
$ are isomorphic;

\item  $
\mathbb{H}_{L,q_1,\varepsilon_0,M}^p(w)= \mathbb{H}_{\Grm_{K-1,\pp},q_1}^p(w)= \mathbb{H}_{\Gcal_{K-1,\pp},q_1}^p(w),
$  with equivalent norms.
\end{list}
\end{proposition}
The proofs of these propositions follow as \cite[Proposition 5.1 and Proposition 6.1]{MartellPrisuelos:II}, respectively, just doing the obvious changes. In particular, in order to prove Proposition \ref{lema:SH-1} we need to change, in \cite[Proposition 5.1]{MartellPrisuelos:II}, $1$ to $p$, $p$ and $q$ to $q_1$ and $q_2$, respectively, and replace the definition of $\lambda_l^j$ in \cite[(5.21)]{MartellPrisuelos:II} with $\lambda_l^j:=2^lw(Q_l^j)^{\frac{1}{p}}$. Proposition \ref{lemma:SKP} follows as \cite[Proposition 6.1]{MartellPrisuelos:II} doing the same changes as the ones indicated for the proof of Proposition \ref{lema:SH-1}, and also considering the $\varepsilon_0$ defined in the statement of Proposition \ref{lemma:SKP} instead of the one in the statement of \cite[Proposition 6.1]{MartellPrisuelos:II}.
From these proposition Theorems \ref{thm:hardychartz} and \ref{thm:hardychartzPoisson} follow at once.

Finally, we characterize the Hardy spaces associated with non-tangential maximal functions. We need  the following result. 

\begin{proposition}\label{lemma:hardy-N}
Let $w\in A_{\infty}$, $q\in \mathcal{W}_w(p_-(L),p_+(L))$, $0<p\leq 1$, $M\in \N$ such that $M>\frac{n}{2}\left(\frac{r_w}{p}-\frac{1}{2}\right)$, and $\varepsilon_0=2M+2+\frac{n}{2}-\frac{r_wn}{p}$, there hold
\begin{list}{$(\theenumi)$}{\usecounter{enumi}\leftmargin=1cm \labelwidth=1cm\itemsep=0.2cm\topsep=.2cm \renewcommand{\theenumi}{\alph{enumi}}}

\item[(a)]$
\mathbb{H}^p_{\Ncal_{\hh},q}(w)=\mathbb{H}^p_{\Scal_{\hh},q}(w)= \mathbb{H}^p_{L,q,\varepsilon_0, M}(w),
$
 with equivalent norms.

\item[(b)]
$
\mathbb{H}^p_{\Ncal_{\pp},q}(w)=\mathbb{H}^p_{\Gcal_{\pp},q}(w)= \mathbb{H}^p_{L,q,\varepsilon_0, M}(w),
$
 with equivalent norms.
\end{list}
\end{proposition}

This last proposition follows as Proposition \cite[Proposition 7.31]{MartellPrisuelos:II}, just replacing $1$ with $p$, and $p$ with $q$, and obviously using the $\varepsilon_0$ defined here.
Theorem \ref{thm:hardychartzNontangential} follows at once from Proposition \ref{lemma:hardy-N}.

\section{Characterization of $H^p_{\mathcal{T}}(w)$, $p\in \mathcal{W}_w(p_-(L),p_+(L))$ }\label{section:p>1}

In this section we prove Theorem \ref{theorem:hp=lp}. That is, for $\mathcal{T}$ being any square function in \eqref{square-H-1}--\eqref{square-P-3} or a non-tangential maximal function in \eqref{non-tangential},  we show that the Hardy spaces $H^p_{\mathcal{T}}(w)$ are isomorphic to the $L^p(w)$ spaces, for an appropriate range of $p$.
\subsection{Proof of Theorem \ref{theorem:hp=lp}}
For $w\in A_{\infty}$ and $p,q\in \mathcal{W}_w(p_-(L),p_+(L))$, 
we claim that $L^q(w)\cap L^p(w)=\mathbb{H}^p_{\mathcal{T},q}(w)$ with
\begin{align}
\|f\|_{\mathbb{H}^p_{\mathcal{T},q}(w)}\approx\|f\|_{L^p(w)},
\end{align}
where $\mathcal{T}$ is any function defined in \eqref{square-H-1}-\eqref{square-P-3} or \eqref{non-tangential}.
Then, taking the closure we would conclude the desired isomorphism
$$
H^p_{\mathcal{T},q}(w)\approx L^p(w), \,\textrm{for all}	\, p,q\in \mathcal{W}_w(p_-(L),p_+(L)),
$$
with constants independent of $q$,
so we can drop the dependence on $q$ and just write
$$
H^p_{\mathcal{T}}(w)\approx L^p(w), \,\textrm{for all}	\, p\in \mathcal{W}_w(p_-(L),p_+(L)).
$$
Let us prove our claim.  If $f\in L^p(w)\cap L^q(w)$, since $\|\mathcal{T}f\|_{L^p(w)}\lesssim \|f\|_{L^p(w)}<\infty$, (see \cite[Theorems 1.12 and 1.13]{MartellPrisuelos} and \cite[Proposition 7.1]{MartellPrisuelos:II}), then $f\in \mathbb{H}^p_{\mathcal{T},q}(w).$ 

In order to show the converse inclusion, let us first consider the particular case of $\mathcal{T}\equiv \Scal_{m,\pp}$, for $m\in \N$. Then, take $f\in \mathbb{H}_{\Scal_{m,\pp},q}^p(w)$, and 
consider the operator $\mathcal{Q}_L$ defined by
$$
\mathcal{Q}_{L}h(x,t):=\mathcal{T}_t^*h(x), \quad \textrm{for all}\quad (x,t)\in \R^{n+1}_+,
$$
where
 $\mathcal{T}_t^*:=(t^2L^*)^{m}e^{-t\sqrt{L^*}}$. This operator is bounded from 
 $L^{p'}(w^{1-p'})$ to $T^{p'}(w^{1-p'})$, for all $p'\in \mathcal{W}_{w^{1-p'}}(p_-(L^*),p_+(L^*))$.
Indeed, by \cite[Theorem 1.13]{MartellPrisuelos}, we have that,
for every $h\in L^{p'}(w^{1-p'})$,
\begin{align*}
\|\mathcal{Q}_Lh\|_{T^{p'}(w^{1-p'})}=
\left(\int_{\R^{n}}\left(\iint_{\Gamma(x)}|(t^2L^*)^me^{-t\sqrt{L^*}}h(y)|^2\frac{dy\,dt}{t^{n+1}}\right)^{\frac{p'}{2}}w^{1-p'}(x)dx\right)^{\frac{1}{p'}}
\lesssim \|h\|_{L^{p'}(w^{1-p'})}.
\end{align*}

Then, if $\mathcal{Q}_L^*$ denotes its adjoint operator with respect to $dx$, we have that for every
$H\in T^2(\R^n)$
\begin{align}\label{extension2}
\mathcal{Q}_L^*H(x)=\int_{0}^{\infty} (t^2L)^{m}e^{-t\sqrt{L}}
H(x,t)\frac{dt}{t}.
\end{align}
Similarly as in the proof of Theorem \ref{thm:interpolationhardy},  we conclude that $\mathcal{Q}_L^*$
has a bounded extension from $T^p(w)$ to $L^p(w)$, for all $p\in \mathcal{W}_w(p_-(L),p_+(L))$.
Next, since the  vertical square function defined by $(t^2L)^me^{-t\sqrt{L}}$ in bounded on $L^q(w)$ (see \cite[(6.3)]{MartellPrisuelos:II} and \cite{AuscherMartell:III}), we can  consider the following Calder\'on reproducing formula  of $f$ (see  \cite[Remark 5.21]{MartellPrisuelos:II}),
\begin{align}\label{calderonp>1}
f(x)=C_m\int_0^{\infty}\left((t^2L)^me^{-t\sqrt{L}}\right)^2f(x)\frac{dt}{t},
\end{align}
where the equality is in $L^q(w)$. 
Note that, as we explain in Remarks \ref{remark:extenssion} and \ref{remark:calderonreproducing}, and in the proof of Theorem \ref{thm:interpolationhardy}, for $q\in \mathcal{W}_w(p_-(L),p_+(L))$, we have 
\eqref{extension2} and also \eqref{calderonp>1}, for functions in $T^q(w)\cap T^p(w)$, and  in $L^q(w)$, respectively, understanding that, by
 abuse of notation, $L$ denotes $L_{q,w}$.

 Now, since for $f\in \mathbb{H}^p_{\Scal_{m,\pp},q}(w)$, we have that $f\in L^q(w)$ and  that $\widetilde{f}(x,t):=
(t^2L)^{m}e^{-t\sqrt{L}}
f(x)\in T^{p}(w)\cap T^q(w)$ ($\|\widetilde{f}\|_{T^p(w)}=\|\Scal_{m,\pp}f\|_{L^p(w)}$ and $\|\widetilde{f}\|_{T^q(w)}=\|\Scal_{m,\pp}f\|_{L^q(w)}\lesssim \|f\|_{L^q(w)}$),
we get, for every $g\in L^{p'}(w)\cap L^{q'}(w)$,
\begin{multline*}
\left|\int_{\R^n}f(y)\bar{g}(y)w(y)\,dy\right|
=C_m
\left|\int_{\R^n}\int_{0}^{\infty} 
\widetilde{f}(y,t)\overline{(t^2L^*)^{m}e^{-t\sqrt{L^*}}(gw)(y)}\frac{dt}{t}\,dy\right|
\\
\lesssim
\|\Scal_{m,\pp}f\|_{L^p(w)}\|\mathcal{Q}_L(\bar{g}w)\|_{T^{p'}(w^{1-p'})}
\lesssim
\|\Scal_{m,\pp}f\|_{L^p(w)}\|gw\|_{L^{p'}(w^{1-p'})}
=\|\Scal_{m,\pp}f\|_{L^p(w)}\|g\|_{L^{p'}({w})}.
\end{multline*}
Then, taking the supremum over all $g\in L^{p'}({w})\cap L^{q'}(w)$ such that $\|g\|_{L^{p'}(w)}= 1$ (note that $L^{q'}(w)\cap L^{p'}(w)$ is dense in $L^{p'}(w)$), we obtain that
\begin{align}\label{comparison:p-q}
\|f\|_{L^p(w)}\lesssim \|\Scal_{m,\pp}f\|_{L^p(w)}.
\end{align}
Therefore, we have that, for all $m\in \N$, $\mathbb{H}_{\Scal_{m,\pp},q}^p(w)=L^p(w)\cap L^q(w)$, with equivalent norms.

Now noticing that, in the proof of \cite[Theorem 1.15, part ($b$)]{MartellPrisuelos}, the authors showed that, for every $m\in \N$, $\|\Scal_{m,\pp}f\|_{L^p(w)}\leq \|\Scal_{m,\hh}f\|_{L^p(w)}$. From this, \eqref{comparison:p-q}, \cite[Theorems 1.14 and 1.15, Remark 4.22]{MartellPrisuelos},  and \cite[Lemma 4.4 and Proposition 7.1]{MartellPrisuelos:II}, we obtain that
$\mathbb{H}_{\mathcal{T},q}^p(w)=L^p(w)\cap L^q(w)$ with equivalent norms, for all $p,\,q\in \mathcal{W}_w(p_-(L),p_+(L))$ and $\mathcal{T}$ being any function in \eqref{square-H-1}-\eqref{square-P-3}, or a non-tangential maximal function in \eqref{non-tangential}. From the observations made at the beginning of the proof, this allows us to conclude the desired isomorphism.
\qed
\begin{remark}
As we explain in the proof we have obtained the isomorphism $H^p_{\mathcal{T},q}(w)\approx L^p(w)$ for all $p,q\in \mathcal{W}_w(p_-(L),p_+(L))$. In particular, this implies that
$$
H^p_{\mathcal{T},q_1}(w)\approx H^p_{\mathcal{T},q_2}(w), \, \textrm{ for all}\, p,q_1,q_2\in \mathcal{W}_w(p_-(L),p_+(L)).
$$
 for $\mathcal{T}$ being any function in \eqref{square-H-1}-\eqref{square-P-3}, or a non-tangential maximal function in \eqref{non-tangential}.
\end{remark}


%
%

\section{Characterization of the weighted Hardy space associated with the Riesz transform}\label{sec:Riesz}

In order to characterize the weighted Hardy space associated with the Riesz transform we proceed as in \cite{HofmannMayborodaMcIntosh}, where the unweighted case was consider.
First of all, we need to prove the following weighted versions of \cite[Propositions 5.32 and 5.34]{HofmannMayborodaMcIntosh} from which we obtain at once Theorem \ref{thm.Rieszcharacterization}.
\begin{proposition}\label{prop:1-Riesz}
Given $w\in A_{\infty}$ and $q\in\mathcal{W}_w(q_-(L),q_+(L))$, we have that, for all $p$ satisfying $\max\left\{r_w,\frac{nr_w\widehat{p}_-(L)}{nr_w+\widehat{p}_-(L)}\right\}<p<\frac{p_+(L)}{s_w}$ and $f\in \mathbb{H}^p_{\nabla L^{-1/2},q}(w)$, 
\begin{align}\label{square-riesz:1}
\|\Scal_{\hh}f\|_{L^p(w)}\lesssim \|\nabla L^{-\frac{1}{2}}f\|_{L^p(w)}.
\end{align}
In particular, we conclude that, for all $\max\left\{r_w,\frac{nr_w\widehat{p}_-(L)}{nr_w+\widehat{p}_-(L)}\right\}<p<\frac{p_+(L)}{s_w}$, $\mathbb{H}^p_{\nabla L^{-1/2},q}(w)\subset \mathbb{H}^p_{\Scal_{\hh},q}(w)$.
\end{proposition}
\begin{proposition}\label{prop:2-Riesz}
Given $w\in A_{\infty}$ and  $q\in\mathcal{W}_w(q_-(L),q_+(L))$, for all $0<p<\frac{q_+(L)}{s_w}$ and $f\in \mathbb{H}_{\Scal_{\hh},q}^p(w)$, we have that
\begin{align*}
\|\nabla L^{-\frac{1}{2}}f\|_{L^p(w)}\lesssim \|\Scal_{\hh}f\|_{L^p(w)}.
\end{align*}
In particular, we conclude that, for all $0<p<\frac{q_+(L)}{s_w}$, $\mathbb{H}^p_{\Scal_{\hh},q}(w)\subset \mathbb{H}^p_{\nabla L^{-1/2},q}(w)$.
\end{proposition}
We start by proving Proposition \ref{prop:1-Riesz}. To this end,
 consider the following conical square function:
$$
\widetilde{\Scal}f(x):=\left(\iint_{\Gamma(x)}|t\sqrt{L}e^{-t^2L}f(y)|^2\frac{dy\,dt}{t^{n+1}}\right)^{\frac{1}{2}}.
$$
We show that, in some range of $p$, its norm is comparable with the norm of $\Scal_{\hh}$ in $L^p(w)$.
\begin{proposition}\label{prop:widetildeS-heatS}
Given $w\in A_{\infty}$, for all $q\in \mathcal{W}_w(p_-(L),p_+(L))$ and $f\in L^q(w)$, there hold
\begin{list}{$(\theenumi)$}{\usecounter{enumi}\leftmargin=1cm \labelwidth=1cm\itemsep=0.2cm\topsep=.2cm \renewcommand{\theenumi}{\alph{enumi}}}

\item $\|\Scal_{\hh}f\|_{L^p(w)}\lesssim\|\widetilde{\Scal}f\|_{L^p(w)}$, for all $p\in \mathcal{W}_w(0,p_+(L)^{2,*})$;

\item $\|\widetilde{\Scal}f\|_{L^p(w)}\lesssim\|\Scal_{\hh}f\|_{L^p(w)}$, for all $p\in \mathcal{W}_w(0,p_+(L)^*)$.
\end{list}
In particular
\begin{align*}
\|\widetilde{\Scal}f\|_{L^p(w)}\approx\|\Scal_{\hh}f\|_{L^p(w)}, \textrm{ for all } p\in \mathcal{W}_w(0,p_+(L)^*).
\end{align*}
\end{proposition}
\begin{proof}
We first prove part ($a$). Note that since $2<p_+(L)^{2,*}$, in view of Theorem \ref{thm:extrapol}, part ($a$), (or  part ($b$) if $p_+(L)^{2,*}=\infty$), it is enough to prove it for $p=2$ and all $w\in RH_{\left(\frac{p_+(L)^{2,*}}{2}\right)'}$.

Assuming this, note that, proceeding as in the estimate of term $II$ when proving \eqref{S_HF(tent)}, given $w\in RH_{\left(\frac{p_+(L)^{2,*}}{2}\right)'}$,
we can find $q_0$ and $r$, so that $2<q_0<p_+(L)$, $q_0/2\leq r<\infty$, $w\in RH_{r'}$, and
\begin{align}\label{positive-III-1}
2+\frac{n}{2r}-\frac{n}{q_0}
>
 0.
\end{align}
After this observation we show the desired estimate.
By \eqref{representationsquarerootofL} and Minkowski's integral inequality, we obtain that
\begin{multline*}
\Scal_{\hh}f(x)
\lesssim
\left(\int_0^{\infty}\left(\int_{0}^{t}\left(\int_{B(x,t)}|sLe^{-s^2L}t^2\sqrt{L}e^{-t^2L}f(y)|^2dy\right)^{\frac{1}{2}}\frac{ds}{s}\right)^2\frac{dt}{t^{n+1}}\right)^{\frac{1}{2}}
\\
+
\left(\int_0^{\infty}\left(\int_{t}^{\infty}\left(\int_{B(x,t)}|sLe^{-s^2L}t^2\sqrt{L}e^{-t^2L}f(y)|^2dy\right)^{\frac{1}{2}}\frac{ds}{s}\right)^2\frac{dt}{t^{n+1}}\right)^{\frac{1}{2}}
=:I+II.
\end{multline*}
In the case that $s<t$, use the $L^2(\R^n)-L^2(\R^n)$ off-diagonal estimates satisfied by the families $\{t^2Le^{-t^2L}\}_{t>0}$ and  $\{e^{-s^2L}\}_{s>0}$, and apply \cite[Lemma 2.3]{HofmannMartell} to get
\begin{multline*}
I
\leq 
\left(\int_0^{\infty}\left(\int_{0}^{t}\frac{s}{t}\left(\int_{B(x,t)}|e^{-s^2L}t^2Le^{-\frac{t^2}{2}L}t\sqrt{L}e^{-\frac{t^2}{2}L}f(y)|^2dy\right)^{\frac{1}{2}}\frac{ds}{s}\right)^2\frac{dt}{t^{n+1}}\right)^{\frac{1}{2}}
\\
\lesssim
\sum_{j\geq 1}e^{-c4^{j}}\left(
\int_0^{\infty}\left(\int_{0}^{t}\frac{s}{t}\frac{ds}{s}\right)^{2}\int_{B(x,2^{j+1}t)}|t\sqrt{L}e^{-\frac{t^2}{2}L}f(y)|^2\frac{dy\,dt}{t^{n+1}}\right)^{\frac{1}{2}}.
\end{multline*}
Then, changing the variable $t$ into $\sqrt{2}t$ and
applying change of angles (\cite[Proposition 3.2]{MartellPrisuelos}), we conclude that
\begin{multline*}
\|I\|_{L^2(w)}
\lesssim \sum_{j\geq 1}e^{-c4^j}\left(\int_{\R^n}
\int_0^{\infty}\int_{B(x,2^{j+2}t)}|t\sqrt{L}e^{-t^2L}f(y)|^2\frac{dy\,dt}{t^{n+1}}w(x)dx\right)^{\frac{1}{2}}
\\
\lesssim
\sum_{j\geq 1}e^{-c4^j}\|\widetilde{\Scal}f\|_{L^2(w)}
\lesssim\|\widetilde{\Scal}f\|_{L^2(w)}.
\end{multline*}
As for the estimate of $II$, consider $\widetilde{f}(y,s):=s\sqrt{L}e^{-\frac{s^2}{2}L}f(y)$,  apply the $L^2(\R^n)-L^2(\R^n)$ off-diagonal estimates satisfied by the family $\{e^{-t^2L}\}_{t>0}$ and Jensen's inequality. Besides, change the  variable $s$ into $st$, apply Minkowski's integral inequality, and then change the variable $t$ into $t/s$. Hence, we have
\begin{align*}
&II
\lesssim
\sum_{j\geq 1}e^{-c4^j}
\left(\int_0^{\infty}\left(\int_{t}^{\infty}\frac{t^2}{s^2}\left(\int_{B(x,2^{j+1}t)}|s^2Le^{-\frac{s^2}{2}L}\widetilde{f}(y,s)|^2dy\right)^{\frac{1}{2}}\frac{ds}{s}\right)^2\frac{dt}{t^{n+1}}\right)^{\frac{1}{2}}
\\&
\lesssim
\sum_{j\geq 1}e^{-c4^j}
\left(\int_0^{\infty}\left(\int_{t}^{\infty}\frac{t^2}{s^2}\left(\int_{B(x,2^{j+1}t)}|s^2Le^{-\frac{s^2}{2}L}\widetilde{f}(y,s)|^{q_0}\frac{dy}{t^n}\right)^{\frac{1}{q_0}}\frac{ds}{s}\right)^2\frac{dt}{t}\right)^{\frac{1}{2}}
\\&
\lesssim
\sum_{j\geq 1}e^{-c4^j}\int_{1}^{\infty}s^{-2}
\left(\int_0^{\infty}\left(\int_{B(x,2^{j+1}t)}|(st)^2Le^{-\frac{(st)^2}{2}L}\widetilde{f}(y,st)|^{q_0}\frac{dy}{t^n}\right)^{\frac{2}{q_0}}\frac{dt}{t}\right)^{\frac{1}{2}}\frac{ds}{s}
\\&
\lesssim
\sum_{j\geq 1}e^{-c4^j}\int_{1}^{\infty}s^{-2+\frac{n}{q_0}}
\left(\int_0^{\infty}\left(\int_{B(x,2^{j+1}t/s)}|t^2Le^{-\frac{t^2}{2}L}\widetilde{f}(y,t)|^{q_0}\frac{dy}{t^n}\right)^{\frac{2}{q_0}}\frac{dt}{t}\right)^{\frac{1}{2}}\frac{ds}{s}
\\&
=:
\sum_{j\geq 1}e^{-c4^j}\int_{1}^{\infty}s^{-2+\frac{n}{q_0}}
\mathcal{J}(x,s)\frac{ds}{s}
.
\end{align*}
In order to estimate the norm in  $L^2(w)$ of the above integral, 
we first
apply Minkowski's inequality, \cite[Proposition 3.30]{MartellPrisuelos}, and change the variable $t$ into $\sqrt{2}t$. Next, we apply the $L^{2}(\R^n)-L^{q_0}(\R^n)$ off-diagonal estimates satisfied by the family $\{t^2Le^{-t^2L}\}_{t>0}$, and recall that $q_0$ and $r$ satisfy $2< q_0<p_+(L)$, $\frac{q_0}{2}\leq r$, $w\in RH_{r'}$, and \eqref{positive-III-1}. Finally, we apply \cite[Proposition 3.2]{MartellPrisuelos}. Thus, we have, for $r_0>r_w$,
\begin{align*}
&\left(\int_{\R^n}\left(\int_{1}^{\infty}s^{-2+\frac{n}{q_0}}\mathcal{J}(x,s)\frac{ds}{s}\right)^2w(x)dx\right)^{\frac{1}{2}}
\\
&\lesssim
\left(\int_1^{\infty}s^{-2-\frac{n}{2r}+\frac{n}{q_0}}\frac{ds}{s}\right)\left(\int_0^{\infty}\int_{\R^n}\left(\int_{B(x,2^{j+2}t)}|t^2Le^{-t^2L}t\sqrt{L}e^{-t^2L}f(y)|^{q_0}\frac{dy}{t^n}\right)^{\frac{2}{q_0}}w(x)dx\frac{dt}{t}\right)^{\frac{1}{2}}
\\
&
\lesssim \sum_{l\geq 1} e^{-c4^l}
\left(\int_{\R^n}\int_0^{\infty}\int_{B(x,2^{j+l+3}t)}|t\sqrt{L}e^{-t^2L}f(y)|^{2}\frac{dy\,dt}{t^{n+1}}w(x)dx\right)^{\frac{1}{2}}
\\
&\lesssim 2^{j\frac{nr_0}{2}}\sum_{l\geq 1} e^{-c4^l}\|\widetilde{\Scal}f\|_{L^2(w)}
\lesssim 2^{j\frac{nr_0}{2}}
\|\widetilde{\Scal}f\|_{L^2(w)}.
\end{align*}
Consequently, 
$
\|II\|_{L^2(w)}
\lesssim
\|\widetilde{\Scal}f\|_{L^2(w)},
$
which, together with the estimate obtained for $\|I\|_{L^2(w)}$, gives us the desired inequality.

\medskip
As for proving part ($b$), note that again it is enough to consider the case $p=2$ and $w\in RH_{\left(\frac{p_+(L)^*}{2}\right)'}$. In this case we proceed as in the proof of part $(a)$, so we skip some details. For $n>p_+(L)$, note that (as in the proof of \eqref{S_HF(tent)}) we can take $\varepsilon_0>0$ small enough and $2<q_0<p_+(L)$, close enough to $p_+(L)$
so that for $r:=\frac{q_0n}{2(1+\varepsilon_0)(n-q_0)}$, we have that
$2<q_0<p_+(L)$, $q_0/2\leq r<\infty$, $w\in RH_{r'}$, and
\begin{align*}
1+\frac{n}{2r}-\frac{n}{q_0}>0.
\end{align*}
If now  $n\leq p_+(L)$, our condition over the weight $w$ becomes $w\in A_{\infty}$. Then, we take $r>s_w$, and $q_0$ satisfying
$\max\left\{2,\frac{2rp_+(L)}{p_+(L)+2r}\right\}<q_0<\min\left\{p_+(L),2r\right\}$ if $p_+(L)<\infty$ and $q_0=2r$ if $p_+(L)=\infty$. Therefore, we have that $2<q_0<p_+(L)$, $q_0/2\leq r<\infty$, and $w\in RH_{r'}$. Besides, 
\begin{align*}
1+\frac{n}{2r}-\frac{n}{q_0}
>
1-\frac{n}{p_+(L)}\geq 0.
\end{align*}
Hence, we have found $q_0$ and $r$ so that $2<q_0<p_+(L)$, $q_0/2\leq r<\infty$, $w\in RH_{r'}$, and
\begin{align}\label{positive-III-2}
1+\frac{n}{2r}-\frac{n}{q_0}>0.
\end{align}
Keeping these choices of $q_0$ and $r$ we prove part ($b$).
Using again \eqref{representationsquarerootofL} and Minkowski's integral inequality, we obtain
\begin{multline*}
\widetilde{\Scal}f(x)\lesssim
\left(\int_0^{\infty}\left(\int_0^{t}\left(\int_{B(x,t)}|tsLe^{-s^2L}e^{-t^2L}f(y)|^2dy\right)^{\frac{1}{2}}\frac{ds}{s}\right)^{{2}}\frac{dt}{t^{n+1}}\right)^{\frac{1}{2}}
\\
+
\left(\int_0^{\infty}\left(\int_{t}^{\infty}\left(\int_{B(x,t)}|tsLe^{-s^2L}e^{-t^2L}f(y)|^2dy\right)^{\frac{1}{2}}\frac{ds}{s}\right)^{{2}}\frac{dt}{t^{n+1}}\right)^{\frac{1}{2}}
=:I+II.
\end{multline*}
We first estimate $I$. Using that $s<t$ and applying the $L^2(\R^n)-L^2(\R^n)$ off-diagonal estimates satisfied by the family $\{e^{-s^2L}\}_{s>0}$, we have
\begin{multline*}
I\leq
\left(\int_0^{\infty}\left(\int_0^{t}\frac{s}{t}\left(\int_{B(x,t)}|e^{-s^2L}t^2Le^{-t^2L}f(y)|^2dy\right)^{\frac{1}{2}}\frac{ds}{s}\right)^{{2}}\frac{dt}{t^{n+1}}\right)^{\frac{1}{2}}
\\
\lesssim
\sum_{j\geq 1}e^{-c4^j}\left(\int_0^{\infty}\left(\int_0^{t}\frac{s}{t}\frac{ds}{s}\right)^2\int_{B(x,2^{j+1}t)}|t^2Le^{-t^2L}f(y)|^2\frac{dy\,dt}{t^{n+1}}\right)^{\frac{1}{2}}
\\
\lesssim
\sum_{j\geq 1}e^{-c4^j}
\left(\int_0^{\infty}\int_{B(x,2^{j+1}t)}|t^2Le^{-t^2L}f(y)|^2\frac{dy\,dt}{t^{n+1}}\right)^{\frac{1}{2}}.
\end{multline*}
Therefore, applying change of angles (\cite[Proposition 3.2]{MartellPrisuelos}), we get
$$
\|I\|_{L^2(w)}
\lesssim
\sum_{j\geq 1}e^{-c4^j}\|\Scal_{\hh}f\|_{L^2(w)}\lesssim \|\Scal_{\hh}f\|_{L^2(w)}.
$$
As for the second term, we first apply the $L^2(\R^n)-L^2(\R^n)$ off-diagonal estimates satisfied by the family $\{e^{-t^2L}\}_{t>0}$, change the variable $s$ into $st$, and apply Jensen's inequality. Next, we apply Minkowski's integral inequality and change the variable $t$ into $t/s$. Hence, we have
\begin{align*}
II
&\lesssim
\sum_{j\geq 1}e^{-c4^j}
\left(\int_0^{\infty}\left(\int_{1}^{\infty}s^{-1}\left(
\int_{B(x,2^{j+1}t)}|(st)^2Le^{-(st)^2L}f(y)|^2dy\right)^{\frac{1}{2}}\frac{ds}{s}\right)^{{2}}\frac{dt}{t^{n+1}}\right)^{\frac{1}{2}}
\\
&\lesssim
\sum_{j\geq 1}e^{-c4^j}
\int_{1}^{\infty}s^{-1+\frac{n}{q_0}}\left(\int_0^{\infty}\left(\int_{B(x,2^{j+1}t/s)}|t^2Le^{-t^2L}f(y)|^{q_0}\frac{dy}{t^n}\right)^{\frac{2}{q_0}}\frac{dt}{t}\right)^{\frac{1}{2}}\frac{ds}{s}
\\&
=:
\sum_{j\geq 1}e^{-c4^j}
\int_{1}^{\infty}s^{-1+\frac{n}{q_0}}\mathcal{J}(x,s)\frac{ds}{s}.
\end{align*}
Thus, applying first Minkowski's integral inequality, \cite[Proposition 3.30]{MartellPrisuelos}, and changing the variable $t$ into $\sqrt{2}t$; next, applying the $L^2(\R^n)-L^{q_0}(\R^n)$ off-diagonal estimates satisfied by the family $\{e^{-t^2L}\}_{t>0}$,  recalling our choices of $q_0$ and $r$ and \eqref{positive-III-2}, and applying \cite[Proposition 3.2]{MartellPrisuelos}, we obtain, for $r_0>r_w$,
\begin{multline*}
\left(\int_{\R^n}\Bigg(\int_{1}^{\infty}s^{-1+\frac{n}{q_0}}\mathcal{J}(x,s)\frac{ds}{s}\Bigg)^2w(x)dx\right)^{\frac{1}{2}}
\\
\lesssim
\left(\int_{1}^{\infty}s^{-1-\frac{n}{2r}+\frac{n}{q_0}}\frac{ds}{s}\right)\left(\int_{\R^n}\int_0^{\infty}
\left(\int_{B(x,2^{j+2}t)}|e^{-t^2L}t^2Le^{-t^2L}f(y)|^{q_0}\frac{dy}{t^n}\right)^{\frac{2}{q_0}}\frac{dt}{t}w(x)dx\right)^{\frac{1}{2}}
\\
\lesssim  2^{j\frac{nr_0}{2}}\sum_{l\geq 1}e^{-c4^l}
\|\Scal_{\hh}f\|_{L^2(w)}
\lesssim  2^{j\frac{nr_0}{2}}
\|\Scal_{\hh}f\|_{L^2(w)}.
\end{multline*}
Using this, we obtain
$
\|II\|_{L^2(w)}
\lesssim\|\Scal_{\hh}f\|_{L^2(w)}.
$
Gathering this and the estimate obtained for $\|I\|_{L^2(w)}$ gives us that,
for all $w\in RH_{\left(\frac{p_+(L)^*}{2}\right)'}$, 
$$
\|\widetilde{\Scal}f\|_{L^2(w)}\lesssim\|\Scal_{\hh}f\|_{L^2(w)},
$$
which, from the observations made at the beginning, finishes the proof.
\end{proof}
\subsection{Proof of Proposition \ref{prop:1-Riesz}}
First of all note that if $f$ is such that $\|\nabla L^{-\frac{1}{2}}f\|_{L^{p}(w)}<\infty$, then for
 $h:=L^{-\frac{1}{2}}f$, we have that $h\in \dot{W}^{1,p}(w)$ (the space
 $\dot{W}^{1,p}(w)$ is defined as the completion of $
 \{h\in C_0^{\infty}(\R^n): \nabla h\in L^p(w)\}$ under the semi-norm $\|h\|_{\dot{W}^{1,p}(w)}:=\|\nabla h\|_{L^p(w)}$). Additionally, note that applying Proposition \ref{prop:widetildeS-heatS},  Theorem \ref{thm:SF-heat-III}, and \cite[Theorem 6.2]{AuscherMartell:III}, for all $w\in A_{\infty}$ such that $\mathcal{W}_{w}(p_-(L),p_+(L))\neq \emptyset$ and $\max\{r_w,\widehat{p}_-(L)\}<p<\frac{p_+(L)}{s_w}$, we have that
 $$
\|\widetilde{\Scal}\sqrt{L}h\|_{L^p(w)}\approx \|\Scal_{\hh}\sqrt{L}h\|_{L^p(w)}\lesssim \|\sqrt{L}h\|_{L^p(w)}\lesssim\|\nabla h\|_{L^p(w)}.
 $$
 This gives us that 
 \begin{align}\label{Sobolev}
 \widetilde{\Scal}\sqrt{L}:\dot{W}^{1,p}(w)\rightarrow L^p(w), \quad \forall\, \max\{r_w,\widehat{p}_-(L)\}<p<\frac{p_+(L)}{s_w}.
 \end{align}
 Therefore, if we show that, for every  $\widehat{p}_-(L)<\widetilde{p}<\frac{q_+(L)}{s_w}$, $r_0>r_w$, so that $r_wq_-(L)<r_0q_-(L)<\frac{q_+(L)}{s_w}$, and for $p_0:=\max\left\{r_0,\frac{nr_0\widetilde{p}}{nr_0+\widetilde{p}}\right\}$,
 \begin{align}\label{weakinterpolation}
 \widetilde{\Scal}\sqrt{L}:\dot{W}^{1,p_0}(w)\rightarrow L^{p_0,\infty}(w),
 \end{align}
then, by interpolation (see \cite{Bard}), applying Proposition \ref{prop:widetildeS-heatS}, and by the observation made at the beginning of the proof, we will conclude \eqref{square-riesz:1}. Besides, note that  $\mathcal{W}_w(q_-(L),q_+(L))\neq \emptyset$ implies   
$\mathcal{W}_w(p_-(L),p_+(L))\neq \emptyset$ (recall that $\mathcal{W}_w(q_-(L),q_+(L))\subset \mathcal{W}_w(p_-(L),p_+(L))$).

We fix  $\widetilde{p}$ and $r_0$ satisfying the above restrictions. Additionally, we take  $r$,  $q_-(L)<r<2$, close enough to $q_-(L)$ so that $rr_0<\frac{q_+(L)}{s_w}$. Then, if we consider $p_1$ so that $\max\{rr_0,\widetilde{p}\}<p_1<\frac{q_+(L)}{s_w}$, we have that  $w\in A_{\frac{p_1}{r}}\cap RH_{\left(\frac{q_+(L)}{p_1}\right)'}$, and  $p_1>p_0$.

Recalling these choices of $\widetilde{p}$, $r_0$, $r$, $p_1$, and $p_0$, note that in order to prove \eqref{weakinterpolation} it suffices to show that, for every $\alpha>0$ and $h\in \dot{W}^{1,p_0}(w)$,
$$
 w\left(\left\{x\in \R^{n}:\widetilde{\Scal}\sqrt{L}h(x)>\alpha\right\}\right)\lesssim \frac{1}{\alpha^{p_0}}\int_{\R^n}|\nabla h(x)|^{p_0}w(x)dx.
$$
To this end, consider the following Calder\'on-Zygmund decomposition of $h$ (see \cite[Lemma 6.6]{AuscherMartell:III}).
\begin{lemma}\label{C-Z-decomposition}
Let $n\geq 1$, $w\in A_{\infty}$, $\mu:=wdx$, and $r_w<p_0<\infty$ (with the possibility of taking $p_0=1$ if $r_w=1$). Assume that $h\in \dot{W}^{1,p_0}(w)$, and let $\alpha>0$. Then, one can find a collection of balls $\{B_i\}_{i\in \N}$ (with radii $r_{B_i}$), smooth functions $b_i$, and a function $g\in L^1_{loc}(w)$ such that
$$
h=g+\sum_{i\in \N}b_i
$$
and the following properties hold
\begin{align}\label{C-Z:g}
|\nabla g(x)|\leq C\alpha, \textrm{ for } \mu-	\textrm{a.e. } x,
\end{align}

\begin{align}\label{C-Z:b}
\supp b_i\subset B_i\quad \textrm{and}\quad \int_{B_i}|\nabla b_i(x)|^{p_0}w(x)dx\leq C\alpha^{p_0}w(B_i),
\end{align}

\begin{align}\label{C-Z:sum}
\sum_{i\in \N} w(B_i)\leq \frac{C}{\alpha^{p_0}}\int_{\R^n}|\nabla h(x)|^{p_0}w(x)dx,
\end{align}

\begin{align}\label{C-Z:sumoverlap}
\sum_{i\in \N}\chi_{B_i}\leq N,
\end{align}
where $C$ and $N$ depend only on the dimension, the doubling constant of $\mu$, and $p_0$. In addition, for $1\leq q<(p_0)_w^*$, where $(p_0)_w^*$ is defined in \eqref{q_wstar}, we have
\begin{align}\label{C-Z:extrab}
\left(\dashint_{B_i}|b_i(x)|^qdw\right)^{\frac{1}{q}}\lesssim \alpha r_{B_i}.
\end{align}
\end{lemma}
Applying this lemma to our function $h$ and to our choice of $p_0$, and considering for $M\in \N$, arbitrarily large, and for every $i\in \N$,   $B_{r_{B_i}}:=(I-e^{-r_{B_i}^2L})^M$ and $A_{r_{B_i}}:=I-B_{r_{B_i}}$, we can write 
$
b_i=B_{r_{B_i}}b_i+A_{r_{B_i}}b_i
$.
Hence,
$$
h= g+\sum_{i\in \N}B_{r_{B_i}}b_i+\sum_{i\in \N}A_{r_{B_i}}b_i.
$$ Then,
\begin{multline}\label{spplitingriesz}
w\left(\left\{x\in \R^n: \widetilde{\Scal}\sqrt{L}h(x)>\alpha\right\}\right)
\leq
w\left(\left\{x\in \R^n: \widetilde{\Scal}\sqrt{L}g(x)>\frac{\alpha}{3}\right\}\right)
\\
+
w\left(\left\{x\in \R^n: \widetilde{\Scal}\sqrt{L}\left(\sum_{i\in \N}A_{r_{B_i}}b_i\right)(x)>\frac{\alpha}{3}\right\}\right)
\\
+
w\left(\left\{x\in \R^n: \widetilde{\Scal}\sqrt{L}\left(\sum_{i\in \N}B_{r_{B_i}}b_i\right)(x)>\frac{\alpha}{3}\right\}\right)
=:I+II+III.
\end{multline}
By our choice of $p_1$, we have that $p_1\in \mathcal{W}_w(q_-(L),q_+(L))\subset \mathcal{W}_w(p_-(L),p_+(L))$. Then, applying Chebychev's inequality, \eqref{Sobolev}, \eqref{C-Z:g}, \eqref{C-Z:b}, and \eqref{C-Z:sum}, we obtain
\begin{multline}\label{Itermriesz}
I\lesssim \frac{1}{\alpha^{p_1}}\int_{\R^n}|\widetilde{\Scal}\sqrt{L}g(x)|^{p_1}w(x)dx
\lesssim \frac{1}{\alpha^{p_1}}
\int_{\R^n}|\nabla g(x)|^{p_1}w(x)dx
\\
\lesssim
\frac{1}{\alpha^{p_0}}\left(
\int_{\R^n}|\nabla h(x)|^{p_0}w(x)dx+\alpha^{p_0}\sum_{i\in \N}w(B_i)\right)
\lesssim
\frac{1}{\alpha^{p_0}}
\int_{\R^n}|\nabla h(x)|^{p_0}w(x)dx.
\end{multline}
In order to estimate the remaining terms, we take 
 $1< p<\infty$ and $u\in L^{p'}(w)$ such that $\|u\|_{L^{p'}(w)}= 1$. Besides, we denote by $\mathcal{M}^w$  the weighted maximal operator defined as in \eqref{weightedmaximal} but  taking the supremum over balls instead of  over cubes. Then,
using a Kolmogorov type inequality and \eqref{C-Z:sum},
we have that
\begin{multline}\label{maximal-u}
\left(\sum_{i\in \N} \int_{B_i}\left(\mathcal{M}^w(|u|^{p'})(x)\right)^{\frac{1}{p'}}w(x)dx\right)^{p}
\lesssim
\left(\int_{\cup_{i\in \N}B_i}\left(\mathcal{M}^w(|u|^{p'})(x)\right)^{\frac{1}{p'}}w(x)dx\right)^{p}
\\
\lesssim
w(\cup_{i\in \N}B_i)\|u\|_{L^{p'}(w)}^p\lesssim \frac{1}{\alpha^{p_0}}\int_{\R^n}|\nabla h(x)|^{p_0}w(x)dx.
\end{multline}
Moreover, note that for $p_2:=\max\{r_0,\widetilde{p}\}$, we have that $1< p_2<(p_0)_w^*$ and hence by \eqref{C-Z:extrab},
\begin{align}\label{CZ-ptildeexta}
\left(\dashint_{B_i}|b_i(x)|^{p_2}dw\right)^{\frac{1}{p_2}}\lesssim \alpha r_{B_i}.
\end{align}
Additionally, note that since $p_2\geq \widetilde{p}$, applying H\"older's inequality we also have the above inequality replacing $p_2$ with $\widetilde{p}$.

Let us next prove \eqref{CZ-ptildeexta}. To this end, first assume that $r_0>\widetilde{p}$. Then, $p_0=r_0$ and $p_2=r_0$, and it is easy to see that 
$1< r_0<(r_0)_w^*$. 

On the other hand, if $\widetilde{p}\geq r_0$, in other words, $p_2=\widetilde{p}>1$, we assume that $nr_w>p_0$, otherwise $(p_0)_w^*=\infty>\widetilde{p}=p_2$. Besides, note that
$nr_w>p_0$ implies $nr_wr_0>\widetilde{p}(r_0-r_w)$, consequently 
$$
(p_0)_w^*=\frac{nr_wp_0}{nr_w-p_0}\geq \frac{nr_wr_0\widetilde{p}}{nr_wr_0-\widetilde{p}(r_0-r_w)}>\widetilde{p}=p_2.
$$

\medskip

Therefore, in order to estimate $II$, we first apply Chebychev's inequality. Next,
by \eqref{Sobolev}, expanding the binomial,  using that $\{\sqrt{t}\nabla_y e^{-tL}\}_{t>0}$ satisfies $L^{p_2}(w)-L^{p_2}(w)$ off-diagonal estimates on balls (see \cite{AuscherMartell:II, AuscherMartell:III}), by \eqref{doublingcondition}, \eqref{CZ-ptildeexta}, and \eqref{maximal-u} with $p=p_2$, we have
\begin{align}\label{IItermriesz}
&II\lesssim\frac{1}{\alpha^{p_2}}\int_{\R^n} \left|\widetilde{\Scal}\sqrt{L}\left(\sum_{i\in \N}A_{r_{B_i}}b_i\right)(x)\right|^{p_2}w(x)dx
\\\nonumber
&\lesssim\frac{1}{\alpha^{p_2}}\int_{\R^n} \left|\nabla\left(\sum_{i\in \N}\sum_{k=1}^{M}{C_{k,M}}e^{-kr_{B_i}^2L}b_i\right)(x)\right|^{p_2}w(x)dx
\\\nonumber
&\lesssim\frac{1}{\alpha^{p_2}}\sup_{\|u\|_{L^{p_2'}(w)}=1}\left(\sum_{k=1}^{M}{C_{k,M}}\sum_{i\in \N}\int_{\R^n} \left|\sqrt{k}r_{B_i}\nabla_y e^{-kr_{B_i}^2L}\left(\frac{b_i}{r_{B_i}}\right)(x)\right|\,|u(x)|w(x)dx\right)^{p_2}
\\\nonumber
&\lesssim\!\!\frac{1}{\alpha^{p_2}}\sup_{\|u\|_{L^{p_2'}(w)}\!=1}\!\!
\left(\!\sum_{k=1}^{M}{C_{k,M}}\!\sum_{i\in \N}\!\sum_{j\geq 1}2^{\frac{jnp_1}{r}}w(B_i)\!\left(\dashint_{C_j(B_i)}\!\left|\!\!\sqrt{k}r_{B_i}\nabla_y e^{-kr_{B_i}^2L}\!\left(\frac{b_i}{r_{B_i}}\right)\!\!(x)\right|^{p_2}\!\!\!\!dw\!\right)^{\frac{1}{p_2}}\!\!\frac{\|u\chi_{C_j(B_i)}\|_{L^{p_2'}(w)}}{w(2^{j+1}B_i)^{\frac{1}{p_2'}}}\!\!\right)^{\!p_2}
\\\nonumber
&\lesssim\frac{1}{\alpha^{p_2}}\sup_{\|u\|_{L^{p_2'}(w)}=1}\left(\sum_{i\in \N}\sum_{j\geq 1}e^{-c4^j}w(B_i)\left(\dashint_{B_i} \left|\frac{b_i(x)}{r_{B_i}}\right|^{p_2}dw\right)^{\frac{1}{p_2}}\,\inf_{x\in B_i}\left(\mathcal{M}^w(|u|^{p_2'})(x)\right)^{\frac{1}{p_2'}}\right)^{p_2}
\\\nonumber
&\lesssim\sup_{\|u\|_{L^{p_2'}(w)}=1}\left(\sum_{i\in \N}\int_{B_i}\left(\mathcal{M}^w(|u|^{p_2'})(x)\right)^{\frac{1}{p_2'}}w(x)dx\right)^{p_2}
\\\nonumber
&\lesssim \frac{1}{\alpha^{p_0}}\int_{\R^n}|\nabla h(x)|^{p_0}w(x)dx.
\end{align}
Next, we estimate $III$.  Note that,
\begin{align}\label{III-Riesz}
III\lesssim 
w\left(\bigcup_{i\in \N}16B_i\right)
+
w\left(\left\{x\in \R^n\setminus \bigcup_{i\in \N}16B_i: \widetilde{\Scal}\sqrt{L}\left(\sum_{i\in \N}B_{r_{B_i}}b_i\right)(x)>\frac{\alpha}{3}\right\}\right)=III_1+III_2.
\end{align}
Applying \eqref{C-Z:sum} we have that 
\begin{align}\label{III1-Riesz}
III_1\lesssim \frac{1}{\alpha^{p_0}}\int_{\R^n}|\nabla h(x)|^{p_0}w(x)dx.
\end{align}
Hence it just remains to control $III_2$. Applying Chebychev's inequality, we obtain
\begin{align}\label{III2-Riesz}
 &III_2
 \\
&\nonumber
\lesssim
\frac{1}{\alpha^{p_1}}\int_{\R^n\setminus \bigcup_{i\in \N} 16B_i}\left(\int_0^{\infty}\int_{B(x,t)}\left|tLe^{-t^2L}\left(\sum_{i\in \N}B_{r_{B_i}}b_i\right)(y)\right|^{2}\frac{dy\,dt}{t^{n+1}}\right)^{\frac{p_1}{2}}w(x)dx
\\\nonumber
&
\lesssim\!
\frac{1}{\alpha^{p_1}}\!\!\sup_{\|u\|_{L^{p_1'}(w)}\!\!=1}\!\left(\sum_{i\in \N}\!\sum_{j\geq 4}\!\!\left(\!\int_{C_j(B_i)}\!\!\left(\int_0^{\infty}\!\!\!\!\int_{B(x,t)}\!\left|tLe^{-t^2L}\!\left(B_{r_{B_i}}b_i\right)\!(y)\right|^{2}\!\frac{dydt}{t^{n+1}}\!\right)^{\!\!\frac{p_1}{2}}\!\!\!\!\!w(x)dx\!\right)^{\!\!\frac{1}{p_1}}\!\!
\!\|u\chi_{C_j(B_i)}\|_{L^{p_1'}(w)}
\!\!\right)^{\!\!p_1} 
\\\nonumber&
=:\frac{1}{\alpha^{p_1}}\sup_{\|u\|_{L^{p_1'}(w)}=1}\left(\sum_{i\in \N}\sum_{j\geq 4}\,\mathcal{III}_{ij}\,\|u\chi_{C_j(B_i)}\|_{L^{p_1'}(w)}\right)^{p_1}.
\end{align}
Splitting the integral in $t$ (recall that $j\geq 4$), we have
\begin{multline*}
\mathcal{III}_{ij}
\lesssim
\left(\int_{C_j(B_i)}\left(\int_0^{2^{j-2}r_{B_i}}\int_{B(x,t)}\left|tLe^{-t^2L}\left(B_{r_{B_i}}b_i\right)(y)\right|^{2}\frac{dy\,dt}{t^{n+1}}\right)^{\frac{p_1}{2}}w(x)dx\right)^{\frac{1}{p_1}}
\\
+
\left(\int_{C_j(B_i)}\left(\int_{2^{j-2}r_{B_i}}^{\infty}\int_{B(x,t)}\left|t^2Le^{-t^2L}\left(B_{r_{B_i}}\left(\frac{b_i}{r_{B_i}}\right)\right)(y)\right|^{2}\frac{dy\,dt}{t^{n+1}}\right)^{\frac{p_1}{2}}w(x)dx\right)^{\frac{1}{p_1}}=
\mathcal{III}_{ij}^1+\mathcal{III}_{ij}^2.
\end{multline*}
We first estimate $\mathcal{III}_{ij}^1$. Recall that $w\in A_{\frac{p_1}{r}}\cap RH_{\left(\frac{q_+(L)}{p_1}\right)'}$. Hence,   we can take $q_0$, $\max\{2,p_1\}<q_0<q_+(L)$, close enough to $q_+(L)$ so that $w\in RH_{\left(\frac{q_0}{p_1}\right)'}$.  Then, applying Jensen's inequality, Fubini's theorem, and noticing that for $x\in C_j(B_i)$ and $0<t\leq 2^{j-2}r_{B_i}$  we have that $B(x,t)\subset 2^{j+2}B_i\setminus 2^{j-1}B_i$, we get
\begin{align*}
&\mathcal{III}_{ij}^1\lesssim
|2^{j+1}B_i|^{-\frac{1}{q_0}}w(2^{j+1}B_i)^{\frac{1}{p_1}}
\left(\int_{C_j(B_i)}\left(\int_0^{2^{j-2}r_{B_i}}\int_{B(x,t)}\left|tLe^{-t^2L}\left(B_{r_{B_i}}b_i\right)(y)\right|^{2}\frac{dy\,dt}{t^{n+1}}\right)^{\frac{q_0}{2}}dx\right)^{\frac{1}{q_0}}
\\
&\lesssim
|2^{j+1}B_i|^{-\frac{1}{q_0}}w(2^{j+1}B_i)^{\frac{1}{p_1}}
\left(\int_{C_j(B_i)}\int_0^{2^{j-2}r_{B_i}}\left(\frac{2^jr_{B_i}}{t}\right)^{\frac{q_0}{2}-1}\!\!\int_{B(x,t)}\left|tLe^{-t^2L}\left(B_{r_{B_i}}b_i\right)(y)\right|^{q_0}\frac{dy\,dt}{t^{n+1}}dx\right)^{\frac{1}{q_0}}
\\
&\lesssim
|2^{j+1}B_i|^{-\frac{1}{q_0}}w(2^{j+1}B_i)^{\frac{1}{p_1}}
\left(\int_0^{2^{j-2}r_{B_i}}\left(\frac{2^jr_{B_i}}{t}\right)^{\frac{q_0}{2}-1}\!\!\!t^{-q_0}\!\!\int_{2^{j+2}B_i\setminus 2^{j-1}B_i}\left|t^2Le^{-t^2L}\!\left(B_{r_{B_i}}b_i\right)\!(y)\right|^{q_0}\frac{dy\,dt}{t}\right)^{\!\frac{1}{q_0}}\!.
\end{align*}
We estimate the integral in $y$ by using functional calculus. We use the notation in \cite{Auscher} and \cite[Section 7]{AuscherMartell:III}.
We write $\vartheta\in[0,\pi/2)$ for the supremum of $|{\rm arg}(\langle Lf,f\rangle_{L^2(\R^n)})|$ over all $f$ in the domain of $L$.
Let $0<\vartheta <\theta<\nu<\mu<\pi /2$ and note that, for a fixed $t>0$, $\phi(z,t):=e^{-t^2 z}(1-e^{-r_{B_i}^2 z})^M$ is holomorphic in the open sector $\Sigma_\mu=\{z\in\mathbb{C}\setminus\{0\}:|{\rm arg} (z)|<\mu\}$ and satisfies $|\phi(z,t)|\lesssim |z|^M\,(1+|z|)^{-2M}$ (with implicit constant depending on $\mu$, $t>0$, $r_{B_i}$, and $M$) for every $z\in\Sigma_\mu$. Hence, we can write
\[
\phi(L,t )=\int_{\Gamma } e^{-zL }\eta (z,t)dz,
\qquad \text{where} \quad 
\eta(z,t) = \int_{\gamma} e^{\zeta z} \phi(\zeta,t ) d\zeta.
\]
Here $\Gamma=\partial\Sigma_{\frac\pi2-\theta}$ with positive orientation (although orientation is irrelevant for our computations) and 
$\gamma=\R_+e^{i\,{\rm sign}({\rm Im} (z))\,\nu}$. It is not difficult to see that for every $z\in \Gamma$,
$$
|\eta(z,t)| \lesssim \frac{r_{B_i}^{2M}}{(|z|+t^2)^{M+1}}.
$$
Consequently, we can write
\begin{align*}
\left(\int_{2^{j+2}B_i\setminus 2^{j-1}B_i}\right.&\left.\left|t^2Le^{-t^2L}B_{r_{B_i}}\left(b_i\right)(y)\right|^{q_0}dy\right)^{\frac{1}{q_0}}
\\
&
\lesssim
\int_{\Gamma}\left(\int_{2^{j+2}B_i\setminus 2^{j-1}B_i}\left|\frac{z}{2}Le^{-\frac{z}{2}L}\left(e^{-\frac{z}{2}L}b_i\right)(y)\right|^{q_0}dy\right)^{\frac{1}{q_0}}
\frac{t^2}{|z|}\frac{r_{B_i}^{2M}}{(|z|+t^2)^{M+1}}|dz|
\\
&\lesssim 
\sum_{l= 1}^{j-3}\int_{\Gamma}\left(\int_{2^{j+2}B_i\setminus 2^{j-1}B_i}\left|\frac{z}{2}Le^{-\frac{z}{2}L}\left(\chi_{C_l(B_i)}e^{-\frac{z}{2}L}b_i\right)(y)\right|^{q_0}dy\right)^{\frac{1}{q_0}}
\frac{t^2}{|z|}\frac{r_{B_i}^{2M}}{(|z|+t^2)^{M+1}}|dz|
\\
&\quad+
\sum_{l\geq j-2}\int_{\Gamma}\left(\int_{2^{j+2}B_i\setminus 2^{j-1}B_i}\left|\frac{z}{2}Le^{-\frac{z}{2}L}\left(\chi_{C_l(B_i)}e^{-\frac{z}{2}L}b_i\right)(y)\right|^{q_0}dy\right)^{\frac{1}{q_0}}
\frac{t^2}{|z|}\frac{r_{B_i}^{2M}}{(|z|+t^2)^{M+1}}|dz|.
\end{align*} 
Note now that since
$j\geq 4$, for $1\leq l\leq j-3$ we have that $d(2^{j+2}B_i\setminus 2^{j-1}B_i,C_l(B_i))\geq 2^{j-2}r_{B_i}\geq 2^{l+1}r_{B_i}$.
Then, in that case, applying the fact that $\frac{z}{2}Le^{-\frac{z}{2}L}$ satisfies $L^r(\R^n)- L^{q_0}(\R^n)$ off-diagonal estimates (see \cite{Auscher}), splitting the exponential term, using that $w\in A_{\frac{p_1}{r}}$,  changing the variable $s$ into $4^jr_{B_i}^2/s^2$, and applying that $e^{-\frac{z}{2}L}$ satisfies  $L^{\widetilde{p}}(w)-L^{p_1}(w)$ off-diagonal estimates on balls (see \cite{AuscherMartell:II, AuscherMartell:III}), and  by \eqref{CZ-ptildeexta}, we obtain
\begin{align*}
&
\int_{\Gamma}\left(\int_{2^{j+2}B_i\setminus 2^{j-1}B_i}\left|\frac{z}{2}Le^{-\frac{z}{2}L}\left(\chi_{C_l(B_i)}e^{-\frac{z}{2}L}b_i\right)(y)\right|^{q_0}dy\right)^{\frac{1}{q_0}}
\frac{t^2}{|z|}\frac{r_{B_i}^{2M}}{(|z|+t^2)^{M+1}}|dz|
\\
&\quad\quad\lesssim
\int_{\Gamma}\left(\int_{C_l(B_i)}\left|e^{-\frac{z}{2}L}b_i(y)\right|^{r}dy\right)^{\frac{1}{r}}|z|^{-\frac{n}{2}\left(\frac{1}{r}-\frac{1}{q_0}\right)}e^{-c\frac{4^jr_{B_i}^2}{|z|}}
\frac{t^2}{|z|}\frac{r_{B_i}^{2M}}{(|z|+t^2)^{M+1}}|dz|
\\
&\quad\quad\lesssim (2^{l}r_{B_i})^{\frac{n}{r}}
\int_{\Gamma}\left(\dashint_{C_l(B_i)}\left|e^{-\frac{z}{2}L}b_i(y)\right|^{p_1}dw\right)^{\frac{1}{p_1}}|z|^{-\frac{n}{2}\left(\frac{1}{r}-\frac{1}{q_0}\right)}e^{-c\frac{4^j r_{B_i}^2}{|z|}}e^{-c\frac{4^l r_{B_i}^2}{|z|}}
\frac{t^2}{|z|}\frac{ r_{B_i}^{2M}}{(|z|+t^2)^{M+1}}|dz|
\\
&\quad\quad
\lesssim 2^{l\theta_1}
(2^{l} r_{B_i})^{\frac{n}{r}}\left(\dashint_{B_i}\left|b_i(y)\right|^{\widetilde{p}}dw\right)^{\frac{1}{\widetilde{p}}}\int_{0}^{\infty}\Upsilon\left(\frac{2^{l} r_{B_i}}{s^{\frac{1}{2}}}\right)^{\theta_2}s^{-\frac{n}{2}\left(\frac{1}{r}-\frac{1}{q_0}\right)}e^{-c\frac{4^j r_{B_i}^2}{s}}e^{-c\frac{4^l r_{B_i}^2}{s}}
t^2\frac{ r_{B_i}^{2M}}{(s+t^2)^{M+1}}\frac{ds}{s}
\\
&\quad\quad
\lesssim \alpha  r_{B_i}2^{l\left(\theta_1+\frac{n}{r}\right)}
 r_{B_i}^{\frac{n}{q_0}}2^{-jn\left(\frac{1}{r}-\frac{1}{q_0}\right)}\int_{0}^{\infty}\Upsilon\left(\frac{2^{l}s}{2^{j}}\right)^{\theta_2}s^{n\left(\frac{1}{r}-\frac{1}{q_0}\right)}e^{-cs^2}e^{-c\frac{4^ls^2}{4^j}}
t^2\frac{ r_{B_i}^{2M}}{(4^{j} r_{B_i}^2/s^2+t^2)^{M+1}}\frac{ds}{s},
\end{align*}
recall that $\Upsilon(u)=\max\{u,u^{-1}\}$.

If we now consider $l\geq j-2$, in this case, we do not have distance between $2^{j+2}B_i\setminus 2^{j-1}B_i$ and $C_l(B_i)$, but we do have between $C_l(B_i)$ and $B_i$. Indeed, since $l\geq j-2\geq 2$, we have that $d(C_l(B_i),B_i)> 2^{l-1} r_{B_i}\geq 2^{j-3} r_{B_i}$. Hence, proceeding as in the above computation, we obtain
\begin{align*}
&
\int_{\Gamma}\left(\int_{2^{j+2}B_i\setminus 2^{j-1}B_i}\left|\frac{z}{2}Le^{-\frac{z}{2}L}\left(\chi_{C_l(B_i)}e^{-\frac{z}{2}L}b_i\right)(y)\right|^{q_0}dy\right)^{\frac{1}{q_0}}
\frac{t^2}{|z|}\frac{ r_{B_i}^{2M}}{(|z|+t^2)^{M+1}}|dz|
\\
&\quad\quad\lesssim
\int_{\Gamma}\left(\int_{C_l(B_i)}\left|e^{-\frac{z}{2}L}b_i(y)\right|^{r}dy\right)^{\frac{1}{r}}|z|^{-\frac{n}{2}\left(\frac{1}{r}-\frac{1}{q_0}\right)}
\frac{t^2}{|z|}\frac{ r_{B_i}^{2M}}{(|z|+t^2)^{M+1}}|dz|
\\
&\quad\quad\lesssim (2^{l} r_{B_i})^{\frac{n}{r}}
\int_{\Gamma}\left(\dashint_{C_l(B_i)}\left|e^{-\frac{z}{2}L}b_i(y)\right|^{p_1}dw\right)^{\frac{1}{p_1}}|z|^{-\frac{n}{2}\left(\frac{1}{r}-\frac{1}{q_0}\right)}
\frac{t^2}{|z|}\frac{ r_{B_i}^{2M}}{(|z|+t^2)^{M+1}}|dz|
\\
&\quad\quad
\lesssim 2^{l\theta_1}
(2^{l} r_{B_i})^{\frac{n}{r}}\left(\dashint_{B_i}\left|b_i(y)\right|^{\widetilde{p}}dw\right)^{\frac{1}{\widetilde{p}}}\int_{0}^{\infty}\Upsilon\left(\frac{2^{l} r_{B_i}}{s^{\frac{1}{2}}}\right)^{\theta_2}s^{-\frac{n}{2}\left(\frac{1}{r}-\frac{1}{q_0}\right)}e^{-c\frac{4^l r_{B_i}^2}{s}}
t^2\frac{ r_{B_i}^{2M}}{(s+t^2)^{M+1}}\frac{ds}{s}
\\
&\quad\quad
\lesssim \alpha r_{B_i}2^{l\theta_1}
(2^{l} r_{B_i})^{\frac{n}{r}}\int_{0}^{\infty}\Upsilon\left(\frac{2^{l} r_{B_i}}{s^{\frac{1}{2}}}\right)^{\theta_2}s^{-\frac{n}{2}\left(\frac{1}{r}-\frac{1}{q_0}\right)}e^{-c\frac{4^j r_{B_i}^2}{s}}e^{-c\frac{4^l r_{B_i}^2}{s}}
t^2\frac{ r_{B_i}^{2M}}{(s+t^2)^{M+1}}\frac{ds}{s}
\\
&\quad\quad
\lesssim \alpha  r_{B_i}2^{l\left(\theta_1+\frac{n}{r}\right)}
 r_{B_i}^{\frac{n}{q_0}}2^{-jn\left(\frac{1}{r}-\frac{1}{q_0}\right)}\int_{0}^{\infty}\Upsilon\left(\frac{2^{l}s}{2^{j}}\right)^{\theta_2}s^{n\left(\frac{1}{r}-\frac{1}{q_0}\right)}e^{-cs^2}e^{-c\frac{4^ls^2}{4^j}}
t^2\frac{ r_{B_i}^{2M}}{(4^{j} r_{B_i}^2/s^2+t^2)^{M+1}}\frac{ds}{s}.
\end{align*}
Next, changing the variable $t$ into $2^j r_{B_i}t$, we have for  $\widetilde{M}>0$ large enough to be chosen later,
\begin{align*}
&
\left(\int_0^{ 2^{j-2}r_{B_i}}\!\!\!
\left(\frac{ 2^jr_{B_i}}{t}\right)^{\frac{q_0}{2}-1}\!\!\!\!t^{-q_0}
\!\left(\int_{0}^{\infty}\Upsilon\left(\frac{2^{l}s}{2^{j}}\right)^{\theta_2}s^{n\left(\frac{1}{r}-\frac{1}{q_0}\right)}e^{-cs^2}e^{-c\frac{4^ls^2}{4^j}}
t^2\frac{ r_{B_i}^{2M}}{(4^{j} r_{B_i}^2/s^2+t^2)^{M+1}}\frac{ds}{s}\right)^{q_0}\frac{dt}{t}\right)^{\frac{1}{q_0}}
\\
&
\lesssim 2^{-j\left(2M+1\right)} r_{B_i}^{-1}
\left(\int_0^{1}\!\!\!
t^{1+\frac{q_0}{2}}\left(\int_{0}^{\infty}\Upsilon\left(\frac{2^{l}s}{2^{j}}\right)^{\theta_2}s^{n\left(\frac{1}{r}-\frac{1}{q_0}\right)}e^{-cs^2}e^{-c\frac{4^ls^2}{4^j}}
\frac{1}{(1/s^2+t^2)^{M+1}}\frac{ds}{s}\right)^{q_0}\frac{dt}{t}\right)^{\frac{1}{q_0}}
\\
&
\lesssim 2^{-l\left(2\widetilde{M}-\theta_2\right)}2^{-j\left(2M+1-\theta_2-2\widetilde{M}\right)} r_{B_i}^{-1}
\left(\left(\int_0^{1}\!\!\!
t^{1+\frac{q_0}{2}}\left(\int_{0}^{1}s^{n\left(\frac{1}{r}-\frac{1}{q_0}\right)-2\widetilde{M}+2M+2-\theta_2}
\frac{ds}{s}\right)^{q_0}\frac{dt}{t}\right)^{\frac{1}{q_0}}\right.
\\
&
\qquad\qquad\qquad\qquad\qquad\qquad\qquad\left.+
\left(\int_0^{1}\!\!\!
t^{1+\frac{q_0}{2}}\left(\int_{1}^{\infty}s^{n\left(\frac{1}{r}-\frac{1}{q_0}\right)-2\widetilde{M}+2M+2+\theta_2}e^{-cs^2}
\frac{ds}{s}\right)^{q_0}\frac{dt}{t}\right)^{\frac{1}{q_0}}\right).
\end{align*}
Therefore, taking $2\widetilde{M}=\theta_1+\theta_2+\frac{n}{r}+1$, and $2M>2\theta_2+\theta_1+\frac{n}{r}$, we have
\begin{align}\label{termcalIII}
\mathcal{III}_{ij}^1
\lesssim \alpha w(2^{j+1}B_i)^{\frac{1}{p_1}}2^{-j\left(2M+\frac{n}{r}+1-\theta_2-2\widetilde{M}\right)}\sum_{l\geq 1}2^{-l}
\lesssim \alpha w(2^{j+1}B_i)^{\frac{1}{p_1}}
2^{-j\left(2M-\theta_1-2\theta_2\right)}.
\end{align}
In order to estimate $\mathcal{III}_{ij}^2$, we consider $\theta_M:=\sqrt{M+2}$ and $B_{r_{B_i},t}:=(e^{-t^2L}-e^{-(t^2+r_{B_i}^2)L})^M$. Hence, applying the fact that $\{t^2Le^{-t^2L}\}_{t>0}\in \mathcal{F}(L^{r}-L^2)$, Proposition \ref{prop:lebesgueoff-dBQ} with $s=r_{B_i}$ and $p=r$, \cite[Lemma 2.1]{MartellPrisuelos}, and next using that $w\in A_{\frac{p_1}{r}}$, applying that $\{e^{-tL}\}_{t>0}$ satisfies $L^{\widetilde{p}}(w)-L^{p_1}(w)$ off-diagonal estimates on balls, and by \eqref{CZ-ptildeexta}, we obtain
\begin{align*}
\left(\int_{B(x,\theta_Mt)}\right.&\left.\left|t^2Le^{-t^2L}B_{r_{B_i},t}\left(e^{-t^2L}\left(\chi_{B(x,9\theta_{M}t)}\frac{b_i}{ r_{B_i}}\right)\right)(y)\right|^{2}\frac{dy}{t^{n}}\right)^{\frac{1}{2}}
\\&
\lesssim \left(\frac{ r_{B_i}^2}{t^2}\right)^{M}
\sum_{l\geq 1}
\left(\int_{C_{l}(B(x,9\theta_{M}t))}\left|e^{-t^2L}\left(\chi_{B(x,9\theta_{M}t)}\frac{b_i}{ r_{B_i}}\right)(y)\right|^{r}\frac{dy}{t^{n}}\right)^{\frac{1}{r}}
\\&
\lesssim \left(\frac{ r_{B_i}^2}{t^2}\right)^{M}
\sum_{l\geq 1}2^{\frac{ln}{r}}
\left(\dashint_{C_{l}(B(x,9\theta_{M}t))}\left|e^{-t^2L}\left(\chi_{B(x,9\theta_{M}t)}\frac{b_i}{ r_{B_i}}\right)(y)\right|^{p_1}dw\right)^{\frac{1}{p_1}}
\\&
\lesssim \left(\frac{ r_{B_i}^2}{t^2}\right)^{M}
\sum_{l\geq 1}e^{-c4^l}
\left(\dashint_{B(x,9\theta_M t)}\left|\frac{b_i(y)}{ r_{B_i}}\right|^{\widetilde{p}}dw\right)^{\frac{1}{\widetilde{p}}}
\\
&
\lesssim \alpha\left(\frac{ r_{B_i}^2}{t^2}\right)^{M}\left(\frac{w(B_i)}{w(B(x,9\theta_M t))}\right)^{\frac{1}{\widetilde{p}}}.
\end{align*}
Therefore, changing the variable $t$ into $t\theta_M$ and noticing that,  for $x\in C_{j}(B_i)$ and $t>\frac{2^{j-2} r_{B_i}}{\theta_M}$, we have that $B_i\subset B(x,9\theta_{M}t)$, using the estimate above, we get
\begin{multline*}
\mathcal{III}_{ij}^2
\lesssim\left(\int_{C_j(B_i)}
\left(\int_{\frac{2^{j-2} r_{B_i}}{\theta_M}}^{\infty}\int_{B(x,\theta_Mt)}\left|t^2Le^{-t^2L}B_{r_{B_i},t}\left(e^{-t^2L}\left(\frac{b_i}{ r_{B_i}}\right)\right)(y)\right|^{2}\frac{dy\,dt}{t^{n+1}}\right)^{\frac{p_1}{2}}w(x)dx\right)^{\frac{1}{p_1}}
\\
\lesssim \alpha
w(2^{j+1}B_i)^{\frac{1}{p_1}}
\left(\int_{\frac{2^{j-2} r_{B_i}}{\theta_M}}^{\infty}\left(\frac{ r_{B_i}^2}{t^2}\right)^{2M}
\frac{dt}{t}\right)^{\frac{1}{2}}
\lesssim \alpha w(2^{j+1}B_i)^{\frac{1}{p_1}}
2^{-j2M}.
\end{multline*}
This  and \eqref{termcalIII} imply that $\mathcal{III}_{ij}\lesssim \alpha w(2^{j+1}B_i)^{\frac{1}{p_1}}2^{-j(2M-\theta_1-2\theta_2)}$.
Therefore, in view of \eqref{III2-Riesz}, and by \eqref{doublingcondition} and \eqref{maximal-u} with $p=p_1$, taking $2M>\frac{np_1}{r}+\theta_1+2\theta_2$, we obtain that
\begin{multline*}
III_2\lesssim \sup_{\|u\|_{L^{p_1'}(w)}=1}\left(\sum_{i\in \N}w(B_i)\inf_{x\in B_i}\left(\mathcal{M}^w(|u|^{p_1'})(x)\right)^{\frac{1}{p_1'}}\sum_{j\geq 4}2^{-j\left(2M-\frac{np_1}{r}-\theta_1-2\theta_2\right)}\right)^{p_1}
\\
\lesssim \sup_{\|u\|_{L^{p_1'}(w)}=1}\left(\sum_{i\in \N}\int_{B_i}\left(\mathcal{M}^w(|u|^{p_1'})(x)\right)^{\frac{1}{p_1'}}w(x)dx\right)^{p_1}\lesssim \frac{1}{\alpha^{p_0}}\int_{\R^n}|\nabla h(x)|^{p_0}w(x)dx.
\end{multline*}
Plugging this and \eqref{III1-Riesz} into \eqref{III-Riesz} gives us $III\lesssim \alpha^{-p_0}\int_{\R^n}|\nabla h(x)|^{p_0}w(x)dx$. Hence, by this, \eqref{IItermriesz}, \eqref{Itermriesz}, and \eqref{spplitingriesz}, we conclude
\eqref{weakinterpolation}.

To complete the proof note that for $\max\left\{r_w,\frac{nr_w\widehat{p}_-(L)}{nr_w+\widehat{p}_-(L)}\right\}<p<\frac{p_+(L)}{s_w}$ and $q\in \mathcal{W}_w(q_-(L),q_+(L))$, if we take $f\in \mathbb{H}^{p}_{\nabla L^{-1/2},q}(w)$, by \eqref{square-riesz:1}, we have that 
$$
\|\Scal_{\hh}f\|_{L^p(w)}\lesssim\|\nabla L^{-\frac{1}{2}}f\|_{L^{p}(w)},
$$
consequently $f\in \mathbb{H}_{\Scal_{\hh},q}^p(w)$.
\qed
\subsection{Proof of Proposition \ref{prop:2-Riesz}}
Given $w\in A_{\infty}$ satisfying that $\mathcal{W}_w(q_-(L),q_+(L))\neq \emptyset$, and $q\in \mathcal{W}_w(q_-(L),q_+(L))\subset \mathcal{W}_w(p_-(L),p_+(L))$, if we take $p\in \mathcal{W}_w(q_-(L),q_+(L))$ and $f\in \mathbb{H}^{p}_{\Scal_{\hh},q}(w)$, applying Theorems \ref{RieszboundednessAM} and \ref{theorem:hp=lp}, we obtain
\begin{align}\label{riesz-square:1}
\|\nabla L^{-\frac{1}{2}} f\|_{L^p(w)}\lesssim \|f\|_{L^p(w)}\approx \|\Scal_{\hh}f\|_{L^p(w)}=\|f\|_{\mathbb{H}^{p}_{\Scal_{\hh},q}(w)}.
\end{align}

On the other hand, for $0<p\leq 1$ by Propositions \ref{prop:contro-mol-Riesz}, part $(b)$ and  \ref{lema:SH-1}, part $(a)$ (see also \cite[Proposition 5.1, part ($a$)]{MartellPrisuelos:II}, where the case $p=1$ was considered), we have 
$$
\|\nabla L^{-\frac{1}{2}}f\|_{L^p(w)}\lesssim \|f\|_{\mathbb{H}^{p}_{\Scal_{\hh},q}(w)}.
$$
Therefore, applying Theorem \ref{thm:interpolationhardy} with $p_0=1$ and $p_1\in \mathcal{W}_w(q_-(L),q_+(L))$, we conclude,  for all $0<p<\frac{q_+(L)}{s_w}$, $q\in \mathcal{W}_w(q_-(L),q_+(L))$, and $f\in \mathbb{H}^p_{\Scal_{\hh},q}(w)$,
$$
\|\nabla L^{-\frac{1}{2}} f\|_{L^p(w)}\lesssim  \|\Scal_{\hh}f\|_{L^p(w)},
$$
consequently
$f\in \mathbb{H}^p_{\nabla L^{-1/2},q}(w)$.
\qed



\end{document}